\documentclass[12pt]{amsart}
\usepackage[utf8x]{inputenc}
\usepackage{amsmath}
\usepackage{mathrsfs}
\usepackage{amstext}
\usepackage{amsxtra}
\usepackage{amssymb}
\usepackage{amsgen}
\usepackage{amsfonts}
\usepackage{amsthm}
\usepackage{esint}
\usepackage{amscd}
\usepackage{latexsym}
\usepackage{mathabx}

\usepackage{fullpage}

\usepackage{cite}
\usepackage{setspace}
\usepackage{bbm}
\usepackage{cancel}

\numberwithin{equation}{section}

\newtheorem{theorem}{Theorem}[section]
\newtheorem{lemma}{Lemma}[section]
\newtheorem{proposition}{Proposition}[section]

\newtheorem{assumption}{Assumption}

\theoremstyle{definition}

\newtheorem{claim}{Claim}

\theoremstyle{remark}
\newtheorem{remark}{Remark}[section]

\title[Global Well-Posedness of 2D Non-Focusing Schr\"odinger Equations]{Global Well-Posedness of 2D Non-Focusing Schr\"odinger Equations via Rigorous Modulation Approximation}
\author{Nathan Totz}
\address{Department of Mathematics \\
University of Massachusetts Amherst \\
Amherst, MA, 01003}

\begin{document}

\begin{abstract}
We consider the long time well-posedness of the Cauchy problem with large Sobolev data for a class of nonlinear Schr\"odinger equations (NLS) on $\mathbb{R}^2$ with power nonlinearities of arbitrary odd degree.  Specifically, the method in this paper applies to those NLS equations having either elliptic signature with a defocusing nonlinearity, or else having an indefinite signature.  By rigorously justifying that these equations govern the modulation of wave packet-like solutions to an artificially constructed equation with an advantageous structure, we show that a priori every subcritical inhomogeneous Sobolev norm of the solution increases at most polynomially in time.  Global well-posedness follows by a standard application of the subcritical local theory.
\end{abstract}

\maketitle

%==============================================================================
%==============================================================================
%==============================================================================
\section{Introduction}
%==============================================================================
%==============================================================================
%==============================================================================

Consider the following nonlinear Schr\"odinger initial value problem on $\mathbb{R}^2$:
\begin{equation}\label{NormalizedNLSPrescale}
\begin{cases} iu_t + \alpha u_{xx} - u_{yy} + u|u|^{q - 1} = 0, \\ u(0) = u_0 \in H^s(\mathbb{R}^2) \end{cases}
\end{equation}
where $u = u(x, y, t)$, $(x, y) \in \mathbb{R}^2$, $\alpha = \pm 1$, and $q \in 2\mathbb{N} + 1$.  Here $H^s$ is the usual $L^2$-Sobolev space of index $s$.  The problem with $\alpha = -1, q = 3$ is the familiar defocusing cubic nonlinear Schr\"odinger equation (NLS).  Similarly, the problem with a cubic nonlinearity $q = 3$ and $\alpha = 1$ is the so-called cubic ``hyperbolic'' NLS equation (HNLS).  Both of these equations arise frequently as governing equations for the modulation of wave packets in weakly nonlinear dispersive media, for example in water waves \cite{AblowitzSegur} and nonlinear optics \cite{ElectromagneticXWaves}.  For general $q \in 2\mathbb{N} + 1$ we call this equation (qNLS) and (qHNLS) for $\alpha = -1, 1$ respectively.  This paper addresses the question of the global existence of solutions to \eqref{NormalizedNLSPrescale}.

We begin by contrasting against existing results for the more familiar equation (NLS).  In the last few decades, there has been extensive work in studying (NLS): the aim is to determine the minimal amount of smoothness needed on the initial data to ensure that the corresponding solution to (NLS) is globally well-posed and scatters.  Often, the methods eventually used to treat (NLS) were developed in the context of other problems.  Earlier approaches and influential works include \cite{BourgainCritNLSRadial}, \cite{GrillakisNLS}, \cite{CazenaveWeissler90}, \cite{TaoEnergyCritNLSRadial}.  A breakthrough occurred with the work \cite{ITeamEnergyCritNLSR3}, which was systematized in \cite{KenigMerleCubicNLS3D} into the method of \textit{concentrated-compactness plus rigidity}.  This ``road map'' was in turn used and refined in \cite{KTVCubicNLS2DRadial}, \cite{VisanGWPSNLSinR4}, \cite{Dodson2DDefocusingNLSL2}.  This last work demonstrates the remarkable fact that there is global well-posedness and scattering for large initial data in $L^2$.  Since $L^2$ is the \textit{critical} Sobolev space for (NLS) (that is, the space that is invariant under the natural scaling of (NLS)), one expects that this is the largest of the Sobolev spaces for which global well-posedness and scattering can be expected to hold. For more details of these results in the context of more general NLS type equations, consult \cite{JasonMurphyThesis}.  If one combines the result of \cite{Dodson2DDefocusingNLSL2} and persistence of scattering (c.f. \cite{SohingerHighNormHartree}), the global well-posedness and scattering for large initial data in $H^s$ for all subcritical indices $s > 0$ also follows.

Despite this progress, comparatively little is known for (HNLS).  The work \cite{GhidagliaSautDaveyStewartson} used dispersive arguments to show global well-posedness of (HNLS) under a small data assumption.  In \cite{GhidagliaSautNoTravelingWaves} it was shown that (HNLS) does not admit localized traveling wave solutions.  The work \cite{KNZRadialHNLS} recasts (HNLS) in well-adapted hyperbolic coordinates and, in these coordinates, constructs special radial and self-similar solutions to (HNLS) which correspond to having large initial data.  While the special solutions are in $L^2$ with respect to the hyperbolic coordinates, they lie outside of $L^2$ in standard rectangular coordinates.  To date, nothing is known about the long-time well-posedness for the initial value problem (HNLS) with general large data in the $L^2$-Sobolev class, even in very regular Sobolev spaces.  A fundamental obstacle is that the natural energy associated to \eqref{NormalizedNLSPrescale},
$$\int_{\mathbb{R}^2} -\alpha\frac{|u_x|^2}{2} + \frac{|u_y|^2}{2} + \frac{|u|^{q + 1}}{q + 1} \, dx \, dy,$$
is indefinite when $\alpha = 1$.  This defeats attempts to show global well-posedness for large data with regularity above the level of the Hamiltonian using classical methods.  Although the total mass
$$\int_{\mathbb{R}^2} |u|^2 \, dx \, dy$$
is conserved and gives control of the $L^2$ norm of the solution, this quantity scales either critically or supercritically with respect to the natural invariant scaling of \eqref{NormalizedNLSPrescale} and hence cannot be directly exploited to study long-time existence, just as is the case for (NLS).  Therefore, as with the study of low-regularity global well-posedness and scattering for (NLS), one's only hope is to appeal to the concentration-compactness plus rigidity method.  However, to date it is not clear that the analytic tools needed to carry out the concentration-compactness plus rigidity program for (HNLS) exist.

In this paper, we use methods quite different from those applied in the past to study \eqref{NormalizedNLSPrescale} in order to provide a unified approach to showing global well-posedness of the problem \eqref{NormalizedNLSPrescale} for powers $q \in 2\mathbb{N} + 1$ of arbitrarily large degree, provided that data is considered in any subcritical Sobolev space.  We do this not by studying \eqref{NormalizedNLSPrescale} in isolation, but by taking advantage of the fact that solutions to \eqref{NormalizedNLSPrescale} arise in certain asymptotic limits of higher-order evolution equations.

The equation \eqref{NormalizedNLSPrescale} arises as a singular perturbation of many physical systems.  In particular, (HNLS) models the modulation of a wave packet solution to the deep water limit of the 3D water wave problem; this was first derived formally in \cite{ZakharovInfiniteDepth}.  The method used in this paper is inspired by \cite{Totz3DNLS}, where this approximation is rigorously justified in the context of deep water waves.  More generally, given a solution $z$ to some (as yet unspecified) evolution equation with formal wave packet solution $\tilde{z}$ of the form\footnote{This is not the actual formal series solution $\tilde{z}$ that we use in the paper, since we will require a number of additional rescalings of $u$ for technical reasons.  The properly rescaled version of the modulation that we will actually use are given \eqref{Rescaling}-\eqref{NormalizedNLS}.}
\begin{equation}\label{WavePacketSolution}
\epsilon u(\epsilon (x + \omega^\prime(k) t), \epsilon y, \epsilon^2 t)e^{i(kx + \omega(k)t)} + O(\epsilon^2)
\end{equation}
for some small typical amplitude $0 < \epsilon \ll 1$ and wave number $k$, a \textbf{full justification} of NLS consists roughly of performing the following steps:
\begin{itemize}
\item[(i)]{The NLS equation \eqref{NormalizedNLSPrescale} is locally well-posed in a suitable function space.}
\item[(ii)]{An approximate solution $\tilde{z}$ of the form \eqref{WavePacketSolution} constructed using the solution $u$ given in Step (i) can be found which formally satisfies the equation for $z$ up to residual terms of physical size at most $o(\epsilon^3)$.}
\item[(iii)]{For all sufficiently small values of $\epsilon > 0$, the evolution equation governing $z$ is well posed on a function space containing an open neighborhood of the initial data $\tilde{z}(0)$ of the approximate solution, and solutions $z$ with initial data in such a space exist for times on the order $O(\epsilon^{-2})$.}
\item[(iv)]{For all sufficiently small values of $\epsilon > 0$, the remainder $z- \tilde{z}$ is of size at most $o(\epsilon)$ in a suitable function space.}
\end{itemize}

There are many justifications of NLS for various evolution equations in the literature.  Such a full justification need not hold even in cases where a formal wave packet solution can be constructed (e.g., \cite{SchneiderValidityNW}, \cite{GallaySchneiderKP}).  The earliest such justification in the modulation regime was given in \cite{KalyakinNLSJust} for evolution equations with quite general semilinear quadratic nonlinearities in the presence of a non-resonance condition.  In the case of cubic and higher order nonlinearities, \cite{KSMCubicNonlinearityLongtimeRemainder} shows that such a justification follows from an easy application of Gr\"onwall's inequality.  Further progress has been made in justifying NLS in special cases of equations with quadratic nonlinearities in which the special structure of the equation is used; in particular much work has been done in studying this problem in the context of the water wave problem, c.f. \cite{Schneider05}, \cite{SchWPD06}, \cite{SchneiderWayneNLSJustify}, \cite{TotzWu2DNLS}, \cite{Totz3DNLS}, \cite{DullSchneiderWayne2DNLSFinite}.

Demonstrating that Step (iii) holds in the above program can be challenging if one only considers the behavior of $z$ according to its evolution equation in isolation, since long-time well-posedness results for solutions $z$ that are initially close to wave packets may not establish existence for sufficiently long times (this is the case in the water wave problem).  However, if one has already shown Steps (i), (ii), (iv), one can sidestep this problem using a priori bounds depending on the approximate solution.  To describe the idea loosely: suppose $\tilde{z}$ exists on a long time scale $[0, \mathbf{T}]$.  Step (iii) requires one to establish that $z$ exists for the times $[0, \mathbf{T}]$ as well.  Suppose it is known that $\tilde{z} = O(1)$ on $[0, \mathbf{T}]$.  If one has a priori bounds on the estimate of the form $\|z - \tilde{z}\| = o(\epsilon)$ on $[0, \mathbf{T}]$ as well, then there are a priori bounds on $z = \tilde{z} + (z - \tilde{z})$ of the form $\|z\| = O(1)$ on $[0, \mathbf{T}]$.  The existence of $z$ on all of $[0, \mathbf{T}]$ then follows if the well-posedness theory of $z$ admits a suitable blow-up alternative.   

The approach of this paper is to take advantage of this interplay between $z$ and $\tilde{z}$ in the opposite direction: if we were to know that solutions $z(t)$ in the modulation regime of the original evolution equation had sufficiently long-time existence \textit{independent of the dynamics of $\tilde{z}(t)$}, the same argument as above would establish a long-time bound on $\tilde{z}$, which might in turn be used to control the corresponding solution of (HNLS).

We emphasize that one \textit{cannot} implement this approach using existing full justification results in the case of deep water waves and (HNLS): the obstacle is showing the water wave problem with wave-packet data exists for sufficiently long times independent of the NLS evolution of the modulation.  While proofs of global existence for 3D deep water waves exist in the literature (c.f. \cite{WuGlobal3D}, \cite{GermainMasmoudiShatah}), all such proofs require localization assumptions on the initial data.  This precludes using known long-time existence results directly, since wave packets do not lie in these admissible classes of initial data.  As mentioned above, standard local existence results yield existence times that are too short to provide full justification.  Fortunately, model equations of nonlinear Schr\"odinger form arise generically as modulation approximations to dispersive equations, so we have some freedom to choose a more suitable dispersive equation.

Hence, instead of using the water wave problem, we introduce an artificial \textit{progenitor equation} patterned off of the water wave problem which (1) yields \eqref{NormalizedNLSPrescale} as the corresponding \textit{progeny equation} in the modulation regime, and (2) possesses an advantageous structure from which long-time well-posedness can be deduced independently of the dynamics of solutions to the progeny equation.  From this approach we will be able to conclude that the $H^s$ norm of the solution to \eqref{NormalizedNLSPrescale} a priori increases at an a polynomial rate in time for any $s > s_c$, where $s_c = 1 - \frac{2}{q - 1}$ is the critical scaling index of \eqref{NormalizedNLSPrescale}.  Global well-posedness in any subcritical space $H^s$ with $s > s_c$ then follows by standard methods for semilinear Schr\"odinger equations.  The precise statement of our result is given by the

\begin{theorem}\label{BigSummaryTheorem}
Let $s_c = 1 - \frac{2}{q - 1}$ be the critical scaling index of \eqref{NormalizedNLSPrescale}, and let $s > s_c$ be given.  Consider the initial value problem \eqref{NormalizedNLSPrescale} for any $\alpha = \pm 1$ and $q \in 2\mathbb{N} + 1$.  Suppose that $u_0 \in H^s(\mathbb{R}^2)$.  Then \eqref{NormalizedNLSPrescale} is globally well-posed with solution $u \in C([0, \infty), H^s)$.  Moreover solutions enjoy the estimate
\begin{equation}\label{APrioriBoundOnU}
\|u(t)\|_{H^s} \leq C_{s, q, \|u_0\|_{H^s}} \|u_0\|_{H^s} (1 + t)^{C_0} 
\end{equation}
where $C_0$ is a universal constant.\footnote{In this paper the constant is on the order of $10^5$, but this can be substantially improved to be on the order of $1$ by optimizing the choice of constants in the proof.  See Remark \ref{BringDownHugeConstant} for more details.}
\end{theorem}

The paper is organized as follows.  Section 1.2 is a standard technical reduction of Theorem \ref{BigSummaryTheorem} to the special case in which (1) the initial data is assumed to be very smooth, (2) the $H^s$-norm of the solution increases sufficiently rapidly, and (3) the time variable is rescaled to be sufficiently small.  The core of the argument is outlined in Section 1.3, and the progenitor equation that we use is stated there as well.  Section 2 collects relevant facts about the approximate solution to the progenitor equation, including the construction of the approximate solution itself.  The a priori bounds of the error between the true and approximate solutions to the progenitor equation are provided in Section 3 with the existence of a solution to the progenitor equation assumed.  The actual proof of local well-posedness of the progenitor equation is then given in Section 4.\footnote{Because of its technical nature, we recommend that this section be omitted on a first reading.}  Finally, Section 5 completes the argument by establishing the long-time well-posedness of the progenitor equation.

\vspace{0.5cm}

\textbf{Acknowledgements.}  The author would like to thank J. C. Saut for bringing this problem to his attention.  The author is also grateful to B. Dodson, A. Nahmod, B. Pasauder and J. Rauch for discussions about the problem and the method of solution presented here, as well as to J. T. Beale, J. Beichman, J. Marzuola, and S. Wu who in addition provided helpful comments on earlier drafts of this paper.  Finally, the manuscript benefitted from the helpful comments of the referee, who also has the author's gratitude.  This manuscript was in part supported by NSF Grant DMS-1440140 while the author was in residence at the Mathematical Sciences Research Institute during the Fall 2015 semester.

%==============================================================================
%==============================================================================
\subsection{Notation and Basic Estimates}
%==============================================================================
%==============================================================================

Given a set $S$, we denote the characteristic function of $S$ by $\mathbf{1}_S$.  For two complex-valued functions $F, G \in L^2(\mathbb{R}^2)$, define the usual complex $L^2$ inner product $\langle F, G \rangle = \int_{\mathbb{R}^2} \Re(F\overline{G}) \, dx \, dy$.  For $\xi = (\xi_1, \xi_2)$ we denote the Fourier transform by $$\hat{f}(\xi) = \frac{1}{(2\pi)^2}\int_{\mathbb{R}^2} e^{-i(x\xi_1 + y\xi_2)} f(x, y) \, dx \, dy$$   Let $|D|^s$ denote the solid derivative with symbol $|\xi|^s$.  Let $\Lambda^s$ be the Fourier multiplier with symbol $(|\xi|^2 + 1)^{\frac{s}{2}}$.  In this paper the unadorned norm $\|f\|$ will denote the $L^2(\mathbb{R}^2)$ norm.  Define the space $\hat{L}^1(\mathbb{R}^2)$ to be the set of all functions $f \in L^1_{loc}(\mathbb{R}^2)$ for which $\|f\|_{\hat{L}^1} := \|\hat{f}\|_{L^1} < \infty$; we recall that $\hat{L}^1(\mathbb{R}^2)$ is a Banach algebra.  Define the usual Sobolev space $H^s = H^s(\mathbb{R}^2)$ as the completion of $C_0^\infty$ under the norm $\|f\|_{H^s}^2 = \|\Lambda^s f\|^2$.  For some parameter $k > 0$, let $\mathcal{B}_k$ be the Fourier multiplier whose symbol is the characteristic function of the set $\{(\xi_1, \xi_2) \in \mathbb{R}^2 : |(\xi_1 - k, \xi_2)| \leq \frac{k}{2} \}$.   Then we have the classical embeddings $H^2(\mathbb{R}^2) \hookrightarrow \hat{L}^1(\mathbb{R}^2) \hookrightarrow L^\infty(\mathbb{R}^2)$.  Given a Fourier multiplier $M$ with symbol $\hat{M}(\xi)$ and a parameter $k > 0$, we denote by $M_\epsilon$ the multiplier with symbol 
\begin{equation}\label{RescaledMultiplier}
\hat{M_\epsilon}(\xi) = \hat{M}\left(\frac{\xi - (k, 0)}{\epsilon}\right) = \hat{M}\left(\frac{\xi - ke_1}{\epsilon}\right)
\end{equation}
This operator is designed so that given any function $F(x, y)$, we have $$M_\epsilon\left(F(\epsilon x, \epsilon y)e^{ikx}\right) = (MF)(\epsilon x, \epsilon y)e^{ikx}$$

We will analyze some of our power nonlinearities using the standard Moser type estimate:

\begin{theorem}\label{Moser}
Let $s \geq 0$ be given, and suppose that $g_1, \ldots, g_q \in L^\infty(\mathbb{R}^2) \cap \dot{H}^s(\mathbb{R}^2)$.  Then $g_1 g_2 \cdots g_q \in \dot{H}^s$ and
\begin{equation}
\|\,|D|^s(g_1 g_2 \cdots g_q)\| \leq C_{(s, q)} \sum_{j = 1}^q \|\,|D|^s g_j\| \prod_{l \neq j} \|g_l\|_{L^\infty}
\end{equation}
\end{theorem}

\begin{proof}  See \cite{TaoNonlinearDispersiveEquations}. \end{proof}

We pause to show that this estimate continues to hold for rescaled derivatives operating on the type of power nonlinearities we consider here.

\begin{lemma}\label{RescaledMoserInequality}
Let $q \in 2\mathbb{N} + 1$ and $s \geq 0$ be given.  Given functions $f_1, \ldots, f_q \in L^\infty(\mathbb{R}^2) \cap \dot{H}^s(\mathbb{R}^2)$ and a parameter $k > 0$, we have the estimate
\begin{equation}
\|\,|D|^s_\epsilon(f_1 \cdots f_{\frac{q + 1}{2}} \overline{f}_{\frac{q + 3}{2}} \cdots \overline{f}_{q})\| \leq C_{(s, q)} \sum_{j = 1}^q \|\,|D|^s_\epsilon f_j\| \prod_{l \neq j} \|f_l\|_{L^\infty}
\end{equation}
\end{lemma}

\begin{proof}
Notice that we can write $|D|^s_\epsilon = \epsilon^{-s} e^{ikx} |D|^s e^{-ikx}$.  Introducing $\psi_j = f_j e^{-ikx}$, we have
\begin{align*}
\|\,|D|^s_\epsilon(f_1 \cdots f_{\frac{q + 1}{2}} \overline{f}_{\frac{q + 3}{2}} \cdots \overline{f}_{q})\| & \leq \|\epsilon^{-s}|D|^s(\psi_1 \cdots \psi_{\frac{q + 1}{2}} \overline{\psi}_{\frac{q + 3}{2}}\cdots \overline{\psi}_{q - 1})\| \\
& \leq C_{(s, q)} \sum_{j = 1}^q \|\,\epsilon^{-s}|D|^s \psi_j\| \prod_{l \neq j} \|\psi_l\|_{L^\infty} \\
& = C_{(s, q)} \sum_{j = 1}^q \|\,|D|^s_\epsilon f_j\| \prod_{l \neq j} \|f_l\|_{L^\infty}
\end{align*}
where the Theorem \ref{Moser} was used in the second inequality.
\end{proof}

We will also need another product estimate which, although weaker than that of Lemma \ref{RescaledMoserInequality}, has the advantage that one can choose a particular factor to lie in an $L^2$ space.

\begin{lemma}\label{L1HatEstimate}
Let $q \in 2\mathbb{N} + 1$, $s \geq 0$, $k > 0$, and $j \in \{1, \ldots, q\}$ be given.  Let $\Lambda^s_\epsilon f_j \in L^2(\mathbb{R}^2)$ and for $l \neq j$, let $\Lambda^s_\epsilon f_l \in \hat{L}^1(\mathbb{R}^2)$.  Then $\Lambda^s_\epsilon\left(f_1 \cdots f_{\frac{q + 1}{2}} \overline{f}_{\frac{q + 3}{2}} \cdots \overline{f}_q\right) \in L^2(\mathbb{R}^2)$ and
\begin{equation}
\left\|\Lambda^s_\epsilon\left(f_1 \cdots f_{\frac{q + 1}{2}} \overline{f}_{\frac{q + 3}{2}} \cdots \overline{f}_q\right)\right\|_{L^2} \leq C_{s, q} \ \|\Lambda^s_\epsilon f_j\|_{L^2} \prod_{l \neq j} \|\Lambda^s_\epsilon f_l\|_{\hat{L}^1}
\end{equation}
\end{lemma} 

\begin{proof}
Introduce $\psi_l = f_l e^{ikx}$ as well as the repeated convolution operator $\mathop{\Asterisk_{l = 1}^m} f_l := f_1 * f_2 * \cdots * f_m$; then we have 
\begin{align*}
& \quad\; \|\Lambda^s_\epsilon(f_1 \cdots f_{\frac{q + 1}{2}} \overline{f}_{\frac{q + 3}{2}} \cdots \overline{f}_q)\| \\
& = \left\|\left\langle\frac{\xi - ke_1}{\epsilon}\right\rangle^s \left( \mathop{\Asterisk_{l = 1}^{\frac{q + 1}{2}}} \hat{f}_l \right) * \left(\mathop{\Asterisk_{l = {\frac{q + 3}{2}}}^q} \hat{\overline{f}_l} \right)\right\| \\
& = \left\|\left\langle\frac{\xi - ke_1}{\epsilon}\right\rangle^s \left( \mathop{\Asterisk_{l = 1}^{\frac{q + 1}{2}}} \hat{\psi}_l(\xi - ke_1) \right) * \left(\mathop{\Asterisk_{l = {\frac{q + 3}{2}}}^q} \overline{\hat{\psi}_l}(-\xi - ke_1) \right)\right\|
\end{align*}
Then we have using Peetre's Inequality $\langle x \rangle^s \leq C_s \langle y \rangle^s \langle x - y \rangle^s$ repeatedly that this last term is bounded by
\begin{align*}
& \leq C_{s, q} \left\| \left( \mathop{\Asterisk_{l = 1}^{\frac{q + 1}{2}}} \left\langle\frac{\xi - ke_1}{\epsilon}\right\rangle^s \hat{\psi}_l(\xi - ke_1) \right) * \left(\mathop{\Asterisk_{l = {\frac{q + 3}{2}}}^q} \left\langle\frac{\xi + ke_1}{\epsilon}\right\rangle^s \overline{\hat{\psi}_l}(-\xi - ke_1) \right)\right\| \\
& = C_{s, q} \left\| \left( \mathop{\Asterisk_{l = 1}^{\frac{q + 1}{2}}} \left\langle\frac{\xi - ke_1}{\epsilon}\right\rangle^s \hat{\psi}_l(\xi - ke_1) \right) * \left(\mathop{\Asterisk_{l = {\frac{q + 3}{2}}}^q} \overline{\left\langle\frac{-\xi - ke_1}{\epsilon}\right\rangle^s \hat{\psi}_l}(-\xi - ke_1) \right)\right\| \\
& = C_{s, q} \left\| \left( \mathop{\Asterisk_{l = 1}^{\frac{q + 1}{2}}} \left\langle\frac{\xi - ke_1}{\epsilon}\right\rangle^s \hat{\psi}_l(\xi - ke_1) \right) * \left(\mathop{\Asterisk_{l = {\frac{q + 3}{2}}}^q} \overline{\left\langle\frac{\xi - ke_1}{\epsilon}\right\rangle^s \hat{\psi}_l}(\xi - ke_1) \right)\right\|
\end{align*}
Now estimating using Young's Inequality in which the $j$th term is taken in $L^2$ and the others in $L^1$, we have that this is controlled by
\begin{align*}
C_{s, q} \left\| \left\langle \frac{\xi - ke_1}{\epsilon} \right\rangle^s \hat{f}_j \right\|_{L^2} \; \prod_{l \neq j} \left\| \left\langle \frac{\xi - ke_1}{\epsilon} \right\rangle^s \hat{f}_l \right\|_{L^1},
\end{align*}
which is just the frequency side version of the desired estimate.
\end{proof}

%==============================================================================
%==============================================================================
\subsection{Reductions to Special Cases}
%==============================================================================
%==============================================================================

In this section we make a number of reductions to special cases.  Let $s > s_c$ be fixed in what follows.

%==============================================================================
%==============================================================================
\subsubsection{Local Well-Posedness Theory of HNLS}
%==============================================================================
%==============================================================================

The local well-posedness theory of \eqref{NormalizedNLSPrescale} in subcritical Sobolev spaces is standard; we record the following version of it here:

\begin{proposition}\label{HNLSLocalWellposedness}
\cite{CazenaveSemilinearSchrodingerEquations}  There exists a unique maximal solution $u \in C([0, T_{max}) : H^s)$ to \eqref{NormalizedNLSPrescale} that enjoys the following properties:
\begin{itemize}
\item{(Blow-up Alternative)  If $T_{max} < \infty$, then $\|u(t)\|_{H^s} \to \infty$ as $t \to T_{max}$.  In addition there is a constant $C_{s, q}$ for which $T_{max} \geq C_{s, q}\|u_0\|_{H^s}^{1 - q}$.}
\item{(Continuous Dependence on Initial Data)  Suppose that $u^{(n)}_0 \to u_0$ in $H^s$.  Let $u$, $u^{(n)}$ be the solutions to \eqref{NormalizedNLSPrescale} with initial data $u_0$, $u^{(n)}_0$, respectively.  Denote by $T^{(n)}_{max}$ the maximal blow-up time of $u^{(n)}$.  Then for any closed interval $I \subset [0, T_{max})$ we have $I \subset [0, T^{(n)}_{max})$ for sufficiently large $n$, and $u^{(n)} \to u$ in $C(I : H^s)$.}  
\item{(Persistence of Regularity) If in addition we have $u_0 \in H^{s^\prime}$ for some $s^\prime > s$, then $u \in C([0, T_{max}), H^{s^\prime})$.}
\item{(Mass Conservation) $\|u(t)\| = \|u_0\|$ whenever $t \in [0, T_{max})$.}
\end{itemize}
\end{proposition}

\begin{proof}
As a reference, see Theorems 4.12.1 and 5.7.1 of \cite{CazenaveSemilinearSchrodingerEquations}.  The local well-posedness in the case $\alpha = -1$ goes back to \cite{CazenaveWeissler90}, and the local well-posedness in the case $\alpha = 1$ follows from the work \cite{GhidagliaSautDaveyStewartson} which adapted the ideas of \cite{CazenaveWeissler90} in the mixed signature setting.  Note that the relatively strong persistence of regularity result here depends on the fact that we study NLS equations with power nonlinearities of odd degree.
\end{proof}

%==============================================================================
%==============================================================================
\subsubsection{Rescaling the Time Variable}
%==============================================================================
%==============================================================================

We write $\mathfrak{u}(\mathfrak{x}, \mathfrak{y}, \mathfrak{t})$ for a solution to \eqref{NormalizedNLSPrescale} in order to reserve the symbols $(x, y, t)$ for the space and time variables of the progenitor equation.  We will rescale $\mathfrak{u}$ using two parameters $p > 0$ and $\omega > 0$ to be chosen later as follows:
\begin{equation}\label{Rescaling}
\mathfrak{u}(\mathfrak{x}, \mathfrak{y}, \mathfrak{t}) = A\left(\frac12 \sqrt{p|2 - p|} \mathfrak{x}, \sqrt{p} \mathfrak{y}, 2 \omega^{\frac{4}{p} - 1} \mathfrak{t}\right) =: A(X, Y, T)
\end{equation}
Notice that with this scaling $\|A(T)\|_{L^\infty} = \|\mathfrak{u}(\mathfrak{t})\|_{L^\infty}$.  Choose $0 < p < 2$ when $\alpha = 1$ and $p > 2$ when $\alpha = -1$; in these cases the rescaled solution $A(X, Y, T)$ now satisfies
\begin{equation}\label{NormalizedNLS}
\begin{cases}
\displaystyle 2i\omega^{\frac{4}{p} - 1} A_T + \frac14 p(2 - p) A_{XX} - pA_{YY} + A|A|^{q - 1} = 0 \\
A(0) = A_0
\end{cases}
\end{equation}

%==============================================================================
%==============================================================================
\subsubsection{Reduction to the Case of Smooth Initial Data and Sufficiently Rapid $H^s$ Norm Growth.}
%==============================================================================
%==============================================================================

We will show that Theorem \ref{BigSummaryTheorem} holds by proving that $\|u(\mathfrak{t})\|_{H^s}$ satisfies an appropriate differential inequality.  Formally one can use \eqref{NormalizedNLSPrescale} to derive the following expression for this derivative:
\begin{align}\label{DerivativeOfHsNormOfU}
\frac{d}{d\mathfrak{t}}(\mathfrak{t}\|\mathfrak{u}(\mathfrak{t})\|_{H^s}^2) = \|\mathfrak{u}(\mathfrak{t})\|_{H^s}^2 + 2\mathfrak{t}\langle \Lambda^s \mathfrak{u}(\mathfrak{t}), \Lambda^s (i\mathfrak{u}(\mathfrak{t})|\mathfrak{u}(\mathfrak{t})|^{q - 1}) \rangle
\end{align}
In order to prove Theorem \ref{BigSummaryTheorem}, it therefore suffices to show the bound
\begin{equation}\label{DesiredDiffBoundonU}
\|\mathfrak{u}(\mathfrak{t})\|_{H^s} + 2\mathfrak{t}\langle \Lambda^s \mathfrak{u}(\mathfrak{t}), \Lambda^s( i\mathfrak{u}(\mathfrak{t})|\mathfrak{u}(\mathfrak{t})|^{q - 1} ) \rangle \leq \mathscr{C}_p\|\mathfrak{u}(\mathfrak{t})\|_{H^s}^2 \qquad \mathfrak{t} \in [0, \mathbf{T}]
\end{equation}
Applying the rescaling \eqref{Rescaling} to the bound \eqref{DesiredDiffBoundonU} gives the equivalent condition
\begin{align}\label{DerivativeOfHsNormOfA}
\frac{d}{dT}(T\|A(T)\|_{H^s}^2) & = \|A(T)\|_{H^s}^2 + 2T\langle \Lambda^s A(T), \Lambda^s (iA(T)|A(T)|^{q - 1}) \rangle \notag \\
& \leq \mathscr{C}\|A(T)\|_{H^s}^2
\end{align}
on the rescaled time interval $[0, 2 \omega^{\frac{4}{p} - 1} \mathbf{T}]$, where now the constant $\mathscr{C}$ is independent of $p$.
\begin{remark}\label{WhyRescaleWhyThisBound}
The reason that the particular estimate \eqref{DesiredDiffBoundonU} was chosen was to introduce a small parameter into later estimates using a rescaling that leave our bounds invariant.  In later estimates needed to show the local well-posedness of our choice of progenitor equation (c.f. the proof of Lemma \ref{NanoLocalWellposedness}) it will therefore be advantageous to choose $\omega$ so as to take advantage of the factor of $T \leq 2\omega^{\frac{4}{p} - 1} \mathbf{T}$ in the expression $2T\langle |D|^s A(T), |D|^s (iA(T)|A(T)|^{q - 1})\rangle$.  Since we will take either $p = 1, 3$ in the sequel, the factor $\omega^{\frac{4}{p} - 1}$ may always be chosen small provided $\omega \ll 1$ is chosen sufficiently small.
\end{remark}

These manipulations of the time derivative of $\|A\|_{H^s}$ only makes sense provided the solution is classical.  In fact, later in the argument, we will require that even higher derivatives of $A$ be defined as well.  Therefore we state the following standard mollification result:

\begin{lemma}\label{MollificationArgument}
Let $s > s_c$ and $N = N(s) > s + 21$ be given.  Suppose that whenever $\mathfrak{u}_0 \in H^N$, for any time $\mathbf{T} > 0$ for which the corresponding solution to \eqref{NormalizedNLSPrescale} $\mathfrak{u} \in C([0, \mathbf{T}] : H^N)$, the solution $\mathfrak{u}$ satisfies the bound 
\begin{equation}\label{SmoothBound}
\|\mathfrak{u}(\mathfrak{t})\|_{H^s} \leq \|\mathfrak{u}_0\|_{H^s} M(\|\mathfrak{u}_0\|_{H^s}, s, q, \mathfrak{t}) \qquad 0 \leq T \leq \mathbf{T}
\end{equation}
where $M$ is increasing in its first argument.  Then \eqref{SmoothBound} in fact holds when $\mathfrak{u}_0 \in H^s$ and for any time $\mathbf{T} > 0$ for which the corresponding solution to \eqref{NormalizedNLS} $\mathfrak{u} \in C([0, \mathbf{T}] : H^s)$.
\end{lemma}

\begin{proof}
Let $\mathfrak{u}_0^\iota$ be the standard mollifier of $\mathfrak{u}_0$, so that $\|\mathfrak{u}_0^\iota\|_{H^N} \leq C(\iota)$, $\|\mathfrak{u}_0^\iota\|_{H^s} \leq \|\mathfrak{u}_0\|_{H^s}$, and $\|\mathfrak{u}_0 - \mathfrak{u}_0^\iota\|_{H^s} \to 0$ as $\iota \to 0$.  Let $\mathfrak{u}^\iota$ be the solution to \eqref{NormalizedNLS} with $\mathfrak{u}^\iota(0) = \mathfrak{u}_0^\iota$.  By the persistence of regularity in Proposition \ref{HNLSLocalWellposedness} we have $\mathfrak{u}^\iota \in C([0, \mathbf{T}], H^{N(s)})$.  By hypothesis we have the bound $\|\mathfrak{u}^\iota(T)\|_{H^s} \leq \|\mathfrak{u}_0^\iota\|_{H^s} M(\|\mathfrak{u}_0^\iota\|_{H^s}, s, q, \mathfrak{t})$ for all $0 \leq T \leq \mathbf{T}$.  But then by the continuous dependence on initial data of Proposition \ref{HNLSLocalWellposedness} we have for all $0 \leq T \leq \mathbf{T}$ that
\begin{align*}
\|\mathfrak{u}(\mathfrak{t})\|_{H^s} & \leq \liminf_{\iota \to 0} \left(\|\mathfrak{u}^\iota(\mathfrak{t})\|_{H^s} + \|\mathfrak{u}(\mathfrak{t}) - \mathfrak{u}^\iota(\mathfrak{t})\|_{H^s}\right) \\
& \leq \liminf_{\iota \to 0} \|\mathfrak{u}_0^\iota\|_{H^s} M(\|\mathfrak{u}_0^\iota\|_{H^s}, s, q, \mathfrak{t}) \\
& \leq \|\mathfrak{u}_0\|_{H^s} M(\|\mathfrak{u}_0\|_{H^s}, s, q, \mathfrak{t}).
\end{align*}
\end{proof}

\begin{remark}
Lemma \ref{MollificationArgument} is the only place in the entire argument that requires the fact that $s$ is a subcritical Sobolev index of \eqref{NormalizedNLSPrescale}.  In particular, this result fails at the critical regularity $s = s_c$ since uniqueness holds in general only in a strict subspace of $C([0, \mathbf{T}] : H^{s_c})$.
\end{remark}

The contradiction argument forming the core of our argument is given in the

\begin{proof}[Beginning of the proof of Theorem \ref{BigSummaryTheorem}.]
First, let us demonstrate how a bound of the form \eqref{DesiredDiffBoundonU} on some $[0, \mathbf{T}]$ implies Theorem \ref{BigSummaryTheorem}.  First, note that from the local theory of Proposition \ref{HNLSLocalWellposedness} we have a solution $\mathfrak{u}(\mathfrak{t})$ defined on some interval of time $T_{loc}$ which is the first time at which $\|\mathfrak{u}(\mathfrak{t})\|_{H^s} = 2\|\mathfrak{u}_0\|_{H^s}$ with $T_{loc} \geq C_{s, q}\|\mathfrak{u}_0\|_{H^s}^{1 - q}$.  Now suppose that $\mathfrak{u}(\mathfrak{t})$ has a finite maximal time of existence $T_{max} > T_{loc}$.  Integrating the ODE \eqref{DesiredDiffBoundonU} on each subinterval $[T_{loc}, \mathbf{T}]$ with $\mathbf{T} < T_{max}$ yields the inequality
\begin{equation*}
\limsup_{\mathfrak{t} \to T_{max}^-} \|\mathfrak{u}(\mathfrak{t})\|_{H^s}^2 \leq 4\|\mathfrak{u}_0\|_{H^s}^2 \left(\frac{\mathfrak{t}}{T_{loc}}\right)^{C_0} \leq C(s, q, \|\mathfrak{u}_0\|_{H^s}) \|u_0\|_{H^s}^2 \mathfrak{t}^{C_0}
\end{equation*}
But the blow-up criterion of Proposition \ref{HNLSLocalWellposedness} now implies that $T_{max} = +\infty$.  Since the a priori bound above can be applied to any $[T_{loc}, \mathbf{T}]$ for $\mathbf{T} > T_{loc}$, the bound \eqref{APrioriBoundOnU} indeed holds for all time.

It remains to prove the bound \eqref{DesiredDiffBoundonU}.  As in the last section, by rescaling it suffices to show \eqref{DerivativeOfHsNormOfA}.  Suppose that \eqref{DerivativeOfHsNormOfA} did not hold; then there is some time $T_* \in (0, \mathbf{T})$ for which $$\|A(T_*)\|_{H^s}^2 + 2T\langle \Lambda^s A(T_*), \Lambda^s( iA(T_*)|A(T_*)|^{q - 1} ) \rangle \geq 2\mathscr{C} \|A(T_*)\|_{H^s}^2;$$ then since $A \in C([0, \mathbb{T}] : H^N)$ we have through \eqref{NormalizedNLS} that $$T \mapsto \|A(T)\|_{H^s}^2 + 2T\langle \Lambda^s A(T), \Lambda^s( iA(T)|A(T)|^{q - 1} ) \rangle$$ is a continuous map, and so there exists an interval $[T_1, T_2] \subset (0, \mathbf{T})$ of positive length on which $$\|A(T)\|_{H^s}^2 + 2T\langle \Lambda^s A(T), \Lambda^s( iA(T)|A(T)|^{q - 1} ) \rangle \geq \mathscr{C} \|A(T)\|_{H^s}^2, \qquad T \in [T_1, T_2].$$  
Now suppose we grant the following
\begin{claim}
Let $\nu \in (1, 2)$ be given.  Suppose that $A(T) \in C([T_1, T_2] : H^N)$ is a solution to \eqref{NormalizedNLS} that for all $T \in [T_1, T_2]$ satisfies $$\|A(T)\|_{H^s}^2 + 2T\langle \Lambda^s A(T), \Lambda^s( iA(T)|A(T)|^{q - 1} ) \rangle \geq \mathscr{C} \|A(T)\|_{H^s}^2.$$  Then $\|A(T)\|_{H^s} \leq (2\nu - 1)\|A(T_1)\|_{H^s}$ for all $T_1 \leq T \leq T_2$.
\end{claim}
Since here we have assumed that the hypotheses of the claim hold, we conclude that for any $\nu \in (1, 2)$ we please, $\|A(T)\|_{H^s} \leq (2\nu - 1)\|A(T_1)\|_{H^s}$.  However, if we integrate \eqref{DerivativeOfHsNormOfA}, we find that 
$\|A(T_2)\|_{H^s}^2 \geq \|A(T_1)\|_{H^s}^2 \left( \frac{T_2}{T_1} \right)^{\mathscr{C}}$.
However, this leads to a contradiction upon choosing $\nu$ sufficiently close to $1$ so that $(2\nu - 1)^2 <  \left( \frac{T_2}{T_1} \right)^{\mathscr{C}}$.
\end{proof}

The rest of the argument proceeds on this interval $[T_1, T_2]$; we relabel $[T_1, T_2]$ as $[0, \mathscr{T}]$.  We make the following assumptions on $A$ that hold for the remainder of the paper:

\begin{assumption}\label{AssumptionsOnA}
In what follows, we may assume that we have a solution $A(T)$ to the initial value problem \eqref{NormalizedNLS} with values in $H^{N(s)}$ defined on a time interval $[0, \mathscr{T}]$ on which we have the estimate
\begin{equation}\label{HsGrowthAssn}
\|A(T)\|_{H^s}^2 + 2(T + T_1)\langle \Lambda^{s}A(T), \Lambda^s(iA(T)|A(T)|^{q - 1}) \rangle \geq \mathscr{C}\|A(T)\|_{H^s}^2
\end{equation}
for some $T_1 > 0$.
\end{assumption}

%==============================================================================
%==============================================================================
\subsection{Outline of the Proof of Claim}
%==============================================================================
%==============================================================================

The rest of this paper is devoted to supplying the proof of the outstanding claim in the proof of Theorem \ref{BigSummaryTheorem}.  The approach of this paper is to regard a solution $A$ of \eqref{NormalizedNLS} as the modulation of a nonlinear geometric optics approximation 
\begin{equation}\label{WavePacketSolution2}
\tilde{z}(x, y, t) = \epsilon A(\epsilon (x + \omega^\prime t), \epsilon y, \epsilon^2 t)e^{i(kx + \omega t)} + O(\epsilon^2)
\end{equation}
to some solution $z(x, y, t) \in \mathbb{C}$ of a (still to be specified) ``progenitor'' equation for which long-time well-posedness can be established independently of the dynamics of $A$, and that indirectly controls the $H^s$-norm of $A$.  Here $k > 0$ is to be chosen later independent of $\epsilon > 0$, and $\omega, \omega^\prime$ are functions of $k$.  The ``$O(\epsilon^2)$'' in \eqref{WavePacketSolution2} typically contains higher order correctors; the more of these higher order terms that are retained, the closer one can expect the approximate solution to agree with the true solution $z$.  The scaling of the above ansatz on the perturbation parameter $\epsilon$ is balanced to yield an NLS-type equation over $O(\epsilon^{-2})$ time scales.  The Fourier transform of the type of wave packets we consider here are given by
$$(F(\epsilon x, \epsilon y)e^{ikx})^{\wedge}(\xi_1, \xi_2) = \frac{1}{\epsilon^2} \hat{F}\left(\frac{\xi_1 - k}{\epsilon}, \frac{\xi_2}{\epsilon} \right)$$
and are highly concentrated about the frequency $(k, 0)$. We also record the following scalings in $\epsilon$ for wave packets:
$$\|\epsilon F(\epsilon x, \epsilon y)e^{ikx}\|_{L^2_{xy}} = \|F(X, Y)\|_{L^2_{XY}} = O(1)$$ $$\|\epsilon F(\epsilon x, \epsilon y)e^{ikx}\|_{L^\infty_{xy}} = \epsilon\|F(X, Y)\|_{L^\infty_{XY}} = O(\epsilon)$$

In this paper it will be necessary to develop the approximate solution $\tilde{z}$ so that the residual terms are of size $O(\epsilon^6)$ in Sobolev space so that the remainder $z - \tilde{z}$ is $O(\epsilon^3)$ in Sobolev space.  In this work, unlike in the usual full justification discussed in \S 1.1, we need only choose a single sufficiently small value of $\epsilon$ and a single solution to the progenitor equation in order to make conclusions about the behavior of the NLS solution $A$.  However, we must also choose an equation for $z$ for which long-time existence can be shown without appealing to the quantitative behavior of $A(T)$ for $T > 0$.

The progenitor equation governing the solution $z(t)$ which this wave packet approximates is motivated by \cite{Totz3DNLS}, and essentially consists of a toy version of the water wave problem\footnote{Specifically, the water wave problem under the influence of gravity alone to recover (qHNLS), and acting under capillary forces alone to recover (qNLS).} with an appropriately scaled power nonlinearity and a suitably constructed penalization term.

We collect the notation needed to state the progenitor equation in the unknown $z$.  Let $\nu > 1$ be an arbitrary positive parameter independent of $\epsilon$ to be determined later.  Define the control norm by
\begin{equation}\label{ControlNorm}
\lambda(t) = \frac{\|\,\Lambda^s_\epsilon z_{tt}(t)\|^2}{\omega^4\nu^2 \|A_0\|_{H^s}^2}
\end{equation}
Recalling that $\Lambda^s_\epsilon(Ae^{ikx}) = (\Lambda^s A)e^{ikx}$, if one grants that $z$ and $\tilde{z}$ are sufficiently close over $O(\epsilon^{-2})$ time scales, then to leading order one expects $$\lambda(t) = \frac{\|A(\epsilon^2t)\|_{H^s}^2}{\nu^2\|A_0\|_{H^s}^2} + O(\epsilon)$$  Notice that $\lambda(t)$ is constructed so that its leading order is independent of $\omega$ and $\epsilon$; hence any estimates based on $\lambda(t)$ will not depend on these quantities.  Note also that $\lambda(0) = \nu^{-2}$ up to terms of order $\epsilon$, which is uniformly bounded away from $1$ as $\epsilon \to 0$.

We would like to design an equation that forces $\lambda(t)$ to remain bounded by $1$ for all $0 \leq t \leq \mathscr{T}\epsilon^{-2}$, since this is equivalent to proving Claim 1.  For $\lambda \in [0, 1)$, we define the function $g(\lambda)$ satisfying:
\begin{align}\label{gODE}
& g(\lambda) = 1 - \sqrt{1 - \lambda} \qquad \qquad g^{\prime}(\lambda) =  \frac{1}{2\sqrt{1 - \lambda}} \notag
\end{align}
Notice that $g$ is uniformly bounded on $[0, 1)$, but $g^\prime$ diverges as $\lambda \to 1$.  

We are at last ready to state the evolution equation governing $z$.  Recall that $\mathcal{B}_k$ is the mode filter whose Fourier transform is the characteristic function of the ball centered at $(k, 0)$ and of radius $\frac{k}{2}$.  Consider
\begin{equation}\label{zEquationOutline}
\begin{cases}
z_{tt} + |D|^p z + C_k \epsilon^{2}\mathcal{B}_k(z|z/\epsilon|^{q - 1}) + \mathcal{N}(z, \omega, \epsilon, t) g(\lambda(t)) = 0, \\
\mathcal{N}(z, \omega, \epsilon, t) := \omega^2\mathcal{B}_k\left(\epsilon^{q + 4}z + 2i\epsilon^5 (\epsilon^2 t + T_1) z|z|^{q - 1}\right) \\
z(0) = \mathcal{B}_k \tilde{z}(0) \\
z_t(0) = \mathcal{B}_k \tilde{z}_t(0)
\end{cases}
\end{equation}

By definition, $\tilde{z}_t(0)$ is calculated by differentiating the formula for $\tilde{z}(t)$ with respect to $t$, replacing any derivatives with respect to $T$ in the functions of slow variables with spatial derivatives and nonlinearities using \eqref{NormalizedNLS}, and evaluating at $t = 0$.  We will only need to set $p = 1, 3$ in the sequel; the case $p = 1$ corresponds to (qHNLS) and the case $p = 3$ to (qNLS).  The strange scaling in the power nonlinearity of \eqref{zEquationOutline} is designed to contribute the appropriate power nonlinearity to \eqref{NormalizedNLS} in the multiple scale calculation.  Similarly, the high powers of $\epsilon$ in the penalization term are chosen to make this term negligible with respect to the error $z - \tilde{z}$.  The coefficient $\mathcal{N}$ is chosen to exploit the growth condition of Assumption 1.

The first task in giving a full justification for \eqref{zEquationOutline} is to construct an approximate solution $\tilde{z}$ as above satisfying the following properties:

\begin{proposition}\label{FormalCalculation}
Let $k > 0$ be given.  Then there exists an $N = N(s)$, and a function $\tilde{z} = \sum_{n = 1}^6 \epsilon^n z^{(n)}$ satisfying 
\begin{itemize}
\item[(1)]{$z^{(1)} = A(\epsilon(x + \omega^\prime t), \epsilon y, \epsilon^2 t)e^{i(kx + \omega t)}$, where $A$ solves the initial value problem \eqref{NormalizedNLS}.}
\item[(2)]{$\|\Lambda^s_\epsilon \tilde{z}\| \leq C(k, \|A\|_{H^N})$.}
\item[(3)]{$\|\Lambda^s_\epsilon(\tilde{z} - \epsilon z^{(1)})\| \leq C(k, \|A\|_{H^N})\epsilon$.}
\item[(4)]{$\tilde{z}$ satisfies the equation $$\tilde{z}_{tt} + |D|^p\tilde{z} + C_k \epsilon^{2} \tilde{z}|\tilde{z}/\epsilon|^{q - 1} = \epsilon^7 \tilde{\mathcal{R}}$$ where $\|\Lambda^s_\epsilon \tilde{\mathcal{R}}\| \leq C(k, \|A\|_{H^N})\epsilon^{-1}$.}
\end{itemize}
\end{proposition}

\begin{proof}  See Section 2.  \end{proof}

The sense in which $z$ and $\tilde{z}$ are close is given by

\begin{proposition}\label{RemainderEstimatesInOutline}
Let $\mathscr{T}_0 \leq \mathscr{T}$ be given, and let $\tilde{z}$ be the approximate solution defined by \eqref{ApproximateSolutionFormula}.  For some $\nu \in (1, 2)$, assume the existence of a function $z$ solving the initial value problem given by \eqref{zEquationOutline} for all $(x, y) \in \mathbb{R}^2$ and almost every time $t \in [0, \mathscr{T}_0\epsilon^{-2}]$ satisfying $z = \mathcal{B}_k z$ and $$(z, z_t, z_{tt}) \in (C^1 \times C^0 \times L^\infty)([0, \mathscr{T}_0\epsilon^{-2}] : L^2).$$  Then there exists an $\epsilon_0 > 0$ sufficiently small so that for all $0 < \epsilon < \epsilon_0$ the following holds: for almost every $t \in [0, \mathscr{T}_0\epsilon^{-2}]$ we have the bounds:
\begin{equation}\label{RemainderBoundsInOutline}
\|\Lambda^s_\epsilon(z_{tt}(t) - \tilde{z}_{tt}(t))\| + \|\Lambda^s_\epsilon(z_{t}(t) - \tilde{z}_t(t))\| + \|\Lambda^s_\epsilon(z(t) - \tilde{z}(t))\| \leq C(s, k, \|A\|_{H^N})\epsilon^3
\end{equation}
where in particular $C$ is independent of $\mathscr{T}_0$.
\end{proposition}

\begin{proof}  See Lemma \ref{RemainderEstimates} in Section 4.  \end{proof}

Finally, we must show that \eqref{zEquationOutline} is well-posed for long times.  If one grants a solution to \eqref{zEquationOutline} exists, one can show using an energy type argument applied to \eqref{zEquationOutline} that $\lambda(t)$ satisfies a differential inequality that forces $\lambda(t)$ to remain bounded uniformly away from 1.  In caricature, the differential inequality in question is of the form
\begin{equation}\label{CartoonODE}
\Bigl(1 - \mathfrak{c} g^\prime(\lambda(t))\Bigr)\lambda^\prime(t) > 0
\end{equation}
where here $\mathfrak{c} = \mathfrak{c}(s, q, \|u_0\|_{H^s}, \epsilon) \ll 1$ (see the calculations preceding \eqref{ODEBoundOnLambda} for a more precise statement).  Now if $g^\prime(\lambda(t)) \geq 2/\mathfrak{c}$, that is, if $\lambda(t) \geq 1 - (\mathfrak{c}/4)^2$, then we would conclude that $\lambda^\prime(t) < 0$, which provides our uniform bound.

However, there is a price to be paid for forcing $\lambda(t)$ to obey the ODE \eqref{CartoonODE}: the the equation \eqref{zEquationOutline} is a fully nonlinear equation with non-Lipschitz dependence, for which it is not even clear that local solutions can be constructed.  The long time existence of \eqref{zEquationOutline} is therefore the most difficult step of the argument, and is summarized in the

\begin{proposition}\label{zLongTimeWellPosed}
Let $1 < \nu < 2$ be given.  Suppose that for all $T \in [0, \mathscr{T}]$, $A \in C([0, \mathscr{T}], H^N)$ satisfies Assumption 1.  Then there exists an $\epsilon_0 > 0$ sufficiently small depending on $s$, $p$, $q$, $k$, $\|A_0\|_{H^N}$, $\mathscr{T}$ as well as a choice of $k > 0$ depending on $\mathbf{T}, \|u\|_{C^0([0, \mathbf{T}] : L^\infty)}$ so that the following holds: there exists a function $z(x, y, t)$ on $\mathbb{R}^2 \times [0, \mathscr{T}\epsilon^{-2}]$ for which $z(t) = \mathcal{B}_k z(t)$ for every $t \in [0, \mathscr{T}\epsilon^{-2}]$, for which $\Lambda^s_\epsilon z_{tt}$ is defined and uniformly bounded for every $(x, y, t)$ with
\begin{equation*}
\Lambda^s_\epsilon z_{tt} \in L^\infty([0, \mathscr{T}\epsilon^{-2}], L^2)
\end{equation*}
satisfying the key bound
\begin{equation}\label{LambdaLessThan1InOutline}
\max_{0 \leq t \leq \mathscr{T}\epsilon^{-2}} \frac{\|\,\Lambda^s_\epsilon z_{tt}(t)\|^2}{\omega^4 \nu^2 \|A_0\|_{H^s}^2} < 1
\end{equation}
and which solves the initial value problem given by \eqref{zEquationOutline} for every $(x, y) \in \mathbb{R}^2$ and almost every $t \in [0, \mathscr{T}\epsilon^{-2}]$.
\end{proposition}

\begin{proof}  See Section 4 for the local well-posedness, and Section 5 for the long-time well posedness.  We pause to note that a sufficiently large choice of $\mathscr{C}$ in the growth condition of Assumption 1 and a sufficiently small choice of $\omega$ play the role of small parameters that allow us to control the full nonlinearity in \eqref{zEquationOutline}.
\end{proof}

Assuming these propositions we can give the 

\begin{proof}[Proof of Claim 1.]
Consider $\epsilon \in (0, \epsilon_0]$ for an $\epsilon_0 > 0$ to be taken sufficiently small in what follows.  Using $A$, $\epsilon$, $k$ and $p$ we construct the approximate solution $\tilde{z}$ given by Proposition \ref{FormalCalculation}.  Introduce $1 < \nu < 2$ to be determined later.  Then since Assumption 1 holds by hypothesis, we may choose $k > 0$ and $\epsilon_0 > 0$ sufficiently small so as to invoke Proposition \ref{zLongTimeWellPosed} to construct the almost everywhere solution $z(t)$ to \eqref{zEquationOutline} that exists and satisfies  \eqref{LambdaLessThan1InOutline} on $[0, \mathscr{T}\epsilon^{-2}]$.  Choose $\epsilon_0 > 0$ sufficiently small so that the conclusion of Proposition \ref{RemainderEstimatesInOutline} holds.  Then using \eqref{LambdaLessThan1InOutline}, Propositions \ref{FormalCalculation}, \ref{RemainderEstimatesInOutline} and \ref{zLongTimeWellPosed}, we have
\begin{align*}
 \|A(\epsilon^2 t)\|_{H^s} & \leq \omega^{-2}\|\, \Lambda^s_\epsilon \tilde{z}_{tt}(t)\| + C(k, \|A\|_{H^N})\epsilon \\
& \leq \omega^{-2}\|\Lambda^s_\epsilon z_{tt}(t)\|  + \omega^{-2}\|\Lambda^s_\epsilon (z_{tt}(t) - \tilde{z}_{tt}(t))\| + C(k, \|A\|_{H^N})\epsilon \\
& \leq \omega^{-2}\|\, \Lambda^s_\epsilon z_{tt}(t)\| +  C(k, \|A\|_{H^N})\epsilon \qquad \text{a.e. } t \in [0, \mathscr{T}\epsilon^{-2}] \\
& < \nu \|A_0\|_{H^s} + C(k, \|A\|_{H^N})\epsilon \qquad \text{a.e. } t \in [0, \mathscr{T}\epsilon^{-2}]
\end{align*}
Finally, choose $\epsilon_0 > 0$ so small so that $C(k, \|A\|_{H^N})\epsilon < (\nu - 1) \|\,\Lambda^s A_0\|$.  Then with $T = \epsilon^2 t$ we have $$\|\Lambda^s A(T)\| < (2\nu - 1) \|\Lambda^s A_0\| \qquad \text{a.e. } T \in [0, \mathscr{T}]$$ By Proposition \ref{HNLSLocalWellposedness}, $\|\Lambda^s A(T)\|$ is a continuous function of $T$, and so
$$\|\Lambda^s A(T)\| \leq (2\nu - 1) \|\Lambda^s A_0\| \qquad T \in [0, \mathscr{T}]$$ which is the conclusion of Claim 1.
\end{proof}

\begin{remark}
We draw special attention to the fact that the choice of $\epsilon$ plays an important analytic role in closing the above argument by eliminating a constant that depends upon the $H^N$-norm of $A$.  This control is vital, since there is of course no reason that the $H^N$-norm of $A$ should be controlled by the $H^s$-norm of $A$.\footnote{Indeed, by an inspection of the proof of Lemma \ref{MollificationArgument}, we must consider scenarios in which the $H^N$-norm is arbitrarily large as $\iota \to 0$ in order to prove the result for low regularity initial data.}
\end{remark}

%==============================================================================
%==============================================================================
%==============================================================================
\section{Formal Calculation of the Approximate Solution}
%==============================================================================
%==============================================================================
%==============================================================================

In this section we derive the formal approximate solution $\tilde{z}(t)$.  The derivation of $\tilde{z}(t)$ is a routine multiple scales calculation, and is given after the presentation of some standard lemmas giving the formal expansion of the Fourier multipliers appearing in our equations.

%==============================================================================
%==============================================================================
\subsection{Lemmas on Multiple Scales Expansions}
%==============================================================================
%==============================================================================

\begin{lemma}\label{WavePacketCutoff}
Let $s, m \geq 0$ be integers, $k > 0$ and $\epsilon > 0$ be given.  Then for any function $F \in H^{s + m}$, there is a constant $C$ depending on $k, s, m$ so that $$\left\|(I - \mathcal{B}_k)\left(F(\epsilon x, \epsilon y)e^{ikx}\right)\right\|_{H^s_{xy}} \leq C \epsilon^{m - 1}\|F\|_{H^{s + m}_{XY}}$$
\end{lemma}

\begin{proof}
We calculate by Plancherel's Identity that for any $m \geq 0$ that
\begin{align*}
& \quad\; \left\|(I - \mathcal{B}_k)\left(F(\epsilon x, \epsilon y)e^{ikx}\right)\right\|_{L^2_{xy}} \\
& = \left(\int_{|(\xi_1 - k, \xi_2)| > \frac{k}{2}} \left|\frac{1}{\epsilon^2}\hat{F}\left(\frac{\xi_1 - k}{\epsilon}, \frac{\xi_2}{\epsilon}\right)\right|^2 d\xi \right)^{1/2} \\
& \leq \left(\int_{|(\xi_1 - k, \xi_2)| > \frac{k}{2}} \left|\frac{\epsilon^m}{\langle(\xi_1 - k, \xi_2)\rangle^m}\frac{1}{\epsilon^2}\widehat{\Lambda^m F}\left(\frac{\xi_1 - 1}{\epsilon}, \frac{\xi_2}{\epsilon}\right)\right|^2 d\xi\right)^{1/2} \\
& \leq C \epsilon^{m - 1}\|F\|_{H^m_{XY}}
\end{align*} where the constant $C$ depends only on $m$ and $k$.  Note that we have lost a power of $\epsilon$ by measuring $F$ in $L^2_{XY}$ where $X = \epsilon x$, $Y = \epsilon y$.  Since $\mathcal{B}_k$ commutes with differentiation, the result now follows upon applying the above to $\partial^j F$ for $|j| \leq s$.
\end{proof}

In order to perform the formal expansion of the approximate solution, we first need an expansion of $|D|$ about a wave packet $Fe^{i\phi}$, that is, about the fixed frequency $\xi = (k, 0)$.  Using the expansion of the symbol $|\xi|$ in Fourier space
\begin{equation}\label{|D|FrequencyExpansion}
|\xi| = |k| + \frac{k}{|k|}(\xi_1 - k) + \frac{1}{|k|}\xi_2^2 + \cdots \end{equation}
we find by formal Taylor expansion that the corresponding series development of the symbol of $|D|^p$ is given by
\begin{align}\label{|D|pFrequencyExpansion}
|\xi|^p & = |k|^p + pk|k|^{p - 2}(\xi_1 - k) + p|k|^{p - 2}\xi_2^2 + \frac12p(p - 1) |k|^{p - 2} (\xi_1 - k)^2\notag \\
& \quad + \sum_{j = 3}^5 P_j(\xi_1 - k, \xi_2) + O(|(\xi_1 - k, \xi_2)|^6)
\end{align}
where the polynomials $P_j$ are of degree $j$, have coefficients depending on $k$ and $p$, and will not be explicitly used in the formal calculation.  When  applied to a wave packet of the form $F(\epsilon x, \epsilon y)e^{ikx}$ and written in physical space, the multiplier formally corresponds to the multiscale operator
\begin{align}\label{|D|Expansion}
|D|^p & = |k|^p - ipk|k|^{p - 2}\epsilon\partial_{x_1} - p|k|^{p - 2}\epsilon^2\partial_{y_1}^2 - \frac12p(p - 1) |k|^{p - 2} \partial_{x_1}^2\notag \\
& \quad + \sum_{j = 3}^5 \epsilon^ j P_j(-i \partial_{x_1}, -i \partial_{y_1}) + \cdots \\
& := \sum_{j = 0}^5 \epsilon^j (|D|^p)^{(j)} + \cdots \notag
\end{align}
To make this rigorous we present the
\begin{lemma}\label{|D|ExpansionLemma}
Let $k \neq 0$, $p > 0$ and $\epsilon > 0$ be given.  Then for any function $F \in H^{p + 6}$, there is a constant $C$ depending only on $k$ so that $$\left\|\left(|D|^p - \sum_{j = 0}^5 \epsilon^j (|D|^p)^{(j)} \right)(Fe^{ikx})\right\| \leq C \epsilon^{5}\|F\|_{H^{6 + p}}$$
\end{lemma}

\begin{proof}
We may use Lemma \ref{WavePacketCutoff} to reduce the estimate to wave packets of the form $\mathcal{B}_k F e^{ikx}$;  specifically we have
\begin{align*}
\left\|\left(|D|^p - \sum_{j = 0}^5 \epsilon^j (|D|^p)^{(j)} \right)(Fe^{ikx})\right\| & = \left\|\left(|D|^p - \sum_{j = 0}^5 \epsilon^j (|D|^p)^{(j)} \right)(\mathcal{B}_kFe^{ikx}) \right\| \\
& \quad + \left\|(I - \mathcal{B}_k)|D|^p(Fe^{ikx})\right\| \\
& \quad + \sum_{j = 0}^5 \epsilon^j \left\|(I - \mathcal{B}_k ) (|D|^p)^{(j)} (Fe^{ikx})\right\|
\end{align*}
The third term above is bounded by $C\epsilon^5\|F\|_{H^7}$, and by Lemma 1 the second term is bounded by $C\epsilon^5\|F\|_{H^{6 + p}}$.  Now writing $|\xi|^p = \xi_1^p(1 + \xi_2^2/\xi_1^2)^\frac{p}{2}$ with principal branch cuts, we have that the Taylor series of $|\xi|^p$ converges absolutely and uniformly on any compact contained in $\{\xi : |\xi_1 - k| < k \; \text{and} \; |\xi_2| < |\xi_1|\}$.  Since the support of $\hat{\mathcal{B}_k}$ is compactly contained in this set, we are justified in expanding the symbol of $|D|^p$ in Fourier space on the support of $\hat{\mathcal{B}}_k$, and doing so gives us the estimate
\begin{align*}
& \quad \left\|\left(|D|^p - \sum_{j = 0}^5 \epsilon^j (|D|^p)^{(j)} \right)(\mathcal{B}_kFe^{ikx}) \right\| \\
& = \left(\int_{|(\xi_1 - k, \xi_2)| \leq \frac{1}{2}k} \left| |\xi|^p - \sum_{j = 0}^5 P_k(\xi_1 - k, \xi_2)\right|^2 \left|\frac{1}{\epsilon^2}\hat{F}\left(\frac{\xi_1 - k}{\epsilon}, \frac{\xi_2}{\epsilon}\right)\right|^2 d\xi \right)^{1/2} \\
& \leq C_k \left(\int_{|(\xi_1 - k, \xi_2)| \leq \frac{1}{2}k} |(\xi_1 - k, \xi_2)|^{12} \left|\frac{1}{\epsilon^2}\hat{F}\left(\frac{\xi_1 - k}{\epsilon}, \frac{\xi_2}{\epsilon}\right)\right|^2 d\xi \right)^{1/2} \\
& \leq C\epsilon^5 \|F\|_{H^6}
\end{align*}
\end{proof}

%==============================================================================
%==============================================================================
\subsection{Multiscale Calculation of the Approximate Solution}
%==============================================================================
%==============================================================================

We now give the construction of the function $\tilde{z}$ that satisfies the equation
\begin{equation}\label{ApproximateEquation}
\tilde{z}_{tt} + |D|^p \tilde{z} + C_\omega \epsilon^2\tilde{z} |\tilde{z}/\epsilon|^{q - 1} = \epsilon^7 \tilde{\mathcal{R}}
\end{equation}
We seek a formal approximate solution in the form of an asymptotic series $$\tilde{z} \sim \sum_{n = 1}^{\infty} z^{(n)}$$ where each term $z^{(n)} = z^{(n)}(x_0, x_1, y_1, t_0, t_1, t_2)$ is a function of the multiscale variables.  

Substitute the asymptotic series into \eqref{ApproximateEquation}.  Since we insist that $z^{(1)} = Ae^{i(kx + \omega t)} =: Ae^{i\phi}$, collecting the terms of size $O(\epsilon)$ gives
$$(-\omega^2 + k^p)z^{(1)} = 0$$ which is solved by taking $\omega^2 = k^p$.  Next, the $O(\epsilon^2)$ terms yield
\begin{align}\label{Epsilon2Terms}
(\partial_{t_0}^2 + (|D|^p)^{(0)})z^{(2)} & = -(2\partial_{t_0}\partial_{t_1} - ipk|k|^{p - 2}\partial_{x_1})z^{(1)} \\
& = -i(2\omega\partial_{t_1} - pk|k|^{p - 2}\partial_{x_1})Ae^{i\phi}, \notag
\end{align}
where $\omega^\prime$ denotes the group velocity $$\omega^\prime = \frac{pk|k|^{p - 2}}{2\omega} = \frac12 p|k|^{\frac{1}{2}(p - 2)}$$  Now if the right hand side of \eqref{Epsilon2Terms} is not zero, then \eqref{Epsilon2Terms} has a resonant forcing term, in which case $z^{(2)}$ will necessarily have a contribution having an amplitude that grows linearly in time.  Since we are considering solutions over time scales on the order $O(\epsilon^{-2})$, such a term would eventually be of order $O(1)$, which would contradict the fact that we seek a solution in the form of an \textit{asymptotic} series.  Our assumption that the formal series solution is asymptotic over $O(\epsilon^{-2})$ time scales therefore forces the resonant in the right hand side of \eqref{Epsilon2Terms} to vanish.  We arrange for this to occur by assuming that $A = A(X, Y, T)$ alone, where we introduce the slow characteristic variable $X = x_1 + \omega^\prime t_1$, as well as the slow variables $Y = y_1$, $T = t_2$.  Since we will need the freedom later in the calculation, we also take $z^{(2)} = A^{(2)}(X, Y, T)e^{i\phi}$ where $A^{(2)}$ is yet to be determined.  Collecting the $O(\epsilon^3)$ terms gives us the equation
\begin{align*}
(\partial_{t_0}^2 + (|D|^p)^{(0)})z^{(3)} & = -(2\partial_{t_0}\partial_{t_1} - i\partial_{x_1})z^{(2)} \\
& \quad - \left(2\partial_{t_0}\partial_{t_2} + \partial_{t_1}^2 + (|D|^p)^{(2)}\right)z^{(1)} \\
& \quad - C_\omega A|A|^{q - 1} e^{i\phi}
\end{align*}
\begin{align*}
& = -\left(2i\omega\partial_T + \left((\omega^\prime)^2 - \frac12 p(p - 1) |k|^{p - 2}\right) \partial_X^2 - |pk|^{p - 2} \partial_Y^2 + C_\omega |A|^{q - 1}\right)Ae^{i\phi}
\end{align*}
We can arrange for the right hand side of this equation to vanish if we insist that $A$ satisfies the NLS equation
\begin{equation}\label{ActualHNLS}
2i\omega A_T + \frac14p(2 - p)|k|^{p - 2}A_{XX} - pk^{p - 2}A_{YY} + C_\omega A|A|^{q - 1} = 0
\end{equation}
Notice that for any choice of $k > 0$, the signs of the coefficients of $A_{XX}$ and $A_{YY}$ are opposite whenever $0 < p < 2$ and agree when $p > 2$; this allows us to treat both the defocusing elliptic case and the hyperbolic case simultaneously.  Notice also that $k^{p - 2} = \omega^{2 - \frac{4}{p}}$; therefore by multiplying this equation by $\omega^{\frac{4}{p} - 2}$, we see that this equation agrees with the rescaled equation \eqref{NormalizedNLS} provided we choose $C_\omega = \omega^{2 - \frac{4}{p}}$.

Next we must be able to construct in principle higher order correctors in order to construct an approximate solution with sufficiently small residual.  We proceed inductively, assuming that the correctors $z^{(1)} = Ae^{i\phi}$, $z^{(2)} = A^{(2)}e^{i\phi}$, \ldots, $z^{(n)} = A^{(n)}e^{i\phi}$ have been found.  Collecting the $O(\epsilon^{n + 1})$ terms yields an equation for $z^{(n + 1)}$ of the form
\begin{align}\label{zn+1Equation}
(\partial_{t_0}^2 + (|D|^p)^{(0)})z^{(n + 1)} & = -(2\partial_{t_0}\partial_{t_1} - i\partial_{x_1})z^{(n)} \notag \\
& \quad - \left(2\partial_{t_0}\partial_{t_2} + \partial_{t_1}^2 + (|D|^p)^{(2)}\right)z^{(n - 1)} \\
& \quad - \sum_{j_1 + j_2 + j_3 = n + 1}^{j_1 + j_2 \geq 2} \partial_{t_{j_1}} \partial_{t_{j_2}} z^{(j_3)} - \sum_{j = 3}^{n} (|D|^p)^{(j)} z^{(n + 1 - j)} \notag  \\
& \quad - \sum_{j_1 + j_2 + \cdots + j_{q } = q + n - 2} \left(\prod_{i = 1}^{\frac{q + 1}{2}} z^{(j_i)} \right) \left( \prod_{i = \frac{q + 3}{2}}^{q} \overline{z}^{(j_i)} \right) \notag
\end{align}
Since $A^{(n)} = A^{(n)}(X, Y, T)$, the first line on the right hand side of \eqref{zn+1Equation} vanishes.  The other terms on the right hand side of \eqref{zn+1Equation} can be arranged to vanish as well provided we choose $A^{(n - 1)}$ to satisfy the equation
\begin{align}\label{CorrectorHNLS}
& \quad 2i\omega A_T^{(n - 1)} + \frac14p(2 - p)k^{p - 2}A_{XX}^{(n - 1)} - pk^{p - 2} A_{YY}^{(n - 1)} \notag  \\
& \qquad = - \frac{q + 1}{2}A^{(n - 1)}|A|^{q - 1} - \left(\frac{q - 1}{2}\right)A^2\overline{A}^{(n - 1)}|A|^{q - 3} \notag \\
& \qquad - \sum_{j_1 + j_2 + j_3 = n + 1}^{j_1 + j_2 \geq 2} \partial_{t_{j_1}} \partial_{t_{j_2}} A^{(j_3)} - \sum_{j = 3}^{n} (|D|^p)^{(j)} A^{(n + 1 - j)} \\
& \qquad - \sum_{j_1 + j_2 + \cdots + j_{q} =q +  n - 2} \left(\prod_{i = 1}^{\frac{q + 1}{2}} A^{(j_i)} \right) \left( \prod_{i = \frac{q + 3}{2}}^{q } \overline{A}^{(j_i)} \right) \notag
\end{align}
For our purposes we need only take the approximate solution to six terms, and so we define
\begin{equation}\label{ApproximateSolutionFormula}
\tilde{z} = \sum_{n = 1}^6 \epsilon^n z^{(n)} = \sum_{n = 1}^6 \epsilon^n A^{(n)}e^{i\phi}
\end{equation}
where $A^{(1)} = A$.  For convenience we set $A^{(j)}(0) = 0$ for all $j \geq 2$.

In order to account for the number of derivatives of $A$ needed to construct the correctors in the approximate solution, we have the

\begin{lemma}\label{DerivativeCountingInApprox}
Suppose that $A$ is a solution to \eqref{ActualHNLS} that exists on a time interval $[0, \mathscr{T}]$ and satisfies $A(0) = A_0 \in H^N$.  Then the modulations $A^{(j)}$, $j = 2, \ldots, 6$ of the correctors satisfying \eqref{CorrectorHNLS} and $A^{(j)}(0) = 0$ also exist on $[0, \mathscr{T}]$ with values in $H^{N - 2j}$.  Therefore if we choose $N(s) > s + 21$ we have the bounds 
\begin{itemize}
\item[(i)]{$\|\Lambda^s_\epsilon \tilde{z}\| \leq C(k, \|A\|_{H^N})$}
\item[(ii)]{$\|\Lambda^s_\epsilon(\tilde{z} - z^{(1)})\| \leq C(k, \|A\|_{H^N})\epsilon$}
\item[(iii)]{$\|\Lambda^s_\epsilon \tilde{\mathcal{R}}\| \leq C(k, \|A\|_{H^N})\epsilon^{-1}$}
\item[(iv)]{$\|\Lambda^s_\epsilon \partial_t \tilde{\mathcal{R}}\| \leq C(k, \|A\|_{H^N})\epsilon^{-1}$}
\end{itemize}
\end{lemma}

\begin{proof}
Since we only need to consider the cases $p = 1, 3$, we require that $A^{(j)} \in H^9$ in order to apply Lemmas \ref{WavePacketCutoff} and \ref{|D|ExpansionLemma}.  Consider  \eqref{CorrectorHNLS} written using the DuHamel formulation.  Now the equations for the $A^{(n + 1)}$ are linear with the variable coefficients consisting of products of $A$ and $\overline{A}$.  By the induction hypothesis, we may replace any time derivatives falling on the factors $A, A^{(2)}, \ldots, A^{(n)}$ by the corresponding space derivatives and higher order nonlinearities.  Since each time derivative is replaced by at most two space derivatives in this way, the nonlinearity consists of factors depending on at most $2(n - j)$ derivatives of $A^{(j)}$, Gr\"onwall's inequality along with the Sobolev embedding $H^2 \hookrightarrow L^\infty$ with a quick induction argument implies that $\|A^{(n - 1)}\|_{H^m} \leq C(\|A\|_{H^{2(n - 1) + m}})$.  Since we need only set $p = 1, 3$, this implies that we need control over $\|A^{(6)}\|_{H^9} \leq C(\|A\|_{H^{19}})$, and similarly the bounds (i), (ii) follow provided we choose $N(s) \geq s + 19$.  In the same way, the highest number of derivatives appearing in the residual $\tilde{\mathcal{R}}$ is through $\partial_t^2$ acting on $A^{(6)}$, so that (iii) follows upon choosing $N(s) > s + 19$, and (iv) then follows upon taking $N(s) > s + 21$.
\end{proof}

In the process of deriving the ODE for $\lambda(t)$ we encounter the expression
$$\langle |D|^{2s}_\epsilon \tilde{z}_{tt}, -|D|^p\tilde{z}_t - \epsilon^{2}\partial_t\mathcal{B}_k(\tilde{z}|\tilde{z}/\epsilon|^{q - 1})\rangle$$
We explicitly calculate the leading order term of this expression here.

\begin{proposition}\label{LambdaPrimeMultiscale}
Assume that $A \in C([0, \mathscr{T}] : H^{N(s)})$ for $N(s) > s + 21$.  Then the following leading order identity holds:
\begin{equation*}
\left| \langle \Lambda^{2s}_\epsilon \tilde{z}_{tt}, -|D|^p \tilde{z}_t - \epsilon^{2}\partial_t\mathcal{B}_k(\tilde{z}|\tilde{z}/\epsilon|^{q - 1})\rangle - C_\omega \epsilon^2 \langle \Lambda^{2s}A, i A|A|^{q - 1} \rangle \right| \leq C(\|A_0\|_{H^N})\epsilon^3
\end{equation*}
\end{proposition}

\begin{proof}
Upon differentiating \eqref{ApproximateEquation} with respect to $t$ we have the following identity:
\begin{equation}
-|D|^p\tilde{z}_t - \epsilon^{2}\partial_t\mathcal{B}_k(\tilde{z}|\tilde{z}/\epsilon|^{q - 1}) = \tilde{z}_{ttt} - \epsilon^7 \partial_t \tilde{\mathcal{R}}
\end{equation}
Now using Lemmas \ref{WavePacketCutoff} and \ref{DerivativeCountingInApprox}, we have
\begin{equation}
\left| \langle \Lambda^{2s}_\epsilon \tilde{z}_{tt}, -|D|^p\tilde{z}_t - \epsilon^{2}\partial_t\mathcal{B}_k(\tilde{z}|\tilde{z}/\epsilon|^{q - 1})\rangle - \langle \Lambda^{2s}_\epsilon \tilde{z}_{tt}, \tilde{z}_{ttt} \rangle \right| \leq C(\|A\|_{H^N})\epsilon^6
\end{equation}
and so it suffices to calculate the leading order term of $\langle \Lambda^{2s}_\epsilon \tilde{z}_{tt}, \tilde{z}_{ttt} \rangle$.  By \eqref{ApproximateSolutionFormula}, each time derivative $\partial_t$ acts on $\tilde{z}$ as $i\omega + \omega^\prime \epsilon \partial_X + \epsilon^2 \partial_T$.  Now the contribution
$$\left\langle \Lambda^{2s}_\epsilon \left(i\omega + \omega^\prime \epsilon \partial_X\right)^2\tilde{z}, \left(i\omega + \omega^\prime \epsilon \partial_X\right)^3\tilde{z} \right\rangle$$ vanishes, since $\Lambda^s_\epsilon$ is self-adjoint and $(i\omega + \omega^\prime \epsilon \partial_X)$ is skew-adjoint.  Therefore in the nonzero contributions there must be at least one occurrence of $\epsilon^2 \partial_T$.  The leading order contribution of what remains is therefore of the form
$$\langle \Lambda^{2s}_\epsilon (i\omega)^2\tilde{z}^{(1)}, (i\omega)^2 \epsilon^2 \partial_T \tilde{z}^{(1)} \rangle = C_\omega \epsilon^2 \langle \Lambda^{2s}A, i A|A|^{q - 1} \rangle$$
with all other contributions of size at most $C(\|A\|_{H^N})\epsilon^3$ by Lemma \ref{DerivativeCountingInApprox}.
\end{proof}

%==============================================================================
%==============================================================================
%==============================================================================
\section{Remainder Estimates}
%==============================================================================
%==============================================================================
%==============================================================================

Since the power nonlinearity of \eqref{zEquationOutline} is constructed to scale in $\epsilon$ just as a cubic nonlinearity, we may follow \cite{KSMCubicNonlinearityLongtimeRemainder} and show the a priori estimates of the remainder $r = z - \tilde{z}$ on the interval $[0, \mathscr{T}\epsilon^{-2}]$ with a straightforward application of Gr\"onwall's inequality.

\begin{lemma}\label{RemainderEstimates}
Suppose $\mathscr{T}_0 \leq \mathscr{T}$ is given.  Suppose further that a solution $z$ to 
\begin{equation}\label{zEquation}
\begin{cases}
z_{tt} + |D|^p z + \epsilon^{2}\mathcal{B}_k(z|z/\epsilon|^{q - 1}) + \mathcal{N} g(\lambda(t)) = 0 \\
z(0) = \mathcal{B}_k \tilde{z}(0) \\
z_t(0) = \mathcal{B}_k \tilde{z}_t(0)
\end{cases}
\end{equation}
exists in the sense that $z = \mathcal{B}_k z$ and $$(z, z_t, z_{tt}) \in (C^1 \times C^0 \times L^\infty)([0, \mathscr{T}_0\epsilon^{-2}], L^2)$$ satisfying the bound $\lambda(t) < 1$ for all $t \in [0, \mathscr{T}_0\epsilon^{-2}]$, and $z$ solves \eqref{zEquation} for every $(x, y) \in \mathbb{R}^2$ and \textbf{almost every} $t \in [0, \mathscr{T}_0\epsilon^{-2}]$.  Then there exists a constant $C = C(s, p, q, k, \mathscr{T}, \|A\|_{H^N})$ that is independent of $\mathscr{T}_0$ and an $\epsilon_0 > 0$ sufficiently small depending on $s, p, q, k, \mathscr{T}, \|A\|_{H^N}$ so that the following estimate of $r := z - \tilde{z}$ holds for all $0 < \epsilon < \epsilon_0$ and almost every $t \in [0, \mathscr{T}_0\epsilon^{-2}]$:
$$\|\Lambda^s_\epsilon r(t)\|^2 + \|\Lambda^s_\epsilon r_t(t)\|^2 + \|\Lambda^s_\epsilon r_{tt}(t)\|^2 \leq C\epsilon^3$$
\end{lemma}

\begin{remark}\label{HighDerivsNotWavePackets}
Heuristically, we expect to control $r$ and $r_t$ using energy based arguments, and $r_{tt}$ using the equation governing $r$.  This is possible since the penalization term depends only on $\lambda$ through the term $g(\lambda(t))$, which is uniformly bounded as a function of $\lambda$.  Differentiating the progenitor equation in time shows that $\partial_t^j r_{tt}$ generally depends on $g^{(j)}(\lambda(t))$, over which we have no control for $j \geq 1$ since we expect that $\lambda$ will be near 1.  Therefore such higher derivatives $\partial_t^j z_{tt}$ do not remain in the modulation regime over NLS time scales.  Indeed this is necessary to the success of our program, since constructing the ODE \eqref{CartoonODE} depends on this occurring.
\end{remark}

\begin{proof}
The governing equation  \eqref{zEquation} reads
\begin{equation}\label{zEquationBeforeTStar}
z_{tt} + |D|^p z = -\epsilon^{2} \mathcal{B}_k(z |z/\epsilon|^{q - 1}) - \mathcal{B}_k \mathcal{N} g(\lambda(t))
\end{equation}
Subtracting from this the approximate equation
\begin{equation}\label{zTildeEquationBeforeTStar}
\tilde{z}_{tt} + |D|^p \tilde{z} = -\epsilon^{2} \tilde{z} |\tilde{z}/\epsilon|^{q - 1} + \epsilon^7\tilde{\mathcal{R}}
\end{equation}
gives the governing for the remainder $r := z - \tilde{z}$ which reads
\begin{align}\label{RemainderEquation}
r_{tt} + |D|^p r & = \epsilon^{2} \mathcal{B}_k(\tilde{z} |\tilde{z}/\epsilon|^{q - 1} - z|z/\epsilon|^{q - 1}) - \mathcal{B}_k \mathcal{N} g(\lambda(t))  \\
& \quad - \epsilon^7 \tilde{\mathcal{R}} + \epsilon^{2} (I - \mathcal{B}_k)(\tilde{z} |\tilde{z}/\epsilon|^{q - 1})\notag
\end{align}
Consider the energy
\begin{equation}\label{RemainderEnergy}
R(t) = \frac12 \|\Lambda_\epsilon^s  r_t\|^2 + \frac12 \|\,|D|^\frac{p}{2} \Lambda_\epsilon^s r\|^2
\end{equation}
By applying $\Lambda_\epsilon^s$ to \eqref{RemainderEquation} and taking the inner product with $\Lambda_\epsilon^s r_t$, we see that $R(t)$ is absolutely continuous and obeys the following differential inequality:
\begin{align}\label{RemainderEnergyEstimateStart}
\frac{dR}{dt} & \leq \epsilon^{2}|\langle \Lambda_\epsilon^s  (z |z/\epsilon|^{q - 1} - \tilde{z} | \tilde{z}/\epsilon|^{q - 1}), \Lambda_\epsilon^s r_t \rangle| \notag \\
& \quad + |\langle \Lambda_\epsilon^s  \mathcal{B}_k \mathcal{N}, \Lambda_\epsilon^s r_t \rangle| \\
& \quad + |\langle \Lambda_\epsilon^s \epsilon^7 \tilde{\mathcal{R}} , \Lambda_\epsilon^s r_t \rangle| \notag \\
& \quad + \epsilon^{2} |\langle \Lambda_\epsilon^s (I - \mathcal{B}_k)(\tilde{z} |\tilde{z}/\epsilon|^{q - 1}), \Lambda_\epsilon^s r_t \rangle | \hspace{1cm} \text{a.e. } t \in [0, \mathscr{T}_0\epsilon^{-2}] \notag \\ 
& := T_1 + T_2 + T_3 + T_4 \notag
\end{align}

In order to bound those terms, first note that Lemma \ref{WavePacketCutoff} as well as the fact that $z = \mathcal{B}_k z$ allows us to estimate:
\begin{align}\label{RemainderL1Hat}
\|r\|_{\hat{L}^1} & \leq \|\mathcal{B}_k r\|_{\hat{L}^1} + \|(I - \mathcal{B}_k)(z - \tilde{z})\|_{\hat{L}^1} \notag \\
& \leq C\|\mathcal{B}_k r\|_{H^2} + \|(I - \mathcal{B}_k)\tilde{z}\|_{\hat{L}^1} \notag \\
& \leq C_{k, p}\|\mathcal{B}_k |D|^\frac{p}{2} r\|_{L^2} + C\|(I - \mathcal{B}_k)\tilde{z}\|_{H^2} \\
& \leq C_{k, p}R(t)^\frac12 + C\epsilon^6 \notag 
\end{align}
We will only estimate $T_1$ in detail, since the other terms are strictly easier.  We have writing $z = \tilde{z} + r$ and using Lemma \ref{L1HatEstimate} and \eqref{RemainderL1Hat} that
\begin{align*}
T_1 & \leq \epsilon^2 \|\Lambda^s_\epsilon r_t\| \, \|\Lambda_\epsilon^s  (z |z/\epsilon|^{q - 1} - \tilde{z} | \tilde{z}/\epsilon|^{q - 1})\| \\
& = \epsilon^{3 - q} \|\Lambda^s_\epsilon r_t\| \, \|\Lambda_\epsilon^s  (z |z|^{q - 1} - \tilde{z} | \tilde{z}|^{q - 1})\| \\
& \leq C_{s, q} \epsilon^{3 - q} \|\Lambda^s_\epsilon r_t\| \, \|\Lambda^s_\epsilon r\| \sum_{j = 1}^{q} \|\Lambda^s_\epsilon r\|_{\hat{L}^1}^{j - 1} \|\Lambda^s_\epsilon \tilde{z}\|_{\hat{L}^1}^{q - j} \\
& \leq C_{s, q, k, p, \|A\|_{H^{s + 2}}} \epsilon^{3 - q} R(t) \sum_{j = 1}^{q} \left(R(t)^\frac12 + \epsilon^6 \right)^{j - 1} \epsilon^{q - j} \\
& \leq C_{s, q, k, p, \|A\|_{H^N}} \epsilon^2 R(t) \sum_{j = 1}^{q} \left(\epsilon^{-1}R(t)^\frac12 + \epsilon^5 \right)^{j - 1} \\
& \leq C_{s, q, k, p, \|A\|_{H^N}} \epsilon^2 R(t) \left(1 + \epsilon^{-1}R(t)^\frac12\right)^{q - 1}
\end{align*}
From this point on in the proof we suppress the dependence of constants $C$ on $s, q, k, p, \|A_0\|_{H^N}, \mathscr{T}$ for brevity.  Applying similar estimates to $T_2$, $T_2$, $T_4$ and noting that $|g(\lambda)| < 1$ and $q + 4 \geq 7$ then yields 
\begin{align*}
\frac{dR}{dt} & \leq C \epsilon^2 R(t)\epsilon^{4 - q}(1 + \epsilon^{-1}R(t)^\frac12)^{q - 1} \\
& \quad + C\epsilon^7(1 + \epsilon^{-1}R(t)^\frac12)^{q - 1})(1 + R(t)^\frac12)R(t)^\frac12 \\
& \quad + C R(t)^\frac12 \epsilon^6 & \text{a.e. } t \in [0, \mathscr{T}_0\epsilon^{-2}] \\
& \leq C R(t)^\frac12 \left( \epsilon^6 + C\epsilon^7(1 + \epsilon^{-1}R(t)^\frac12)^{q - 1}\right) + (1 + \epsilon^{-1}R(t)^\frac12)^{q - 1}\epsilon^2R(t) & \text{a.e. } t \in [0, \mathscr{T}_0\epsilon^{-2}]
\end{align*}
where we have used Lemma \ref{WavePacketCutoff} to estimate $T_4$, as well as lost one power of $\epsilon$ using Lemma \ref{DerivativeCountingInApprox} to estimate $\tilde{\mathcal{R}}$ in $L^2$.  For $0 \leq t \leq \mathcal{T}$, define $$S(\mathcal{T}) = \sup_{0 \leq t \leq \mathcal{T}} R(t)$$  Then we have in addition that for $0 \leq t \leq \mathcal{T} \leq \mathscr{T}_0\epsilon^{-2}$ that
\begin{align*}
\frac{dR}{dt}(t) & \leq C S(\mathcal{T})^\frac12 \left( \epsilon^6 + C\epsilon^7(1 + \epsilon^{-1}S(\mathcal{T})^\frac12)^{q - 1}\right) \\
& \quad + (1 + \epsilon^{-1}S(\mathcal{T})^\frac12)^{q - 1}\epsilon^2R(t) \hspace{1cm} \text{a.e. } t \in [0, \mathscr{T}_0\epsilon^{-2}]
\end{align*}
Integrating this equation and taking the supremum over all $t \in [0, \mathcal{T}]$ then yields the following bound for every $0 \leq t \leq \mathcal{T}$:
\begin{align*}
S(\mathcal{T}) & \leq \left(R(0) + CS(\mathcal{T})^\frac12\left( \epsilon^6 + C\epsilon^7(1 + \epsilon^{-1}S(\mathcal{T})^\frac12)^{q - 1}\right) t \right) \exp\left((1 + \epsilon^{-1}S(\mathcal{T})^\frac12)^{q - 1}\epsilon^2 t \right) \\
& \leq \left(C\epsilon^7 + CS(\mathcal{T})^\frac12\left( \epsilon^4 + C\epsilon^5(1 + \epsilon^{-1}S(\mathcal{T})^\frac12)^{q - 1}\right) \right) \exp\left(C(1 + \epsilon^{-1}S(\mathcal{T})^\frac12)^{q - 1} \right)
\end{align*}
where we have used $\mathcal{T} \leq \mathscr{T}_0\epsilon^{-2} \leq \mathscr{T}\epsilon^{-2}$ as well as Lemma \ref{WavePacketCutoff} to estimate that $R(0) \leq C\epsilon^7$.

We now begin a proof by continuity.  Let $\mathcal{T}_* = \inf\{\mathcal{T} \in [0, \mathscr{T}_0\epsilon^{-2}] : S(\mathcal{T}) > \epsilon^6\}$, which is well-defined since $S(T)$ is continuous.  If $\mathcal{T}_* = \mathscr{T}_0\epsilon^{-2}$, then we are done.  If not, then evaluating the above bound at $\mathcal{T} = \mathcal{T}_*$ would read
\begin{align*}
\epsilon^6 & = S(\mathcal{T}_*) \\ 
& \leq \left(C\epsilon^7 + CS(\mathcal{T}_*)^\frac12\epsilon^4 \right) \exp\left(C(1 + \epsilon^{-1}S(\mathcal{T}_*)^\frac12)^{q - 1} \right) \\
& \leq \left(C\epsilon^7 \right) \exp\left(C(1 + \epsilon^2)^\frac12)^{q - 1} \right) \\
& < \epsilon^6
\end{align*}
where we have chosen $\epsilon_0 > 0$ sufficiently small depending on $s, p, q, k, \mathscr{T}, \|A\|_{H^N}$.  This contradiction establishes control of the first two terms in the estimate of this proposition. Control of $\|\Lambda^s_\epsilon r_{tt}(t)\|$ then follows by applying $\Lambda^s_\epsilon$ to \eqref{RemainderEquation} and estimating directly in $L^2$.  The necessary estimates parallel those just performed in the above energy estimate, and so we omit the details.
\end{proof}

%==============================================================================
%==============================================================================
%==============================================================================
\section{Local Well-Posedness of the Progenitor Equation}
%==============================================================================
%==============================================================================
%==============================================================================

At this point we have computed an approximate solution $\tilde{z} = \epsilon Ae^{i\phi} + O(\epsilon^2)$ on the interval $[0, \mathscr{T}\epsilon^{-2}]$ satisfying the approximate progenitor equation \eqref{ApproximateEquation} and constructed through a solution $A$ of the HNLS equation \eqref{ActualHNLS} on the interval $[0, \mathscr{T}]$ with $A_0 \in C([0, \mathscr{T}] : H^N)$ and subject to Assumption \ref{AssumptionsOnA}.  We have also shown that if one \textit{assumes} an appropriate form of existence of a solution $z(t)$ to \eqref{zEquationOutline} on $[0, \mathscr{T}\epsilon^{-2}]$, then it remains within $O(\epsilon^3)$ of $\tilde{z}(t)$ in Sobolev space on the interval $[0, \mathscr{T}\epsilon^{-2}]$.  We now address the well-posedness of \eqref{zEquationOutline} to which $\tilde{z}$ serves as an approximation.

The reason that long-time well-posedness of \eqref{zEquationOutline} is not immediate is the appearance of $g(\lambda(t))$ and its derivatives.  If a time $t^*$ is reached at which $\lambda(t^*) = 1$, then one cannot construct a continuation of the solution beyond $t^*$ using any kind of contraction mapping arguments since $g^\prime(\lambda)$ is not defined for $\lambda \geq 1$.  Therefore, on any interval on which we would like this equation to make sense, we must show that $\lambda(t)$ stays uniformly bounded away from 1.  This problem is solved by the ODE argument made possible by the crucial choice $\mathcal{N}$ of coefficient in the penalization term together with the growth condition of Assumption 1.

However, even if one grants that $\lambda(t)$ remains uniformly bounded away from 1, the introduction of $\lambda(t)$ in the nonlinearity forces the equation to be fully nonlinear, which in turn prevents us from using a straightforward contraction mapping argument to establish local well-posedness.  Therefore we give a more roundabout construction that produces a weaker notion of solution\footnote{This notion of solution is reminiscent of a PDE analogue of ``solutions of extended type'' in the 
Carath\'eodory existence theorem for ODEs; see p. 42 as well as Theorem 1.1 in Chapter 2 of \cite{CoddingtonLevinsonODE}.} for \eqref{zEquationOutline}, but is nonetheless sufficient for our purposes.

Since, in general, we will need to construct a solution at some given time $\mathscr{T}_1\epsilon^{-2} \in [0, \mathscr{T}\epsilon^{-2}]$, we apply a translation in time $t \mapsto t - \mathscr{T}_1\epsilon^{-2}$ so that the initial conditions begin at $t = 0$.  The precise statement of the local well-posedness result is as follows:

\begin{proposition}\label{zLocalWellPosed}
(Local Well-Posedness of \eqref{zEquationOutline})  Fix $\nu \in (1, 2)$, and let $\mathscr{T}_1 \leq \mathscr{T}$ be given.  Denote
\begin{equation}
\lambda(f) = \frac{\|\Lambda^{s}_\epsilon f\|^2}{\nu^2 \omega^4 \|\Lambda^s A_0\|^2}
\end{equation}
Then there is an $\epsilon_0 > 0$ depending on $s, p, q, k, \|A_0\|_{H^s}, \nu$ so that for all $0 < \epsilon < \epsilon_0$ the following holds: 
\begin{itemize}
\item[(a)]{Assume the initial data $z(0)$ in \eqref{zEquationOutline}.  Then there exists a $\mathfrak{w}_0$ satisfying the compatibility condition 
\begin{equation}\label{IdealW0Definition}
\mathfrak{w}_0 = -|D|^p z(0) - \epsilon^{2}\mathcal{B}_k(z(0)|z(0)/\epsilon|^{q - 1}) - \mathcal{N}(0) g(\lambda(\mathfrak{w}_0))
\end{equation}
where
\begin{equation}
\mathcal{N}(0) := \mathcal{B}_k \omega^2 \bigr(\epsilon^{q + 4} z(0) + 2i(\epsilon^7 t + \epsilon^5(T_1 + \mathscr{T}_1))z(0)|z(0)|^{q - 1}\bigr)
\end{equation}
as well as the bound $\lambda(\mathfrak{w}_0) \leq \frac{1}{2} + \frac{1}{2\nu^2} < 1$.}
\item[(b)]{Assume that \eqref{zEquationOutline} has initial data $(u_0, v_0, w_0) = (\Lambda^s_\epsilon z(0), \Lambda^s_\epsilon z_t(0), \Lambda^s_\epsilon z_{tt}(0))$ which satisfies $u_0 = \mathcal{B}_k u_0$, $v_0 = \mathcal{B}_k v_0$, $w_0 = \mathcal{B}_k w_0$, the bounds
\begin{equation}
\max(\|u_0\|, \omega^{-1}\|v_0\|, \omega^{-2}\|w_0\|) \leq 2 \|A_0\|_{H^s},
\end{equation}
the bound
\begin{equation}
\|\Lambda^{-s}_\epsilon u_0\|_{L^\infty} \leq 2\epsilon\|A_0\|_{L^\infty},
\end{equation}
as well as satisfying the compatibility condition \eqref{IdealW0Definition}.  
Suppose further that $\lambda_0 := \lambda(w_0) < 1$ satisfies
\begin{equation}
g^\prime(\lambda_0) \leq \frac{16\epsilon^{-(q + 4)}}{\mathscr{C}}
\end{equation}
for some universal constant $\mathscr{C} > 1$ to be chosen in the course of the proof.  Then there is a $\mathbf{t} > 0$ for which there is a lower bound $t_0 < \mathbf{t}$ depending only on $\epsilon, k, s, p, q, 1 - \lambda_0$ so that
\begin{itemize}
\item[(i)]{There is a $z$ so that $\Lambda^s_\epsilon z_{tt}(x, y, t)$ is a bounded measurable function on $\mathbb{R}^2 \times [0, \mathbf{t}]$ with $\Lambda^s_\epsilon z_{tt} \in L^\infty([0, \mathbf{t}], L^2)$ which solves \eqref{zEquationOutline} for every $(x, y) \in \mathbb{R}^2$ and almost every $t \in [0, \mathbf{t}]$, and in particular solves the equation at $t = \mathbf{t}$.}
\item[(ii)]{For every $t \in [0, \mathbf{t}]$ we have $z(t) = \mathcal{B}_k z(t)$.}
\item[(iii)]{For every $t \in [0, \mathbf{t}]$ we have 
\begin{equation}\label{SolutionSize}
\max(\|\Lambda^s_\epsilon z(t)\|_{C([0, \mathfrak{t}], L^2)}, \omega^{-1}\|\Lambda^s_\epsilon z_t(t)\|_{C([0, \mathfrak{t}], L^2)}^2) \leq 3\|A_0\|_{H^s}
\end{equation}}
\item[(iv)]{For every $t \in [0, \mathbf{t}]$, we have the bound \begin{equation}\label{LocalLambdaBound}
|\lambda(z_{tt}(t)) - \lambda_0| \leq \frac12(1 - \lambda_0)
\end{equation}}
\end{itemize}
}
\end{itemize}
\end{proposition}

\begin{remark}\label{LambdaSomewhatControlled}
Although we have only shown that $z(t)$ is such that the equation \eqref{zEquationOutline} is satisfied for almost every $t \in [0, \mathbf{t}]$, it is still true that the bound \eqref{LocalLambdaBound} gives control of $\lambda(t)$ \textit{everywhere} in $[0, \mathbf{t}]$.  This bound holds in particular at $t = \mathbf{t}$ by construction, and hence the fact that $z(t)$ is a solution only almost everywhere poses no trouble in iterating this local well-posedness result.  Moreover this gives sufficient control over the growth of $\lambda$ to construct an analogue of a bootstrap argument using the ODE \eqref{CartoonODE} for $\lambda$.
\end{remark}

The proof of Proposition \ref{zLocalWellPosed} will occupy the remainder of this section and entails several steps which we outline here.  In essence, we will construct a sequence of approximate solutions formed by discretizing the full nonlinearity in time, and extract a solution by compactness.  In particular, we emphasize that throughout the construction, the growth condition \eqref{HsGrowthAssn} and the rescaling \eqref{Rescaling} are used crucially to control the terms contributed by the full nonlinearity in a contraction mapping argument.  In more detail, we will perform the following steps:

\vspace{0.5cm}

\begin{itemize}
\item[(\S 4.1)]{Fixing a small interval $[0, 2t_0]$, partition into $n$ equal subintervals $I_0, I_1 \ldots, I_{n - 1}$.  On each, construct a solution $z^{\{n\}}$ to an approximate version of \eqref{zEquationOutline} on $I_j$ by replacing the dangerous full nonlinearity $g(\lambda(z_{tt}))$ by its average value and passing to a time-differentiated version of \eqref{zEquationOutline}, assuming that this construction has been carried out on the subintervals $I_0, \ldots, I_{j - 1}$.  Since we expect $z_{tt}^{\{n\}}$ to exhibit jumps across the endpoints of the subintervals, we must simultaneously construct initial data satisfying the correct compatibility conditions as well.  Closing these arguments requires an appropriate choice of $\mathscr{C}$ and $\omega$.  This step subsumes the proof of part (a).}
\item[(\S 4.2)]{By concatenating the solutions inductively constructed in \S 4.1, construct a sequence of approximate solutions on the whole interval $[0, 2t_0]$ where the full nonlinearity has been approximated by a step function subordinated to a partition of $[0, 2t_0]$ into $n$ subintervals.  This solution will be discontinuous and is not \textit{prima facie} uniformly bounded in $n$.  Use the almost-conservative nature\footnote{Note that this requires a defocusing choice of sign on the power nonlinearity of the progenitor equation \eqref{zEquationOutline} and hence for the progeny equation \eqref{NormalizedNLS} as well.} of the approximate equations to show that the sequence of approximate solutions are uniformly bounded pointwise on a common space-time domain $\mathbb{R}^2 \times [0, 2t_0)$.}
\item[(\S 4.3)]{Use compactness to extract a subsequence that converges pointwise everywhere on $\mathbb{R}^2 \times [0, 2t_0)$ to a function that satisfies the original equation for every $(x, y) \in \mathbb{R}^2$ and almost every $[0, 2t_0)$.  One can then choose $\mathbf{t}$ to be any time between $t_0$ and $2t_0$ at which the equation is satisfied.}
\end{itemize}

\subsection{Constructing the solution on subintervals of $[0, 2t_0]$.}

In this subsection we fix notation to be more specific about the discretization of the full nonlinearity in the penalization term.  Let $n$ be given, and set $h := 2t_0/n$ and $t_j = jh$ for $j = 1, 2, \ldots, n$.  Partition the whole interval $[0, 2t_0]$ into subinterval of equal length $[0, t_1], [t_1, t_2], \ldots, [t_{n - 1}, t_n] = I_0, I_1, \ldots, I_{n - 1}$.  Given a function $f(t)$ defined on $I_j$, denote its average on $I_j$ by
\begin{equation}\label{SubintervalAverage}
m_n^{(j)} f(t) := \fint_{I_j} f(\tau) \, d\tau
\end{equation}
For $f(t)$ defined on $[0, 2t_0]$, define the piecewise constant approximation
\begin{equation}\label{PiecewiseConstantAverage}
m_n f(t) := \sum_{j = 0}^{n - 1} \mathbf{1}_{I_j} m_n^{(j)} f(t)
\end{equation}

Our local well-posedness result will require differentiating our evolution equation in time.  Since we are only given initial data for $z, z_t$, we show in (a) that we can supply initial data for $z_{tt}$ that is compatible with the evolution equation.

\begin{proof}[Part (a) of Proposition \ref{zLocalWellPosed}.]
Recall $z(0) = \mathcal{B}_k \tilde{z}(0)$.  Consider the mapping
\begin{equation}
\mathbf{F}(W) = -|D|\Lambda^s_\epsilon z(0) - \epsilon^{3 - q}\mathcal{B}_k \Lambda^s_\epsilon (z(0)|z(0)|^{q - 1}) - \mathcal{N}(0) g(\lambda(|D|^s_\epsilon \Lambda^{-s}_\epsilon W))
\end{equation}
and the associated space
\begin{equation}
\mathfrak{X} = \{ W \in L^2 : W = \mathcal{B}_k W, \, \|W - \Lambda^s_\epsilon \tilde{z}_{tt}(0)\| \leq \epsilon^\frac12 \}
\end{equation}
equipped with the $L^2$-norm.  The space $\mathfrak{X}$ is nonempty since $\mathcal{B}_k \Lambda^s_\epsilon \tilde{z}_{tt}(0) \in \mathfrak{X}$ by Lemma \ref{WavePacketCutoff}.  To show that $\mathbf{F}$ maps $\mathfrak{X}$ into itself, note first that $\mathcal{B}_k \mathbf{F}(W) = \mathbf{F}(W)$ immediately from the definition.  Moreover, given a $W \in \mathfrak{X}$, we have for $\epsilon_0 > 0$ chosen sufficiently small depending on $s, p, q, k, \|A_0\|_{H^N}$ that 
\begin{align*}
\|\mathbf{F}(W) - \Lambda^s_\epsilon \tilde{z}_{tt}(0)\| & \leq \|\Lambda^s_\epsilon\left(\tilde{z}_{tt}(0) + |D|\tilde{z}(0)\right)\| + C \epsilon^{3 - q}\|z(0)\|^{q - 1}_{L^\infty}\|\Lambda^s_\epsilon z(0)\| \\
& \quad  + C \epsilon^7 g\left( \lambda \left(|D|^s_\epsilon \tilde{z}_{tt}(0) + (|D|^s_\epsilon \Lambda^{-s}_\epsilon W - |D|^s_\epsilon \tilde{z}_{tt}(0))\right)\right) \\
& \leq C \epsilon + C\epsilon^2 + C \epsilon^{7} g\left(\frac{1}{2} + \frac{1}{2\nu}\right) \\
& \leq \epsilon^\frac12
\end{align*}
Thus $\mathbf{F}(W) \in \mathfrak{X}$ as well.  Finally, let $W_1, W_2 \in \mathfrak{X}$ be given, and suppose without loss of generality that $\lambda(W_1) \leq \lambda(W_2)$.  Then we have using \eqref{ControlNorm} et. seq. that 
\begin{align*}
\|\mathbf{F}(W_1) - \mathbf{F}(W_2)\| & \leq C \epsilon^{q + 4} \left( g(\lambda(W_1)) - g(\lambda(W_2))\, \right) \\
& \leq C \epsilon^7 \int_0^1 g^\prime\left((1 - \theta)\lambda(W_2) + \theta\lambda(W_1)\right)\left(\lambda(W_1) - \lambda(W_2)\right) \, d\theta \\
& \leq C  \epsilon^7 g^\prime(\lambda(W_2)) (\|W_1\| + \|W_2\|)\|W_1 - W_2\| \\
& \leq C \epsilon^7 g^\prime\left(\frac12 + \frac{1}{2\nu}\right) \|W_1 - W_2\| \\
& \leq \frac12 \|W_1 - W_2\|
\end{align*}
where we have taken $\epsilon_0 > 0$ to be possibly smaller still depending on $s, p, q, k, \|A_0\|_{H^N}, \nu$ .  Thus a fixed point $W$ exists, and we take $\mathfrak{w}_0 = \Lambda^{-s}_\epsilon W$.
\end{proof}

\begin{remark}
Since $\lambda(0)$ is bounded $O(1)$ away from $1$ for $\epsilon_0 > 0$ chosen sufficiently small, the difficulties associated with the full nonlinearity are avoided here.  Notice that we only need (a) at time $t = 0$; at all later times at which we would like to use Proposition \ref{zLocalWellPosed} to continue solutions to \eqref{zEquationOutline}, we will have constructed a solution that satisfies the compatibility conditions at its right hand endpoint so that the solution can be continued using (b) alone.
\end{remark}

Our goal is now to inductively construct our solution on each subinterval $I_j$ under the assumption that a similar solution has already been constructed on $I_{j - 1}$.  The proof proceeds by a standard contraction mapping argument, but with a twist:  since we approximate the full nonlinearity by a typically discontinuous function, we expect $z_{tt}$ to exhibit jump discontinuities across the endpoints of the subintervals $I_0, I_1, \ldots, I_{n - 1}$.  Thus as part of the contraction mapping argument we must construct initial data for $z_{tt}$ that satisfies the jump condition forced by the discontinuous nonlinearity.

Since it will be convenient at this stage to allow values of $\lambda \geq 1$ in the argument of $g$, we modify $g$ as follows: let $\lambda_* < 1$ be a quantity to be determined.  Then define
\begin{equation}\label{DefnGStar}
g_*^\prime(\lambda) = \begin{cases} g^\prime(\lambda) & \lambda \leq \lambda_* \\ 0 & \lambda > \lambda_* \end{cases} \qquad g_*(\lambda) = \int_0^\lambda g_*^\prime(\lambda^\prime) \, d\lambda^\prime
\end{equation}
Observe that with this definition, $g_*^\prime(\lambda) \leq \frac12(1 - \lambda_*)^{-\frac12}$ and $g_*(\lambda) \leq 1 - (1 - \lambda_*)^\frac12 < 1$.\footnote{We are justified in cutting off $g$ in this way since we will eventually show using the ODE argument alluded to in \eqref{CartoonODE} that for the solutions we construct, $\lambda(t)$ is bounded far enough away from 1 for all times so that $g_\star(\lambda(t)) = g(\lambda(t))$, provided $\lambda_*$ is chosen sufficiently close to 1.  This will be shown rigorously in \S 5.}  We proceed to give the ``induction step'' construction of solutions on the subintervals $I_j$:

\begin{lemma}\label{NanoLocalWellposedness}
(Construction of approximate solution on $I_j$ given a solution on $I_{j - 1}$.)  Let $n \in \mathbb{N}$ and $j = 1, 2, \ldots, n - 1$ be given.  Let $M \leq 2$ be given, let $\lambda_*$ satisfy
\begin{equation}
\epsilon^{q + 4} g^\prime(\lambda_*) = \frac{32}{\mathscr{C}}
\end{equation}
and define $g_\star$ as in \eqref{DefnGStar} above.  

Suppose that we are given initial data $(u_j, v_j, w_j)$ satisfying the properties
\begin{itemize}
\item[(i)]{$u_j = \mathcal{B}_k u_j$, $v_j = \mathcal{B}_k v_j$.}
\item[(ii)]{$\max(\|u_j\|, \omega^{-1}\|v_j\|) \leq M\|A_0\|_{H^s}$,}
\item[(iii)]{$\|\Lambda^{-s}_\epsilon u_j\|_{L^\infty} \leq 2\epsilon \|A_0\|_{L^\infty}$.}
\end{itemize}

Suppose further that we are given a function\footnote{Since we will also be denoting $\Lambda^s_\epsilon z_{tt}$ on $I_j$ by $w$ in the contraction mapping argument, this is an abuse of notation.  One should think of $w|_{I_{j - 1}}$ as having been already constructed with $w_j^-$ being the left-hand limit of the jump discontinuity at $t = t_j$.} $w \in C(I_{j - 1} : L^2)$ satisfying $w = \mathcal{B}_k w$ and $\|w\|_{C(I_{j - 1} : L^2)} \leq M\omega^2\|A_0\|_{H^s}$, and set $w_j^- := \lim_{t \to t_j^-} w(t)$.

Define 
$$
\lambda(t) = 
\begin{cases} 
\lambda(w(t)) & t \in I_{j - 1} \\ 
\lambda(\Lambda^s_\epsilon z_{tt}(t)) & t \in I_{j} 
\end{cases}
$$
Let $\mathscr{T}_1 < \mathscr{T}$ be given.  Then for $t_0 > 0$ chosen sufficiently small depending on $s, p, q, \epsilon, \|A_0\|_{H^s}$, there exists initial data $w_j^+$ satisfying the bound 
$$\|w_j^+\| \leq (M + 1)\|A_0\|_{H^s}$$
and a solution 
$$(z, z_t, z_{tt}) \in C^2(I_j : L^2) \times C^1(I_j : L^2) \times C^0(I_j : L^2)$$
to the initial value problem
\begin{equation}\label{ApproxEqnOnSubinterval}
\begin{cases}
z_{tt} + |D|^p z + \epsilon^2\mathcal{B}_k(z|z/\epsilon|^{q - 1}) +  \mathcal{N} g_\star(m_n^{(j)}\lambda(t)) = 0 \qquad (x, t) \in \mathbb{R}^2 \times I_{j} \\
\mathcal{N} := \mathcal{B}_k \left(\epsilon^{q + 4}z + 2i\omega^2(\epsilon^7 t + \epsilon^5(T_1 + \mathscr{T}_1))z|z|^{q - 1}\right) \\
z(t_j) = \Lambda^{-s}_\epsilon u_j \\
z_t(t_j) = \Lambda^{-s}_\epsilon v_j \\
\lim_{t \to t_j^+} z_{tt}(t) = \Lambda^{-s}_\epsilon w_j^+
\end{cases}
\end{equation}
which satisfies
$$\max_{t \in I_{j + 1}} \left( \|\Lambda^s_\epsilon z(t)\|, \omega^{-1} \|\Lambda^s_\epsilon z_t(t)\| \right) \leq \left(M + C\frac{t_0}{n}\right)\|A_0\|_{{H}^s}$$
and
$$\|\Lambda^s_\epsilon z_{tt}(t)\| \leq \omega^2\left(M + 1 + C \frac{t_0}{n}\right)\|A_0\|_{H^s}$$
for a constant $C$ depending on $s, k, p, q, \epsilon, \|A_0\|_{H^s}$
\end{lemma}

\begin{proof}
We begin by applying the operator $\Lambda^{s}_\epsilon$ and formally differentiating \eqref{ApproxEqnOnSubinterval} in time.  Introduce $\mathcal{W}(t) = (v(t), w(t))^T = (\Lambda^s_\epsilon z_t(t), \Lambda^s_\epsilon z_{tt}(t))^T$ as well as $u(t) = u_j + \int_{t_j}^t v(\tau) \, d\tau$.  Writing the resulting equation in a DuHamel formulation motivates studying the map
\begin{equation}\label{JustDuhamel}
\mathcal{W} \mapsto e^{\mathscr{L} (t - t_j)}\begin{pmatrix} v_j \\ w_j^+ \end{pmatrix} + \int_{t_j}^t e^{\mathscr{L}(t - \tau)} \begin{pmatrix} 0 \\ \mathscr{N}(\tau) \end{pmatrix} \, d\tau
\end{equation}
where we introduce the semigroup
\begin{equation}\label{LSemigroup}
e^{\mathscr{L}t} = \begin{pmatrix} \cos(|D|^\frac{p}{2}t) & \frac{\sin(|D|^\frac{p}{2} t)}{|D|^\frac{p}{2}} \\ -|D|^\frac{p}{2} \sin(|D|^\frac{p}{2} t) & \cos(|D|^\frac{p}{2}t) \end{pmatrix}
\end{equation}
associated to the operator 
\begin{equation}
\mathscr{L} = \begin{pmatrix} 0 & 1 \\ -|D|^p & 0 \end{pmatrix}
\end{equation}
as well as the nonlinearity
\begin{align}\label{DifferentiatedNonlinearity}
\mathscr{N}(t) & = \mathcal{B}_k \epsilon^{3 - q} \Lambda^s_\epsilon \Bigl(C_q \Lambda^{-s}_\epsilon v |\Lambda^{-s}_\epsilon u|^{q - 1} + C_q (\Lambda^{-s}_\epsilon u)^2 \Lambda^{-s}_\epsilon \overline{v} |\Lambda^{-s}_\epsilon u|^{q - 3} \Bigr) \notag \\
& \qquad + \mathcal{B}_k \Bigl( \epsilon^{q +4} v + 2i\omega^2\epsilon^7 \Lambda^{-s}_\epsilon u |\Lambda^{-s}_\epsilon u|^{q - 1})g_\star(m^{(j)}\lambda(t)) \notag \\
& \qquad + \mathcal{B}_k \Lambda^s_\epsilon \Bigl(C_q \Lambda^{-s}_\epsilon v |\Lambda^{-s}_\epsilon u|^{q - 1} + C_q (\Lambda^{-s}_\epsilon u)^2 \Lambda^{-s}_\epsilon \overline{v} |\Lambda^{-s}_\epsilon u|^{q - 3} \Bigr) \\
& \hspace{4cm} \times 2i\omega^2(\epsilon^7t + \epsilon^5(T_1 + \mathscr{T}_1))g_\star(m^{(j)}\lambda(t)) \notag
\end{align}
where in the above different instances of the constants $C_q$, $C_\omega$ may differ in value.  Next, formally subtracting the equations \eqref{ApproxEqnOnSubinterval} defined on $I_{j - 1}$ and $I_{j}$ and evaluating the jump discontinuity $w_j^+ - w_j^- = \lim_{t \to t_j^+} z_{tt}(t) - \lim_{t \to t_j^-} z_{tt}(t)$ suggests studying the following mapping

\begin{equation}\label{JustJump}
w_j^+ \mapsto w_j^- - \mathcal{B}_k\left(\epsilon^{q + 4}\omega^2 u_j + 2i\omega^2\epsilon^5(T_1 + \mathscr{T}_1)\Lambda^s_\epsilon(\Lambda^{-s}_\epsilon u_j |\Lambda^{-s}_\epsilon u_j|^{q - 1}|)\right)\left(g_\star(m^{(j)}_n \lambda) - g_\star(m^{(j - 1)}_n \lambda)\right)
\end{equation}

Since neither \eqref{JustDuhamel} nor \eqref{JustJump} taken alone would produce a closed contraction mapping, we combine them into one mapping $\mathbf{F}$ on the space $C(I_{j}: L^2) \times C(I_{j}: L^2) \times L^2$.  Denote $\mathcal{W} = (v(t), w(t), w_j^+)^T$ with $u(t) := u_j + \int_{t_j}^t v(\tau) \, d\tau$.  Then define

\begin{equation}\label{ContractionMapping}
\mathbf{F}(\mathcal{W})(t) := 
\begin{pmatrix} 
e^{\mathscr{L} (t - t_j)}\begin{pmatrix} v_j \\ w_j^+ \end{pmatrix} + \int_{t_j}^t e^{\mathscr{L}(t - \tau)} \begin{pmatrix} 0 \\ \mathscr{N}(\tau) \end{pmatrix} \, d\tau \\ 
w_j^- - \mathcal{B}_k\Bigl(\epsilon^{q + 4}\omega^2u_j + 2i\omega^2\epsilon^5(T_1 + \mathscr{T}_1) \\
\times \Lambda^s_\epsilon(\Lambda^{-s}_\epsilon u_j |\Lambda^{-s}_\epsilon u_j|^{q - 1})\Bigr)\left(g_\star(m^{(j)}_n \lambda) - g_\star(m^{(j - 1)}_n \lambda)\right)
\end{pmatrix}
\end{equation}

We proceed by showing that $\mathbf{F}$ is a contraction mapping on the space
\begin{align}
\mathfrak{X} = \Bigl\{\mathcal{W} \in C(I_{j}: L^2) \times C(I_{j}: L^2) \times L^2 : &\;v(t_j) = v_j \notag \\
& \lim_{t \to t_j^+} w(t) = w_j^+ \notag \\
& \mathcal{W} = \mathcal{B}_k\mathcal{W} \\
& \|w_j^+\|_{L^2} \leq \omega^2(M + 1)\|A_0\|_{H^s} \notag \\
& \|w\|_{C(I_{j}: L^2)} \leq \omega^2\left(M + 1 + C\frac{t_0}{n}\right)\|A_0\|_{H^s} \notag \\
& \|v\|_{C(I_{j}: L^2)} \leq \omega \left(M + C\frac{t_0}{n}\right)\|A_0\|_{H^s} \Bigr\} \notag
\end{align}
equipped with the norm\footnote{The extra factor of $4$ in the definition of $\|\cdot\|_\mathfrak{X}$ is used to ensure that $\mathbf{F}$ is contractive; see the estimate of $\|\mathbf{F}(\mathcal{W}_1)_{2} - \mathbf{F}(\mathcal{W}_2)_{2}\|$ below.}
\begin{equation}
\|(f(t), g(t), h)^T\|_\mathfrak{X} := \|f\|_{C(I_{j} : L^2)} + \|g\|_{C(I_{j} : L^2)} + 4\|h\|_{L^2}
\end{equation}

First we show that $\mathbf{F}(\mathcal{W}) \in \mathfrak{X}$ whenever $\mathcal{W} \in \mathfrak{X}$.  Let $\mathcal{W} \in \mathfrak{X}$ be given, and denote by $\mathbf{F}(\mathcal{W})_l$ the $l$th component of $\mathbf{F}(\mathcal{W})$.  Then 
$$\mathbf{F}(\mathcal{W})_1(t_j) = v_j, \qquad \lim_{t \to t_j^+} \mathbf{F}(\mathcal{W})_2(t) = w_j^+, \qquad \mathbf{F}(\mathcal{W}) = \mathcal{B}_k \mathbf{F}(\mathcal{W})$$ all follow immediately from the definition of $\mathbf{F}$.  It is helpful to observe that, when $\mathcal{W} \in \mathfrak{X}$, we have by definition of $u(t)$ the estimates
\begin{equation}
\|u(t)\| \leq \|u_j\| + \left\| \int_{t_j}^t v(\tau) \, d\tau \right\| \leq M\|A_0\|_{H^s} + C(p, q, k, \|A_0\|_{H^s})h
\end{equation}
as well as
\begin{equation}\label{uLInfty}
\|\Lambda^{-s}_\epsilon u(t)\|_{L^\infty} \leq \|u_j\|_{L^\infty} + \int_{t_j}^t \|v(\tau)\|_{L^\infty} \, d\tau \leq 2\epsilon \|A_0\|_{H^s} + C(p, q, k, \|A_0\|_{H^s})h
\end{equation}

\vspace{.5cm}

\textbf{Estimate of $\|F(\mathcal{W})_3\|$.}  We cannot close this estimate by the usual perturbative method of restricting to very small time intervals; we must use $\mathscr{C}$ and $\omega$.  Set $$\mathcal{N}_j := - \mathcal{B}_k\left(\epsilon^{q + 4}\omega^2 u_j + 2i\omega^2\epsilon^5(T_1 + \mathscr{T}_1)\Lambda^s_\epsilon(\Lambda^{-s}_\epsilon u_j |\Lambda^{-s}_\epsilon u_j|^{q - 1}|)\right).$$ We have using Lemma \ref{RescaledMoserInequality} and the Fundamental Theorem of Calculus that
\begin{align*}
\|F(\mathcal{W})_3\| & \leq \|w_j^-\| + \|\mathcal{N}_j\| \left| \int_0^1 g_\star^\prime \left(\theta m_n^{(j)}\lambda + (1 - \theta) m_n^{(j - 1)} \lambda \right) d\theta \right| \, |m_n^{(j)} \lambda - m_n^{(j - 1)} \lambda| \\
& \leq \omega^2 M\|A_0\|_{H^s} + \|\mathcal{N}_j\| \frac{32\epsilon^{-(q + 4)}}{\mathscr{C}} |m_n^{(j)} \lambda - m_n^{(j - 1)} \lambda| \\
& \leq \omega^2 M\|A_0\|_{H^s} + \|\mathcal{N}_j\| \frac{32\epsilon^{-(q + 4)}}{\mathscr{C}} \left(\frac{(M + 2)^2\|A_0\|_{H^s}^2}{\nu\|A_0\|_{H^s}^2} + \frac{M^2\|A_0\|_{H^s}^2}{\nu\|A_0\|_{H^s}^2}\right) \\
& \leq \omega^2 M\|A_0\|_{H^s} + \|\mathcal{N}_j\| \frac{1600\epsilon^{-(q + 4)}}{\mathscr{C}}
\end{align*}
We can also estimate the term $\|\mathcal{N}_j\|$ using Lemma \ref{RescaledMoserInequality}, the rescaling \eqref{Rescaling} as in Remark \ref{WhyRescaleWhyThisBound}, and bound (iii) along with \eqref{uLInfty} as follows:
\begin{align}\label{N1Bound}
\|\mathcal{N}_j\| & = \| \epsilon^{q + 4}\omega^2 u_j + 2i\omega^2\epsilon^5(T_1 + \mathscr{T}_1)\Lambda^s_\epsilon(\Lambda^{-s}_\epsilon u_j |\Lambda^{-s}_\epsilon u_j|^{q - 1}|) \| \notag \\
& \leq \epsilon^{q + 4}\omega^2 M\|A_0\|_{H^s} + 2C_{(s, q)}3^q\epsilon^{q + 4}\omega^2 \left(2 \omega^{\frac{4}{p} - 1} \mathbf{T}\right)\|\mathfrak{u}\|_{C([0, \mathbf{T}] : L^\infty)}^{q - 1} M\|A_0\|_{H^s}
\end{align}
where $C_{(s, q)}$ depends on the sharp constants in the estimate of Lemma \ref{RescaledMoserInequality} and in the Sobolev embedding $H^2(\mathbb{R}^2) \hookrightarrow L^\infty(\mathbb{R}^2)$.  Combining these estimates now gives
\begin{align*}
\|F(\mathcal{W})_3\| & \leq \omega^2 M\|A_0\|_{H^s} + \left(3 + C_{s, q} \left(2 \omega^{\frac{4}{p} - 1} \mathbf{T}\right)\|\mathfrak{u}\|_{C([0, \mathbf{T}] : L^\infty)}^{q - 1} \right)\frac{1600}{\mathscr{C}}\omega^2\|A_0\|_{H^s} \\
& \leq \omega^2 M\|A_0\|_{H^s} + \frac{6400}{\mathscr{C}}\omega^2\|A_0\|_{H^s} \\
& \leq \omega^2(M + 1)\|A_0\|_{H^s}
\end{align*}
where in the second step we have chosen $\omega$ small enough so that 
$C_{s, q} \omega^{\frac{4}{p} - 1} \mathbf{T}\|\mathfrak{u}\|_{C([0, \mathbf{T}] : L^\infty)}^{q - 1} = 1$
and in the third step that $\mathscr{C} = 10^5$.

\vspace{.5cm}

\textbf{Estimate of $\|F(\mathcal{W})_2\|_{C(I_{j + 1} : L^2)}$.}   This estimate is purely perturbative thanks to the DuHamel  integral and the symbol expansions of the multipliers $\cos(|D|^{\frac{p}{2}})$ and $|D|^{-\frac{p}{2}}\sin(|D|^{\frac{p}{2}})$ about the wave number $\xi = k$.  Since we have $|g_*(\lambda)| \leq 1$ for any $\lambda \geq 0$, we estimate using Lemma \ref{RescaledMoserInequality} that
\begin{align*}
\|F(\mathcal{W})_2\|_{C(I_{j + 1} : L^2)^2} & \leq \|w_j^+\| + C_k h (\|v_j\| + \|w_j^+\|) + h\|\mathscr{N}\|_{C(I_{j + 1} : L^2)} \\
& \leq (M + 1)\|A_0\|_{H^s} + C_k h (\|v_j\| + \|w_j^+\|) + h\|\mathscr{N}\|_{C(I_{j + 1} : L^2)} \\
& \leq \left(M + 1 + C_{(s, q, p, k, \|A_0\|_{H^s})}h\right)\|A_0\|_{H^s}
\end{align*}

\textbf{Estimate of $\|F(\mathcal{W})_1\|_{C(I_{j + 1} : L^2)}$.}  This estimate parallels the estimate of $\|F(\mathcal{W})_2\|_{C(I_{j + 1} : L^2)}$, except that it is easier due to an extra factor of $h$ in the DuHamel integral thanks to the symbol expansions of $\cos(|D|^{\frac{p}{2}})$ and $|D|^{\frac{p}{2}}\sin(|D|^{\frac{p}{2}})$; we omit estimating in detail.

\vspace{.5cm}

We now show that $\mathbf{F}$ is a contraction mapping.  Let $\mathcal{W}_1, \mathcal{W}_2 \in \mathfrak{X}$ be given.  We split the estimates by component.

\vspace{.5cm}

\textbf{Estimate of $\|\mathbf{F}(\mathcal{W}_1)_3 - \mathbf{F}(\mathcal{W}_2)_3\|$.}  We begin again with the non-perturbative estimate; the steps here parallel those of the estimate of $\|F(\mathcal{W})_3\|$ above.  With $\mathcal{N}_j$ as above, we have
\begin{align*}
\mathbf{F}(\mathcal{W}_1)_3 - \mathbf{F}(\mathcal{W}_2)_3 & = \mathcal{N}_j \left(g_*(m^{(j)}_n \lambda_1) - g_\star(m^{(j)}_n \lambda_2)\right)
\end{align*}
For brevity, set $\lambda_j := \lambda(w_j)$.  Using the definition of $g_\star$ yields the bound
\begin{align*}
& \quad \|\mathbf{F}(\mathcal{W}_1)_3 - \mathbf{F}(\mathcal{W}_2)_3\| \\
& \leq \|\mathcal{N}_j\| \left| \int_0^1 g_*^\prime \left(m^{(j)}_n (\theta \lambda_1 + (1 - \theta) \lambda_2 )\right) \, d\theta \right| \cdot |m^{(j)}_n(\lambda_1 - \lambda_2)| \\
& \leq \|\mathcal{N}_j\| \frac{32\epsilon^{-(q + 4)}}{\mathscr{C}} \frac{(\|w_1\| + \|w_2\|)\|w_1 - w_2\|}{\nu^2\omega^4 \|A_0\|_{H^s}^2} \\
& \leq \|\mathcal{N}_j\| \frac{32\epsilon^{-(q + 4)}}{\mathscr{C}} \frac{8\|A_0\|_{H^s}\|w_1 - w_2\|}{\omega^2\|A_0\|_{H^s}^2} \\
& \leq \|\mathcal{N}_j\|\frac{1024\epsilon^{-(q + 4)}}{\mathscr{C}} \frac{\|w_1 - w_2\|}{\omega^2\|A_0\|_{H^s}}
\end{align*}
With the bound \eqref{N1Bound} of $\|\mathcal{N}_j\|$ and the same choices of $\omega$ and $\mathscr{C}$ as above, we find that
\begin{align*}
& \quad \|\mathbf{F}(\mathcal{W}_1)_3 - \mathbf{F}(\mathcal{W}_2)_3\| \\
& \leq \left(\omega^2 M\|A_0\|_{H^s} + 6C_{s, q}\omega^2 \left(2 \omega^{\frac{4}{p} - 1} \mathbf{T}\right)\|\mathfrak{u}\|_{C([0, \mathbf{T}] : L^\infty)}^{q - 1} \|A_0\|_{H^s}\right)\frac{1024}{\mathscr{C}} \frac{\|w_1 - w_2\|}{\omega^2\|A_0\|_{H^s}} \\
& \leq \left(3 + C_{s, q}\left( \omega^{\frac{4}{p} - 1} \mathbf{T}\right)\|\mathfrak{u}\|_{C([0, \mathbf{T}] : L^\infty)}^{q - 1} \right)\frac{1024}{\mathscr{C}} \|w_1 - w_2\| \\
& \leq \frac{4096}{\mathscr{C}} \|w_1 - w_2\| \\
& \leq \frac{1}{16} \|\mathcal{W}_1 - \mathcal{W}_2\|_\mathfrak{X}
\end{align*}

\vspace{.5cm}

\textbf{Estimate of $\|\mathbf{F}(\mathcal{W}_1)_{2} - \mathbf{F}(\mathcal{W}_2)_{2}\|$.}  This estimate is perturbative, but also crucially uses the extra factor of 4 in the definition of $\|\cdot\|_\mathfrak{X}$.  We first record that
\begin{align*}
\mathbf{F}(\mathcal{W}_1)_{1, 2} - \mathbf{F}(\mathcal{W}_2)_{1, 2} & = e^{\mathscr{L} (t - t_j)}\begin{pmatrix} 0 \\ (w_1)_j^+ - (w_2)_j^+ \end{pmatrix} + \int_{t_j}^t e^{\mathscr{L}(t - \tau)} \begin{pmatrix} 0 \\ \mathscr{N}_1(\tau) - \mathscr{N}_2(\tau) \end{pmatrix} \, d\tau
\end{align*}
We then estimate that
\begin{align*}
& \quad \|\mathbf{F}(\mathcal{W}_1)_{2} - \mathbf{F}(\mathcal{W}_2)_{2}\|_{C(I_{j + 1} : L^2)} \\
& \leq \|(w_1)_j^+ - (w_2)_j^+\| + C_k h(\|(w_1)_j^+\| + \|(w_2)_j^+\| + \|v_j\|) + \left\|\mathscr{N}_1(\tau) - \mathscr{N}_2(\tau) \right\| h \\
& \leq \frac{1}{4}\|\mathcal{W}_1 - \mathcal{W}_2\|_\mathfrak{X} + C_k h(\|(w_1)_j^+\| + \|(w_2)_j^+\| + \|v_j\|)+ \left\|\mathscr{N}_1(\tau) - \mathscr{N}_2(\tau) \right\| h
\end{align*}
Now expand $\mathscr{N}_1 - \mathscr{N}_2$ into terms having differences of the form either $u_1 - u_2$, $v_1 - v_2$ or $g_*(m^{(j)}_n \lambda_1) - g_*(m^{(j)}_n \lambda_2)$.  In terms where the differences are of the first type, we may estimate $|g_*(m^{(j)}_n \lambda_l)| \leq 1$ for $l = 1, 2$ and bound these terms by 
\begin{align*}
& \quad\; C_{(s, q, k, \epsilon)}h\|u_1 - u_2\|_{C(I_{j} : L^2)} \\
& \leq C_{(s, q, k, \epsilon)}h\left\| \int_{t_j}^{t_{j + 1}} v_1(\tau) - v_2(\tau) \, d\tau\right\|_{C(I_{j} : L^2)} \\
& \leq C_{(s, q, k, \epsilon)}h^2\|\mathcal{W}_1 - \mathcal{W}_2\|_\mathfrak{X}
\end{align*}
For the second type of term the estimate is similar, but with one less factor of $h$.  In the final case, the term in question is bounded by
\begin{align*}
& \quad C_{(s, q, k, \epsilon)}h|m^{(j)}_n(\lambda_1 - \lambda_2)| \\
& \leq C_{(s, q, k, \epsilon)}h\frac{(\|w_1\| + \|w_2\|)(\|w_1 - w_2\|)}{\nu^2 \omega^4 \|A_0\|_{H^s}^2} \\
& \leq C_{(s, q, k, \epsilon)}h\frac{8\|A_0\|_{H^s}(\|w_1 - w_2\|)}{ \|A_0\|_{H^s}^2} \\
& \leq C_{(s, q, k, \epsilon, \|A_0\|_{H^s})}h\|w_1 - w_2\| \\
& \leq C_{(s, q, k, \epsilon, \|A_0\|_{H^s})}h\|\mathcal{W}_1 - \mathcal{W}_2\|_\mathfrak{X} 
\end{align*}
Summing over all such terms then gives the bound
\begin{equation}
\|\mathbf{F}(\mathcal{W}_1)_{2} - \mathbf{F}(\mathcal{W}_2)_{2}\| \leq \left(\frac{1}{4} + C_{(s, q, k, \epsilon, \|A_0\|_{H^s})}h\right)\|\mathcal{W}_1 - \mathcal{W}_2\|_\mathfrak{X}
\end{equation}

\vspace{0.5cm}

\textbf{Estimate of $\|\mathbf{F}(\mathcal{W}_1)_{1} - \mathbf{F}(\mathcal{W}_2)_{1}\|$.}  This estimate proceeds just as in the estimate of $\|\mathbf{F}(\mathcal{W}_1)_{2} - \mathbf{F}(\mathcal{W}_2)_{2}\|$, except that we have an extra factor of $h$, and the initial data of the difference exactly cancels in this component.  The estimate, very similar to the previous one whose details we therefore omit, concludes
\begin{equation}
\|\mathbf{F}(\mathcal{W}_1)_{1} - \mathbf{F}(\mathcal{W}_2)_{1}\| \leq C_{(s, q, k, \epsilon, \|A_0\|_{H^s})}h\|\mathcal{W}_1 - \mathcal{W}_2\|_\mathfrak{X}
\end{equation}

\vspace{0.5cm}

\textbf{The Conclusive Contraction Estimate.}  Summing the previous estimates and choosing $t_0 > 0$ sufficiently small depending on $s, q, p, \epsilon, \|A_0\|_{H^s}$ then gives
\begin{align*}
\|\mathbf{F}(\mathcal{W}_1) - \mathbf{F}(\mathcal{W}_2)\|_\mathfrak{X} & = \|\mathbf{F}(\mathcal{W}_1)_1 - \mathbf{F}(\mathcal{W}_2)_1\| \\
& \quad + \|\mathbf{F}(\mathcal{W}_1)_2 - \mathbf{F}(\mathcal{W}_2)_2\| \\
& \quad + 4\|\mathbf{F}(\mathcal{W}_1)_3 - \mathbf{F}(\mathcal{W}_2)_3\| \\
& \leq \left(C_{(s, q, k, \epsilon, \|A_0\|_{H^s})}h + \frac{1}{4} + \frac{4}{16}\right)\|\mathcal{W}_1 - \mathcal{W}_2\|_\mathfrak{X} \\
& \leq \frac34 \|\mathcal{W}_1 - \mathcal{W}_2\|_\mathfrak{X}
\end{align*}
as desired.
\end{proof}

\begin{remark}\label{BringDownHugeConstant}
The constant $\mathscr{C}$ in the above proof is universal, and is chosen to be much larger than necessary.  By introducing more parameters, one can refine the proof of Lemma \ref{NanoLocalWellposedness} by replacing various factors of 2 and 3 by factors arbitrarily close to 1.  In so doing, we can show that in fact $\mathscr{C}$ can be brought to within an order of magnitude of 1.  In principle one could therefore trace through this argument and specify more precisely that the degree of the polynomial bound in Theorem \ref{BigSummaryTheorem} is also within an order of magnitude of 1.  Therefore this method can give an \textit{effective} bound on the polynomial growth of $\|A\|_{H^s}$.
\end{remark}

An easy modification of the above proof also proves the following ``first step'' of the construction for local well-posedness on $I_0$.

\begin{lemma}\label{FirstStepLWP}
(Construction of approximate solution on $I_0$.)  Let $n \in \mathbb{N}$ and $j = 1, 2, \ldots, n - 1$ be given.  Let $M \leq 2$ be given, let $\lambda_*$ satisfy
\begin{equation}
\epsilon^{q + 4} g^\prime(\lambda_*) = \frac{32}{\mathscr{C}}
\end{equation}
and define $g_\star$ as in \eqref{DefnGStar} above.

Suppose further that we are given initial data $(u_0, v_0, w_0)$ satisfying the properties
\begin{itemize}
\item[(i)]{$u_0 = \mathcal{B}_k u_0$, $v_0 = \mathcal{B}_k v_0$, $w_0 = \mathcal{B}_k w_0$,}
\item[(ii)]{$\max(\|u_0\|, \omega^{-1}\|v_0\|, \omega^{-2}\|w_0\|) \leq M\|A_0\|_{H^s}$}
\item[(iii)]{$\|\Lambda^{-s}_\epsilon u_j\|_{L^\infty} \leq 2\epsilon \|A_0\|$.}
\end{itemize}

Define $$\lambda(t) = \frac{\|\,\Lambda^s_\epsilon z_{tt}(t)\|^2}{\nu^2 \omega^4 \|A_0\|_{H^s}^2}$$  Then for any $\mathscr{T}_1 > 0$ for which $T_1 + \mathscr{T}_1 \leq \mathscr{T}$, there is a $t_0 > 0$ chosen sufficiently small depending on $s, p, q, k, \epsilon, \|A_0\|_{H^s}, \epsilon$, there exists a solution 
$$(z, z_t, z_{tt}) \in (C^2 \times C^1\times C^0)(I_0 : L^2)$$
to the initial value problem
\begin{equation}\label{ApproxEqnOnSubinterval}
\begin{cases}
z_{tt} + |D|^p z + \epsilon^2\mathcal{B}_k(z|z/\epsilon|^{q - 1}) + \mathcal{N} g_\star(m_n^{(0)}\lambda(t)) = 0 \qquad (x, t) \in \mathbb{R}^2 \times I_{0} \\
\mathcal{N} := \mathcal{B}_k \left(\epsilon^{q + 4}\omega^2z + 2\omega^2i (\epsilon^7 t + \epsilon^5(T_1 + \mathscr{T}_1))z|z|^{q - 1}\right) \\
z(0) = \Lambda^{-s}_\epsilon u_0 \\
z_t(0) = \Lambda^{-s}_\epsilon v_0 \\
z_{tt}(0) = \Lambda^{-s}_\epsilon w_0
\end{cases}
\end{equation}
which satisfies
$$\max_{t \in I_{1}} \left( \|\Lambda^s_\epsilon z(t)\|, \omega^{-1}\|\Lambda^s_\epsilon z_t(t)\|, \omega^{-2}\|\Lambda^s_\epsilon z_{tt}(t)\|  \right) \leq \left(M + C\frac{t_0}{n}\right)\|A_0\|_{{H}^s}$$
for a constant $C$ depending on $s, p, q, k, \epsilon, \|A_0\|_{H^s}$
\end{lemma}

\begin{proof}
The proof reproduces the contraction mapping argument of Lemma \ref{NanoLocalWellposedness} but omitting the third component of the contraction mapping \eqref{ContractionMapping}, since the initial data for $z_{tt}$ in Lemma \ref{FirstStepLWP} is the $w_0$ provided by hypothesis rather than involving a right hand jump value.  Since the argument is a subset of the proof of Lemma \ref{NanoLocalWellposedness} and is purely perturbative, it is strictly easier in this case, and so we omit the details.
\end{proof}

\subsection{Constructing the Approximate Solution on all of $[0, 2t_0]$.}

Our goal in this section is to use Lemma \ref{NanoLocalWellposedness} successively on the intervals $[0, t_1]$, $[t_1, t_2]$, $\ldots$, $[t_{n - 1}, t_n]$ to construct an approximate solution $z^{\{n\}}$ to the initial value problem
\begin{equation}\label{ApproxEqnOnSubinterval}
\begin{cases}
z_{tt} + |D|^p z + \epsilon^2\mathcal{B}_k(z|z/\epsilon|^{q - 1}) + \mathcal{N} g_\star(m_n \lambda(t)) = 0 \qquad (x, t) \in \mathbb{R}^2 \times [0, 2t_0] \\
\mathcal{N} := \epsilon^{5} \mathcal{B}_k \left(\epsilon^{q + 4} \omega^2 z + 2i\omega^2(\epsilon^7t + \epsilon^5 (T_1 + \mathscr{T}_1)) z|z|^{q - 1}\right) \\
z(t_j) = \Lambda^{-s}_\epsilon u_0 \\
z_t(t_j) = \Lambda^{-s}_\epsilon v_0 \\
\end{cases}
\end{equation}

If one directly applies Lemmas \ref{FirstStepLWP} and \ref{NanoLocalWellposedness}, one cannot in general construct such a solution for all $n$.  Suppose that we know that the norm of $w$ is initially of the form $M_0\omega^2\|A_0\|_{H^s}$ for some $M_0 < 2$.  Since $z_{tt}$ on the $j$th interval can be bounded at best by $(M_0 + j)\omega^2\|A_0\|_{H^s}$, and so the condition $M \leq 2$ must be violated for sufficiently large $n$.  The condition $M \leq 2$ itself must not be neglected: if we allowed $M$ to be arbitrarily large, then inspecting the proof of Lemma \ref{NanoLocalWellposedness} shows that the crucial growth constant $\mathscr{C}$ must also be arbitrarily large if one hopes to construct a sequence of approximate solutions for all $n$.  To remedy this, we appeal to the almost conserved Hamiltonian to provide the necessary uniform bounds.

\begin{lemma} (Uniform bounds from almost conservation.) \label{AlmostConserved}
Let $\mathscr{T}_0 \leq \mathscr{T}$ be given, and suppose $\mathscr{P} : [0, \mathscr{T}_0\epsilon^{-2}] \to \mathbb{R}$ satisfies $|\mathscr{P}(t)| \leq 1$ for all $t \in [0, \mathscr{T}_0\epsilon^{-2}]$.  Suppose there exists a function $\mathfrak{z}(t)$ satisfying $\mathfrak{z}, \mathfrak{z}_t, \in C([0, \mathscr{T}_0\epsilon^{-2}], L^2)$ and $\mathfrak{z}_{tt} \in L^\infty([0, \mathscr{T}_0\epsilon^{-2}], L^2)$ and which solves almost everywhere in $[0, \mathscr{T}_0\epsilon^{-2}]$ the initial value problem
\begin{equation}\label{GeneralParentEquation}
\begin{cases}
\mathfrak{z}_{tt} + |D|^p \mathfrak{z} + \epsilon^2\mathcal{B}_k(\mathfrak{z}|\epsilon^{-1}\mathfrak{z}|^{q - 1}) + \mathcal{B}_k( \epsilon^{q + 4} \mathfrak{z} + 2i(\epsilon^7t + \epsilon^5T_1)\mathfrak{z}|\mathfrak{z}|^{q - 1} ) \mathscr{P}(t) = 0 \\
\mathfrak{z}(0) = \mathfrak{u}_0 \\
\mathfrak{z}_t(0) = \mathfrak{v}_0
\end{cases}
\end{equation}
for which $\mathfrak{u}_0 = \mathcal{B}_k \mathfrak{u}_0$ and $\mathfrak{v}_0 = \mathcal{B}_k \mathfrak{v}_0$.  Suppose further that $\mathfrak{E}(0)$ is uniformly bounded away from zero in $\epsilon$ over all $0 < \epsilon < \epsilon_0$.  Then 
\begin{itemize}
\item[\textbf{(a)}]{The solution satisfies $\mathfrak{z} = \mathcal{B}_k \mathfrak{z}$, and the Hamiltonian energy
\begin{equation}
\mathfrak{E}(t) := \frac12\|\mathfrak{z}_t(t)\|^2 + \frac12\|\,|D|^\frac{p}{2} \mathfrak{z}(t)\|^2 + \frac{\epsilon^{3 - q}}{q + 1}\left\|\mathfrak{z}\right\|_{L^{q + 1}}^{q + 1}
\end{equation}
satisfies the a priori bound
\begin{equation}\label{L2APrioriBound}
\mathfrak{E}(t) \leq 2\mathfrak{E}(0), \qquad 0 \leq t \leq \mathscr{T}_0 \epsilon^{-2}
\end{equation}
provided $\epsilon_0 > 0$ is chosen sufficiently small depending on $s, p, q, k, \mathscr{T}, \|A_0\|_{H^N}$.}
\item[\textbf(b)]{The following bounds hold:
\begin{equation}
\max_{0 \leq t \leq \mathscr{T}_0\epsilon^{-2}} \|\Lambda^s_\epsilon \mathfrak{z}(t)\| + \|\Lambda^s_\epsilon \mathfrak{z}_t(t)\| + \text{esssup}_{0 \leq t \leq \mathscr{T}_0\epsilon^{-2}} \|\Lambda^s_\epsilon \mathfrak{z}_{tt}(t)\| \leq C(s, p, q, k, \|A_0\|)\epsilon^{-(s + 1)(q + 1)}
\end{equation}
}
\item[(c)]{If $[\tau_1, \tau_2] \subset [0, \mathscr{T}_0\epsilon^{-2}]$, then
\begin{equation}
\left| \|\mathfrak{z}(\tau_2)\| - \|\mathfrak{z}(\tau_1)\| \right| + \left| \|\mathfrak{z}_t(\tau_2)\| - \|\mathfrak{z}_t(\tau_1)\| \right| \leq C(s, p, q, k, \|A_0\|)\epsilon^{-(s + 1)(q + 1)}|\tau_2 - \tau_1|
\end{equation}}
\item[(d)]{Suppose further that $\mathscr{P}(t) = \mathscr{P}$ is constant on some interval $[\tau_1, \tau_2]$.  Then for all $t \in [\tau_1, \tau_2]$ we have the estimate
\begin{equation}
\bigl| \|\Lambda^s_\epsilon \mathfrak{z}_{tt}(t)\| - \|\Lambda^s_\epsilon \mathfrak{z}_{tt}(0)\|\bigr| \leq C(s, p, q, k, \|A_0\|)\epsilon^{-(s + 1)(q + 1)} |\tau_2 - \tau_1|
\end{equation}}
\end{itemize}
\end{lemma}

\begin{proof}
\textbf{Proof of (a).}  The fact that $\mathfrak{z} = \mathcal{B}_k \mathfrak{z}$ is immediate once \eqref{GeneralParentEquation} is written in DuHamel form.  Differentiating $\mathfrak{E}(t)$ and using \eqref{GeneralParentEquation} yields the following differential inequality valid almost everywhere in $[0, \mathscr{T}_0\epsilon^{-2}]$:
\begin{align*}
\mathfrak{E}^\prime(t) & = \epsilon^5 \langle \mathfrak{z}_t, \mathcal{B}_k (\epsilon^{q + 4}\mathfrak{z} + 2i(\epsilon^7 t + \epsilon^5 T_1)\mathfrak{z}|\mathfrak{z}|^{q - 1}) \rangle  \mathscr{P}(t) \quad \text{a.e. } t \in [0, \mathscr{T}_0\epsilon^{-2}]\\
& \leq C\epsilon^5 \mathfrak{E}(t)^\frac12 (\epsilon^{7} \|\mathfrak{z}\| + C_\mathbf{T} \epsilon^5 \|\mathfrak{z}\| \|\mathfrak{z}\|_{L^\infty}^{q - 1}) \quad \text{a.e. } t \in [0, \mathscr{T}_0\epsilon^{-2}]\\
& \leq C_p \mathfrak{E}(t) (\epsilon^{7} \|\mathfrak{z}\| + C_{(q, \mathbf{T})} \epsilon^5 \|\mathfrak{z}\| \|\mathcal{B}_k\mathfrak{z}\|_{H^2}^{q - 1}) \quad \text{a.e. } t \in [0, \mathscr{T}_0\epsilon^{-2}]\\
& \leq C_{(p, q, k, \mathbf{T})} \epsilon^5 (\mathfrak{E}(t) + \mathfrak{E}(t)^{\frac{q - 1}{2}})\quad \text{a.e. } t \in [0, \mathscr{T}_0\epsilon^{-2}]
\end{align*}
We observe that $\mathfrak{E}(t)$ itself is an absolutely continuous function in $t$.  We begin a bootstrap argument: suppose that there is a first time $t_* \in [0, \mathscr{T}_0\epsilon^{-2})$ at which $\mathfrak{E}(t_*) = 2\mathfrak{E}(0)$.  If so, the above inequality would imply that $\mathfrak{E}^\prime(t) \leq C_{(p, q, k, \mathbf{T}, \mathfrak{E}(0))} \epsilon^5$ whenever $t \in [0, t_*]$.  But then integrating would give
\begin{align*}
\mathfrak{E}(t) \leq \mathfrak{E}(0) + C_{(p, q, k, \mathbf{T}, \mathfrak{E}(0))}\epsilon^5t \leq \mathfrak{E}(0) + C_{(p, q, k, \mathbf{T}, \mathfrak{E}(0))}\epsilon^2
\end{align*}
However, since $\mathfrak{E}(0)$ is uniformly bounded away from zero as $\epsilon \to 0$, we can choose $\epsilon$ sufficiently small so that $2\mathfrak{E}(0) = \mathfrak{E}(t_*) < 2\mathfrak{E}(0)$, which is the desired contradiction.

\vspace{0.5cm}

\textbf{Proof of (b).}  Using \eqref{L2APrioriBound} and the fact that $\mathfrak{z} = \mathcal{B}_k \mathfrak{z}$ and $\mathfrak{z}_t = \mathcal{B}_k \mathfrak{z}_t$, we have
\begin{align*}
\|\Lambda^s_\epsilon \mathfrak{z}(t)\| + \|\Lambda^s_\epsilon \mathfrak{z}_t(t)\| & \leq C\epsilon^{-s}\left(\|\mathfrak{z}(t)\| + \|\mathfrak{z}_t(t)\|\right) \\
& \leq C_{p, k}\epsilon^{-s} \mathfrak{E}(t)^\frac12 \\
& \leq C_{p, k}\epsilon^{-s} \mathfrak{E}(0)^\frac12 \\
& \leq C(p, k, \|A_0\|) \epsilon^{-s}
\end{align*}
Now appealing to the equation \eqref{GeneralParentEquation} we have for almost every $t \in [0, \mathscr{T}_0\epsilon^2]$ that
\begin{align*}
\|\mathfrak{z}_{tt}(t)\| & \leq \||D|^p \mathfrak{z}\| + \epsilon^2\|\mathfrak{z}|\mathfrak{z}/\epsilon|^{q - 1}\| + \|\mathcal{N}\| |\mathscr{P}(t)| \\
& \leq C(p, k, \|A_0\|)\epsilon^{-s} + (\epsilon^{-q} + \epsilon^5)C(p, q, \|A_0\|)\epsilon^{-s(q - 1)} \\
& \leq C(p, q, k, \|A_0\|)\epsilon^{-(s + 1)(q + 1)}
\end{align*}

\vspace{0.5cm}

\textbf{Proof of (c).}  This follows from (b) and the Fundamental Theorem of Calculus.

\vspace{0.5cm}

\textbf{Proof of (d).}  Writing \eqref{GeneralParentEquation} in the DuHamel formulation and applying $\Lambda^s_\epsilon$ gives the identity
\begin{align*}
\Lambda^s_\epsilon\mathfrak{z}_{tt}(\tau_2) & = |D|^\frac{p}{2}\sin(|D|^\frac{p}{2}(\tau_2 - \tau_1))\Lambda^s_\epsilon\mathfrak{z}_t(\tau_1) + \cos(|D|^\frac{p}{2}(\tau_2 - \tau_1))\Lambda^s_\epsilon\mathfrak{z}_{tt}(\tau_1) \\
& \qquad + \int_{\tau_1}^{\tau_2} \cos(|D|^\frac{p}{2}\tau)\Lambda^s_\epsilon\mathscr{N}(\tau) \, d\tau
\end{align*}
where the nonlinearity $\mathscr{N}(t)$ is given by \eqref{DifferentiatedNonlinearity} with $g_*(m^{(j)}_n \lambda)$ replaced by $\mathscr{P}$.  Estimating the initial data using the Taylor expansions of $\sin(|D|^\frac{p}{2}t)$ and $\cos(|D|^\frac{p}{2}t)$ in frequency as well as estimating $\|\Lambda^s_\epsilon \mathcal{N}\|$ just as in Lemma \ref{NanoLocalWellposedness} gives the desired bound; we omit the details.
\end{proof}

The method of controlling $\mathfrak{z}_{tt}$ as in Lemma \ref{AlmostConserved}(d) fails on intervals of times where the full nonlinearity exhibits jump discontinuities, since we cannot differentiate in time to quasilinearize the equation.  However, we can control the size of the jump discontinuities to provide acceptable estimates on the growth of $z_{tt}$.  The exact estimates are given in the

\begin{proposition}\label{Construct ApproximateSolutions}
(Construction of the $n$th approximate solution on $[0, 2t_0]$.)  Let $n \in \mathbb{N}$ be given.  Let $(u_0, v_0, w_0)$ satisfy the hypotheses of Proposition \ref{zLocalWellPosed}(b).  Then for $t_0 > 0$ chosen sufficiently small depending on $p, q, \|A_0\|_{H^s}, 1 - \lambda_0$, a solution $z^{\{n\}}$ to \eqref{ApproxEqnOnSubinterval} exists in the class $(z^{\{n\}}, z_t^{\{n\}}, z_{tt}^{\{n\}}) \in (C^1 \times C^0 \times L^\infty)([0, 2t_0]:L^2)$ which moreover satisfies $z(t) = \mathcal{B}_k z(t)$ for all $t \in [0, 2t_0]$, for which $t \mapsto \|\Lambda^s_\epsilon z_{tt}^{\{n\}}(t)\|$ is continuous from the right on $[0, 2t_0)$, as well as the bounds
\begin{equation}\label{ZAndZtStepBounds}
\max_{t \in [0, jh)} \|\Lambda^s_\epsilon z(t)\| + \omega^{-1}\|\Lambda^s_\epsilon z_t(t)\| \leq \left(\frac32 + \frac{j}{2n}\right)\|A_0\|_{H^s} \leq 2\|A_0\|_{H^s},
\end{equation}
\begin{equation}\label{ZLInftyStepBounds}
\max_{t \in [0, jh)} \|z(t)\|_{L^\infty} \leq \left(2 + \frac{j}{n}\right)\|A_0\| \leq 3\epsilon\|A_0\|_{L^\infty},
\end{equation}
and, defining $\lambda(t) = \frac{\|\Lambda^s_\epsilon z_{tt}(t)\|^2}{\nu^2 \omega^4 \|A_0\|_{H^s}^2}$ as usual,
\begin{equation}\label{ZttStepBounds}
\max_{t \in [0, jh)} \lambda(t) \leq \left(\lambda_0 + \frac{j}{2n}(1 - \lambda_0)\right) \leq \frac{1}{2}(\lambda_0 + 1)
\end{equation}
where in particular the size of the interval of existence $[0, 2t_0]$ is independent of $j$.
\end{proposition}

\begin{proof}
We proceed by induction on $j$.  The base step is given by Lemma \ref{FirstStepLWP}.  Now assume that we have constructed a solution $z^{\{n, j\}}$ defined on $[0, jh]$ satisfying the above bounds up on the interval $[0, jh]$.  We claim that we can construct a solution $z^{\{n, j + 1\}}$ defined on $[0, (j + 1)h]$ and satisfying the bounds \eqref{ZAndZtStepBounds}-\eqref{ZttStepBounds} with $j$ replaced by $j + 1$ on $[0, (j + 1)h]$.  

Choose $z^{\{n, j\}}(t_j), z^{\{n, j\}}_t(t_j)$ as the initial data for the initial value problem \eqref{ApproxEqnOnSubinterval}, as well as $z^{\{n, j\}}_{tt}(t_j) = z^{\{n, j\}}_{tt}(t_j^-)$ the left-hand value of the jump discontinuity at $t = t_j$.  These data satisfy the hypotheses of Lemma \ref{NanoLocalWellposedness} with the choice $M = 2$ uniformly in $j$.  By Lemma \ref{NanoLocalWellposedness}, there exists a solution $z^{[n, j]}$ to \eqref{ApproxEqnOnSubinterval} defined on $[jh, (j + 1)h]$ that satisfies the estimates
$$\|z^{[n, j]}_{tt}(t_j^+)\| \leq 4\|A_0\|_{H^s}\omega^2$$
and a solution $(z, z_t, z_{tt}) \in (C^2 \times C^1 \times C^0)(I_j : L^2)$ to the initial value problem
\begin{equation}\label{ApproxEqnOnSubinterval}
\begin{cases}
z_{tt} + |D|^p z + \epsilon^2\mathcal{B}_k(z|z/\epsilon|^{q - 1}) + \mathcal{N} g_\star(m_n^{(j)}\lambda(t)) = 0, \quad (x, t) \in \mathbb{R}^2 \times I_{j + 1} \\
\mathcal{N} := \mathcal{B}_k\left(\epsilon^{q + 4}\omega^2 z + 2i\omega^2(\epsilon^7 t + \epsilon^5(T_1 + \mathscr{T}_1))z|z|^{q - 1}\right) \\
z^{[n, j]}(t_j) =  z^{\{n, j\}}(t_j) \\
z^{[n, j]}_t(t_j) =  z^{\{n, j\}}_t(t_j) \\
\end{cases}
\end{equation}
which satisfies
$$\max_{t \in I_{j}} \left( \|\Lambda^s_\epsilon z(t)\|, \omega^{-1}\|\Lambda^s_\epsilon z_t(t)\| \right) \leq \left(\frac32 + C\frac{t_0}{n}\right)\|A_0\|_{{H}^s}$$
and
$$\max_{t \in I_j} \|\Lambda^s_\epsilon z_{tt}(t)\| \leq \left(3 + C \frac{t_0}{n}\right)\|A_0\|_{H^s}$$
for a constant $C$ depending on $s, p, q, \|A_0\|_{H^s}$.  Denote $$z^{\{n, j + 1\}}(t) := \begin{cases} z^{\{n, j\}}(t) & 0 \leq t < t_j \\ z^{[n, j]}(t) & t_j \leq t \leq t_{j + 1} \end{cases}$$

Now we can apply Lemma \ref{AlmostConserved}(c) to $z^{\{n, j + 1\}}$ to conclude the bounds
\begin{align*}
\max_{0 \leq t \leq t_{j + 1}} \left( \|\Lambda^s_\epsilon z(t)\|, \omega^{-1}\|\Lambda^s_\epsilon z_t(t)\| \right) & \leq \left(\frac32 + C_{(p, q, k, \|A_0\|, \epsilon)}(j + 1)h\right)\|A_0\|_{{H}^s} \\
& \leq \left(\frac32 + \frac{(j + 1)}{2n}\right)\|A_0\|_{{H}^s}
\end{align*}
where we have chosen $t_0 > 0$ sufficiently small depending on $p, q, \|A_0\|, \epsilon$, and independently of $j$.

The estimate of $L^\infty$ can proceed directly using the equation:  integrating twice in time yields the formula for $0 \leq t \leq (j + 1)h$:
\begin{align*}
z^{\{n, j + 1\}}(t) & = z^{\{n, j + 1\}}(0) + tz^{\{n, j + 1\}}_t(0) \\
& \quad + \int_0^t \int_0^\tau \epsilon^2z^{\{n, j + 1\}}(\tau^\prime)|z^{\{n, j + 1\}}(\tau^\prime)/\epsilon|^{q - 1} + \mathcal{N}(z^{\{n, j + 1\}}(\tau^\prime))g_\star(m\lambda(z^{\{n, j + 1\}}_{tt}(\tau^\prime)) \, d\tau_1 \, d\tau 
\end{align*}
which leads to
\begin{align*}
\|z^{\{n, j + 1\}}(t)\|_{L^\infty} & \leq \|z^{\{n, j + 1\}}(0)\|_{L^\infty} + t\|z^{\{n, j + 1\}}_t(0)\|_{L^\infty}\\
& \quad  + C_{\epsilon, k, \mathbf{T}} t^2 (\|z^{\{n, j + 1\}}\|_{C[0, t_{j + 1}] : L^\infty)} + \|z^{\{n, j + 1\}}\|_{C[0, t_{j + 1}] : L^\infty)}^{q - 1})
\end{align*}
Choosing $t_0$ sufficiently small depending on $\epsilon, q, k, \mathbf{T}, \|A_0\|_{H^s}$ now implies \eqref{ZLInftyStepBounds}.

\vspace{0.5cm}

It only remains to demonstrate control of $|\|\Lambda^s_\epsilon z^{\{n, j + 1\}}_{tt}(t_j^+)\| - \|\Lambda^s_\epsilon z^{\{n, j + 1\}}_{tt}(t_j^-)\||$.  Denote as usual 
\begin{equation}
\mathcal{N}_j := (\epsilon^{q + 4}\omega^2z^{\{n, j + 1\}}(t_j) + 2i(\epsilon^7 t_j + \epsilon^5(T_1 + \mathscr{T}_1))z^{\{n, j + 1\}}(t_j)|z^{\{n, j + 1\}}(t_j)|^{q - 1}
\end{equation}
Using the equation and noting that $\Lambda^s_\epsilon z^{\{n, j + 1\}}$ and $\Lambda^s_\epsilon z^{\{n, j + 1\}}_t$ are continuous at $t = t_j$, we find the following expression for the jump:
\begin{align*}
|\|\Lambda^s_\epsilon z^{\{n, j + 1\}}_{tt}(t_j^+)\| - \|\Lambda^s_\epsilon z^{\{n, j + 1\}}_{tt}(t_j^-)\|| & \leq \|\Lambda^s_\epsilon z^{\{n, j + 1\}}_{tt}(t_j^+) - \Lambda^s_\epsilon z^{\{n, j + 1\}}_{tt}(t_j^-)\| \\
& \leq \|\mathcal{N}_j\| \left| g_\star(m_n^{(j + 1)} \lambda) - g_\star(m_n^{(j)} \lambda) \right| \\
& \leq \|\mathcal{N}_j\| \max_{\lambda > 0}(g_\star(\lambda))\left|(m_n^{(j + 1)} \lambda) - (m_n^{(j)} \lambda) \right| \\
& \leq \|\mathcal{N}_j\| \max_{\lambda > 0}(g_\star(\lambda))\left|(m_n^{(j + 1)} \lambda) - \lambda(z^{\{n, j + 1\}}_{tt}(t_j^+)) \right| \\
& \quad + \|\mathcal{N}_j\|\max_{\lambda > 0}(g_\star(\lambda))\left|\lambda(z^{\{n, j + 1\}}_{tt}(t_j^+)) - \lambda(z^{\{n, j + 1\}}_{tt}(t_j^-)) \right| \\
& \quad + \|\mathcal{N}_j\|\max_{\lambda > 0}(g_\star(\lambda))\left|(m_n^{(j)} \lambda) - \lambda(z^{\{n, j + 1\}}_{tt}(t_j^-)) \right| \\
& := I_1 + I_2 + I_3
\end{align*}
We estimate $\|\mathcal{N}_j\|$ as usual, using the same values of $\omega$ and $\mathscr{C}$ chosen in the proof of Lemma \ref{NanoLocalWellposedness}:
\begin{align*}\label{N1Bound}
\|\mathcal{N}_j\| & \leq \epsilon^{q + 4}\omega^2 3\|A_0\|_{H^s} + 6C_{s, q}\epsilon^{q + 4}\omega^2 \left(2 \omega^{\frac{4}{p} - 1} \mathbf{T}\right)\|\mathfrak{u}\|_{C([0, \mathbf{T}] : L^\infty)}^{q - 1} \|A_0\|_{H^s} \\
& \leq 4\epsilon^{q + 4}\omega^2\|A_0\|_{H^s}
\end{align*}
Consider $I_1$.  Since $\lambda(z^{\{n, j + 1\}}_{tt}(t))$ is continuous on $I_{j + 1}$, by the Mean Value Theorem there is a value $t_*$ at which $m^{(j + 1)}_n \lambda = \lambda(z^{\{n, j + 1\}}_{tt}(t_*))$.  Therefore we have using Lemma \ref{AlmostConserved}(d) that
\begin{align*}
I_1 & = \|\mathcal{N}_j\| \max_{\lambda > 0}(g_\star(\lambda))\left|\lambda(z^{\{n, j + 1\}}_{tt}(t_*)- \lambda(z^{\{n, j + 1\}}_{tt}(t_j^+)) \right| \\
& \leq C_{(s, q, \|A_0\|_{H^s})} \left|\lambda(z^{\{n, j + 1\}}_{tt}(t_*))^\frac12 + \lambda(z^{\{n, j + 1\}}_{tt}(t_j^+))^\frac12 \right|\left|\lambda(z^{\{n, j + 1\}}_{tt}(t_*))^\frac12 - \lambda(z^{\{n, j + 1\}}_{tt}(t_j^+))^\frac12 \right| \\
& \leq C_{(s, q, k, \|A_0\|_{H^s})} 2\left(3 + C\frac{t_0}{n}\right) \left|\|z^{\{n, j + 1\}}_{tt}(t_*))\| - \|z^{\{n, j + 1\}}_{tt}(t_j^+)\| \right| \\
& \leq C_{(s, q, k, \|A_0\|_{H^s})} \frac{t_0}{n}
\end{align*}
Estimating $I_3$ in exactly the same way gives $I_3 \leq C_{(s, q, k, \|A_0\|)}\frac{t_0}{n}$.  To estimate $I_2$, we proceed more carefully:
\begin{align*}
I_2 & \leq \|\mathcal{N}_j\| \frac{32\epsilon^{-(q + 4)}}{\mathscr{C}}\left|\lambda(z^{\{n, j + 1\}}_{tt}(t_j^+)) - \lambda(z^{\{n, j + 1\}}_{tt}(t_j^-)) \right| \\
& \leq 4\omega^2\|A_0\|_{H^s} \frac{32}{\mathscr{C}}\left|\lambda(z^{\{n, j + 1\}}_{tt}(t_j^+)) - \lambda(z^{\{n, j + 1\}}_{tt}(t_j^-)) \right| \\
& \leq 4\omega^2\|A_0\|_{H^s} \frac{32}{\mathscr{C}}2\left(3 + C\frac{t_0}{n}\right)\omega^2\left|\frac{\|\Lambda^s_\epsilon z^{\{n, j + 1\}}_{tt}(t_j^+)\|}{\omega^4\|A_0\|_{H^s}} - \frac{\|\Lambda^s_\epsilon z^{\{n, j + 1\}}_{tt}(t_j^-)\|}{\omega^4\|A_0\|_{H^s}} \right| \\
& \leq \frac{1024}{\mathscr{C}}\left|\|\Lambda^s_\epsilon z^{\{n, j + 1\}}_{tt}(t_j^+)\| - \|\Lambda^s_\epsilon z^{\{n, j + 1\}}_{tt}(t_j^-)\| \right| \\
& \leq \frac12\left|\|\Lambda^s_\epsilon z^{\{n, j + 1\}}_{tt}(t_j^+)\| - \|\Lambda^s_\epsilon z^{\{n, j + 1\}}_{tt}(t_j^-)\| \right|
\end{align*}
But then we have the bound
\begin{align*}
|\|\Lambda^s_\epsilon z^{\{n, j + 1\}}_{tt}(t_j^+)\| - \|\Lambda^s_\epsilon z^{\{n, j + 1\}}_{tt}(t_j^-)\|| & \leq 2\left(|\|\Lambda^s_\epsilon z^{\{n, j + 1\}}_{tt}(t_j^+)\| - \|\Lambda^s_\epsilon z^{\{n, j + 1\}}_{tt}(t_j^-)\|| - I_2\right) \\
& \leq 2(I_1 + I_3) \\
& \leq C_{(s, q, \|A_0\|_{H^s})}\frac{t_0}{n}
\end{align*}
so that, upon choosing $t_0 > 0$ sufficiently small depending on $s, q, \|A_0\|_{H^s}, 1 - \lambda_0$, we can complete the induction step as follows:
\begin{align*}
\max_{t \in [0, (j + 1)h)} \lambda(t) & \leq \max_{t \in [0, jh)} \lambda(t) \\
& \quad + |\lambda(z^{\{n, j + 1\}}_{tt}(t_j^+)) - \lambda(z^{\{n, j + 1\}}_{tt}(t_j^-))| \\
& \quad + \max_{t \in I_{j + 1}} |\lambda(z^{\{n, j + 1\}}_{tt}(t)) - \lambda(z^{\{n, j + 1\}}_{tt}(t_j^+))| \\
& \leq \left(\lambda_0 + \frac{j}{2n}(1 - \lambda_0)\right) + C_{(s, q, \|A_0\|_{H^s})}\frac{t_0}{n} \\
& \leq \left(\lambda_0 + \frac{j}{2n}(1 - \lambda_0)\right) + \frac{1}{2n}(1 - \lambda_0) \\
& \leq \left(\lambda_0 + \frac{j + 1}{2n}(1 - \lambda_0)\right)
\end{align*}
which completes the induction step.  Now take $z^{\{n\}}(t) := z^{\{n, n\}}(t)$ to be the desired solution.  The right hand continuity of $t \mapsto \|\Lambda^s_\epsilon z_{tt}^{\{n\}}(t)\|$ is a consequence of the definition of $m_n$ given in \eqref{PiecewiseConstantAverage}.
\end{proof}

\vspace{0.5cm}

\subsection{Extraction of a Solution by Compactness}

Denote $w_n := \Lambda^s_\epsilon z^{\{n\}}_{tt}$ for brevity.  In \S 4.2 we have constructed a sequence $(w_n)$ in $L^\infty(\mathbb{R}^2 \times [0, 2t_0])$ with the uniform bound
\begin{align*}
\max_{(x, y, t) \in \mathbb{R}^2 \times [0, 2t_0]} |w_n(x, y, t)| & \leq \max_{0 \leq t \leq 2t_0} C\|w_n(t)\|_{H^2} \\
& \leq \max_{0 \leq t \leq 2t_0} C_k\|w_n(t)\|_{L^2} \\
& \leq C_{k, \|A_0\|_{H^s}}\left(\frac{1 + \lambda_0}{2}\right)^\frac12 \\
& \leq C_{k, \|A_0\|_{H^s}}
\end{align*}

Fix a time $t \in [0, 2t_0)$.  By the Banach-Alaoglu Theorem, we may extract from $(w_n(t))$ a $\text{weak*-}L^\infty(\mathbb{R}^2)$ Cauchy subsequence, which we will again denoted by $(w_n(x, y, t))$.  Hence for all $f \in L^1(\mathbb{R}^2)$, the sequence $\langle w_n(t), f \rangle$ is Cauchy.

We claim that in fact for each fixed $(x, y) \in \mathbb{R}^2$, the sequence $(w_n(x, y, t))$ is Cauchy (note that this is trivial for all $(x, y) \in \mathbb{R}^2$ when $t = 0$).  To see this, let $\delta > 0$ be given, take $n_1, n_2 \in \mathbb{N}$ to be chosen later, and let $\psi = \mathbf{1}_Q$ where
\begin{equation}\label{BumpFunctionSupport}
Q = \left\{(x, y, t) \in \mathbb{R}^3 : |x| \leq \frac12, \, |y| \leq \frac12\right\}.
\end{equation}
For $\rho > 0$ consider the approximations to the identity $\psi_\rho(x, y) = \rho^{-2} \psi(\rho^{-1}(x, y))$.  Recall that since $w_n = \mathcal{B}_k w_n$ for each fixed $t$, $(x, y) \mapsto w_n(x, y, t)$ is continuous.  Hence $$w_{n_1}(x, y, t) = \lim_{\rho \to 0} \int_{[0, 2t_0] \times \mathbb{R}^2}  w_{n_1}(x^\prime, y^\prime, t) \psi_\rho(x - x^\prime, y - y^\prime) \, dx^\prime \, dy^\prime$$ by the continuity of $w_1(\cdot, t)$, and similarly for $w_2$.  We estimate that
\begin{align}\label{WeakStarTriangle}
& \quad \; |w_{n_1}(x, y, t) - w_{n_2}(x, y, t)| \notag \\
& \leq \left| w_{n_1}(x, y, t) - \int_{\mathbb{R}^2}  w_{n_1}(x^\prime, y^\prime, t) \psi_\rho(x - x^\prime, y - y^\prime) \, dx^\prime \, dy^\prime \right| \notag \\
& \quad + \left| \int_{\mathbb{R}^2}  (w_{n_1}(x^\prime, y^\prime, t^\prime) - w_{n_2}(x^\prime, y^\prime, t^\prime) ) \psi_\rho(x - x^\prime, y - y^\prime, t - t^\prime) \, dx^\prime \, dy^\prime \right| \\
& \quad + \left| w_{n_2}(x, y, t) - \int_{\mathbb{R}^2}  w_{n_2}(x^\prime, y^\prime, t) \psi_\rho(x - x^\prime, y - y^\prime) \, dx^\prime \, dy^\prime \right| \notag 
\end{align}
Since each $w_n = \mathcal{B}_k w_n$, we have the bound
\begin{align*}
|w_n(x, y, t) - w_n(x^\prime, y^\prime, t)| & \leq \|\nabla w_n(t)\|_{L^\infty_{xy}}|(x, y) - (x^\prime, y^\prime)| \\
& \leq \|w_n(t)\|_{H^3}|(x, y) - (x^\prime, y^\prime)| \\
& \leq C\|w_n(t)\|\,|(x, y) - (x^\prime, y^\prime)|
\end{align*}
from which we have the following uniform bound in $n$, thanks to Proposition \ref{Construct ApproximateSolutions}:
\begin{align*}
& \quad \left| w_{n}(x, y, t) - \int_{\mathbb{R}^2}  w_{n}(x^\prime, y^\prime, t) \psi_\rho(x - x^\prime, y - y^\prime) \, dx^\prime \, dy^\prime \right| \\
& \leq \left| \int_{\mathbb{R}^2} C(\|A_0\|_{H^s}) |(x - x^\prime, y - y^\prime)| \psi_\rho(x - x^\prime, y - y^\prime) \, dx^\prime \, dy^\prime \right| \\
& \leq C(\|A_0\|_{H^s})\rho
\end{align*}
Choose $\rho$ so small so that $C(\|A_0\|_{H^s})\rho < \frac13 \delta$.  By the weak$^*$ convergence of $(w_n(t))$, there is an $n_0$ (depending only on $\rho$, and thus only on $\|A_0\|_{H^s}$ and $\delta$) so that the second expression on the right hand side of \eqref{WeakStarTriangle} can be made smaller than $\frac13 \delta$ whenever $n_1, n_2 \geq n_0$.

This establishes that the sequence $(w_n)$ converges pointwise everywhere on $\mathbb{R}^2 \times [0, 2t_0)$ to a limit $w \in L^\infty([0, 2t_0) \times \mathbb{R}^2)$ with $\max_{(x, y, t) \in \mathbb{R}^2 \times [0, 2t_0)} |w(x, y, t)| \leq C_{k, \|A_0\|_{H^s}}$.  This $w$ is our candidate for a solution to \eqref{zEquationOutline}.

Since $w(x, y, t)$ is defined pointwise everywhere on $\mathbb{R}^2 \times [0, 2t_0)$, we may consider restrictions of the limit $w(x, y, t)$ for fixed $t$, and for fixed $(x, y)$.  By dominated convergence we have $w(t) \in L^2_{xy}$.  Moreover, we claim that for every fixed $t$, $w(t) = \mathcal{B}_k w(t)$.  To see this, we have
\begin{align*}
\|(I - \mathcal{B}_k)w\| & = \sup \{ |\langle (I - \mathcal{B}_k)w, f \rangle| : \|f\| = 1, \, f \text{ simple}\} \\
& = \sup \{ |\langle w, (I - \mathcal{B}_k)f \rangle|  : \|f\| = 1, \, f \text{ simple}\} \\
& =  \sup \{ \lim_{n \to \infty} |\langle w_n, (I - \mathcal{B}_k)f \rangle|  : \|f\| = 1, \, f \text{ simple}\} \\
& =  \sup \{ \lim_{n \to \infty} |\langle (I - \mathcal{B}_k) w_n, f \rangle|  : \|f\| = 1, \, f \text{ simple}\} \\
& = 0
\end{align*}
If we define as usual $$u_n(t) = u_0 + tv_0 + \int_0^t \int_0^\tau w_n(\tau^\prime) d\tau^\prime \, d\tau$$ and $u(t)$ similarly, then this also allows us to conclude quickly that $|D|^p u^{\{n\}} \to |D|^p u$ pointwise.  By dominated convergence in $t$ the question reduces to showing that $|D|^p (w_n - w) \to 0$ pointwise.  Since we have constructed $w_n$ so that $w_n = \mathcal{B}_k w_n$ and $w = \mathcal{B}_k w$ from above, we can write in addition that $|D|^p (w_n - w) = |D|^p \hat{\varphi}(|D|)(w_n - w)$ where $\hat{\varphi}(|\xi|)$ is a smooth compactly supported cutoff function that is identically 1 on the ball $\{\xi : |\xi| \leq 2\}$.  But then this allows us to write $|D|^p (w_n - w)$ as convolution of $w_n - w$ with a radial function $\psi(x, y)$ in Schwarz class, and so pointwise convergence follows by another application of dominated convergence in $x, y$.  

Similarly, if we denote $\lambda_n = \lambda(w_n)$ and $\lambda = \lambda(w)$, we have by dominated convergence in $x, y$ that $\lambda_n \to \lambda$ pointwise everywhere on $[0, 2t_0)$, and therefore that $\max_{0 \leq t < 2t_0} |\lambda(t) - \lambda_0| \leq \frac12(1 - \lambda_0)$.  Then dominated convergence in $t$ implies $\|\lambda_n - \lambda\|_{L^1_t} \to 0$ as $n \to \infty$.

The only fact left to show is that the term $g_*(m_n(\lambda_n))$ converges to $g_*(\lambda)$ almost everywhere in $[0, 2t_0)$.  It suffices to show that $\|g_*(m_n(\lambda_n)) - g_*(\lambda)\|_{L^1_t} \to 0$ as $n \to \infty$.  But it then suffices in turn to show that $\|m_n(\lambda_n) - \lambda \|_{L^1_t} \to 0$, since we have again by the Fundamental Theorem of Calculus and the fact that $m_n(\lambda_n), \lambda \leq \frac12 + \frac12 \lambda_0$ that
\begin{align*}
\|g_*(m_n(\lambda_n)) - g_*(\lambda)\|_{L^1_t} \leq \left\| g_*^\prime\left(\frac12 + \frac12\lambda_0 \right) \left(m_n(\lambda_n) - \lambda \right) \right\|_{L^1_t} \leq \frac{32\epsilon^{-(q + 4)}}{\mathscr{C}} \left\|m_n(\lambda_n) - \lambda \right\|_{L^1_t}
\end{align*}
To show the remaining estimate, we write
\begin{align*}
\|m_n(\lambda_n) - \lambda\|_{L^1_t} & \leq \|m_n(\lambda_n - \lambda)\|_{L^1_t} + \|m_n(\lambda) - \lambda\|_{L^1_t}
\end{align*}
The first term approaches zero as $n \to \infty$ since 
\begin{equation}\label{MeansConvergeInL1}
\|m_n(f)\|_{L^1} \leq \|f\|_{L^1}
\end{equation}
as
\begin{align*}
\|m_n(f)\|_{L^1[0, 2t_0)} \leq \sum_{j = 1}^n \left\|\fint_{I_j} f(\tau) \, d\tau \right\|_{L^1(I_j)} \leq \sum_{j = 1}^n h \left( \frac{1}{h}\int_{t_{j - 1}}^{t_j} |f(\tau)| \, d\tau \right) = \|f\|_{L^1[0, 2t_0)}
\end{align*}
To show that the second term approaches zero, recall (see for instance \cite{Royden}) that for any $\iota > 0$ we may find a step function $\lambda_\iota(t) = \sum_{l = 1}^m c_l \mathbf{1}_{[a_l, b_l]}$ where $\|\lambda_\iota\|_{L^\infty_t} \leq \|\lambda\|_{L^\infty_t}$,  $\|\lambda_\iota - \lambda\|_{L^1_t} < \iota$, and where without loss of generality we have $a_1 < b_1 < a_2 < \cdots < a_m < b_m$.  Then by a triangle inequality argument using \eqref{MeansConvergeInL1}, it suffices to show $\|m_n(\lambda_\iota) - \lambda_\iota\|_{L^1_t}$ can be made arbitrarily small for large $n$.  Now there can be at most $2m$ of the intervals $I_j = [(j - 1)t_n, jt_n]$ containing the endpoints $a_l, b_l$; if an interval $I_j$ does not contain one of these endpoints, $\lambda_\iota(t)$ is a constant on $I_j$ and so $\fint_{I_j} \lambda_\iota(\tau)\, d \tau - \lambda_\iota(t) = 0$ on all of $I_j$.  But then we have the bound
\begin{align*}
\|m_n(\lambda_\iota) - \lambda_\iota\|_{L^1_t} \leq \frac{2m}{n} 2\|\lambda_\iota\|_{L^\infty_t} \leq  \frac{4m}{n}\left(\frac12 + \frac12\lambda_0 \right) < \frac{4m}{n} < \iota
\end{align*}
provided we choose $n$ large enough depending on $m$ and $\iota$.  Now taking the pointwise limit of the equations \eqref{ApproxEqnOnSubinterval} yields the equation \eqref{zEquationOutline} for all $t$ excluding the measure zero set on which $m_n(\lambda_n(t)) - \lambda(t) \not\to 0$.  

Observe that since the initial data of the approximate solutions $(z^{\{n\}}(0), z_t^{\{n\}}(0))$ all agree, $z$ itself satisfies the correct initial data.  Since $z$ solves \eqref{zEquationOutline} with $g$ replaced by $g_*$ almost everywhere, there is a $\mathbf{t} \in (t_0, 2t_0)$ at which $z$ solves \eqref{zEquationOutline} with $g$ replaced by $g_*$.  Finally, note that since $g^\prime(\lambda_0) \leq \frac{16\epsilon^{-(q + 4)}}{\mathscr{C}}$, we have by \eqref{ZttStepBounds} that $\max_{t \in [0, \mathbf{t}]} g^\prime(\lambda(t)) \leq \frac{16\sqrt{2}\epsilon^{-(q + 4)}}{\mathscr{C}}$.  Therefore on $[0, \mathbf{t}]$ we have $g_*(\lambda(t)) = g(\lambda(t))$, and so the solution $z(t)$ constructed here satisfies \eqref{zEquationOutline}.  This completes the Proof of Proposition \ref{zLocalWellPosed}.

\begin{remark}\label{HamIsQuantUseless}
The role of the conserved Hamiltonian in the local well-posedness argument deserves comment.  One might wonder at the outset of this paper why one cannot use the coarse bound of Lemma \ref{AlmostConserved} provided by the almost conserved Hamiltonian of \eqref{zEquationOutline} to derive a bound on $\|\Lambda^s_\epsilon z_{tt}(t)\|$ without introducing a penalization term at all.  The reason is that, although the bounds given by Lemma \ref{AlmostConserved} are uniform on $[0, \mathscr{T}\epsilon^{-2}]$, they scale badly in $\epsilon$.  The key estimate in the proof of Theorem \ref{BigSummaryTheorem} therefore will not close in this case due to the poor dependence of $\epsilon_0$ on $\|A\|_{H^N}$.  The exception is if one were to replace $s$ by $0$, which only recovers the already-known mass conservation of $A$.  However, despite the fact that the bounds of Lemma \ref{AlmostConserved} are quantitatively useless, they are nevertheless indispensable in providing the qualitative uniform control needed to complete the construction of the approximating sequence in \S4.2.
\end{remark}

%\begin{remark}
%Although the solution constructed in Proposition \ref{zLocalWellPosed} is a limit of discontinuous functions, the finer control over $\|\Lambda^s_\epsilon z_{tt}^{\{n\}}(t)\|$ provided in Proposition \ref{Construct ApproximateSolutions} is in fact enough to show the continuity in time of $t \mapsto \|\Lambda^s_\epsilon z_{tt}(t)\|$.  However, we will not record the proof here since it is not needed in the rest of the argument.
%\end{remark}

%==============================================================================
%==============================================================================
%==============================================================================
\section{Long-Time Wellposedness of the Progenitor Equation}
%==============================================================================
%==============================================================================
%==============================================================================

It only remains to show that the solution $z(t)$ to the problem \eqref{zEquationOutline} exists on the whole interval $[0, \mathscr{T}\epsilon^{-2}]$; this is equivalent to providing an a priori upper bound on $\lambda(t)$ that is uniformly away from 1 on $[0, \mathscr{T}\epsilon^{-2}]$.  We do this by showing that $\lambda^\prime(t)$ is negative for values of $\lambda(t)$ sufficiently close to 1.  We begin with the formal calculation of this fact.  Unless otherwise indicated, the constants $C$ that appear below may change from line to line, but depend only on $s, p, q, k, \|A_0\|_{H^N}, \mathscr{T}, \nu$.  Assuming that the quantity $z_{ttt}$ is well-behaved and we are considering an interval of time on which $z(t)$ solves \eqref{zEquationOutline}, we have by Proposition \ref{RemainderEstimates} that
\begin{align*}
\lambda^\prime(t) & = 2\frac{\langle \Lambda^{2s}_\epsilon z_{tt}, z_{ttt} \rangle}{\nu^2 \omega^4 \|\Lambda^s A_0\|^2} \\
& = 2\frac{\langle \Lambda^{2s}_\epsilon z_{tt}, -|D|z_t - \epsilon^{2}\partial_t\mathcal{B}_k(z|z/\epsilon|^{q - 1}) - \mathcal{N}_t g(\lambda) - \mathcal{N} g^\prime(\lambda) \lambda^\prime(t)\rangle}{\nu^2 \omega^4 \|\Lambda^s A_0\|^2} \\
& \geq 2\frac{\langle \Lambda^{2s}_\epsilon z_{tt}, -|D|z_t - \epsilon^{2}\partial_t\mathcal{B}_k(z|z/\epsilon|^{q - 1}) - \mathcal{N} g^\prime(\lambda) \lambda^\prime(t)\rangle}{\nu^2 \omega^4 \|\Lambda^s A_0\|^2} - C\epsilon^7
\end{align*}
Using Propositions \ref{LambdaPrimeMultiscale} and \ref{RemainderEstimates}, we find that to leading order that there is a positive constant $C_\omega > 0$ so that
\begin{align*}
& \quad 2\frac{\langle \Lambda^{2s}_\epsilon z_{tt}, -|D|z_t - \epsilon^{2}\partial_t\mathcal{B}_k(z|z/\epsilon|^{q - 1})\rangle}{\nu^2 \omega^4 \|\Lambda^s A_0\|^2} \\
& \geq 2\frac{\langle \Lambda^{2s}_\epsilon \tilde{z}_{tt}, -|D|\tilde{z}_t - \epsilon^{2}\partial_t\mathcal{B}_k(\tilde{z}|\tilde{z}/\epsilon|^{q - 1}) \rangle}{\nu^2 \omega^4 \|\Lambda^s A_0\|^2} - C\epsilon^3\\
& \geq C_\omega \epsilon^2 \frac{\langle \Lambda^{2s}A, iA|A|^{q - 1}\rangle}{\nu^2 \|\Lambda^s A_0\|^2} - C\epsilon^3.
\end{align*}
so that we have the estimate
\begin{align*}
\lambda^\prime(t) & \geq C_\omega \epsilon^2 \frac{\langle \Lambda^{2s}A, iA|A|^{q - 1}\rangle}{\nu^2\|\Lambda^s A_0\|^2} - C\epsilon^3 + 2\frac{\langle \Lambda^{2s}_\epsilon z_{tt}, -\mathcal{N} \rangle}{\nu^2 \omega^4 \|\Lambda^s A_0\|^2} g^\prime(\lambda(t))\lambda^\prime(t)
\end{align*}
Now since $\mathscr{C} > 1$, we can write using Assumption 1 that
\begin{align*}
\frac{\langle \Lambda^{2s}A, iA|A|^{q - 1}\rangle}{\nu^2 \|A_0\|_{H^s}^2} & = \frac{1}{2(T + T_1)\nu^2\|A_0\|_{H^s}^2}\left(\frac{d}{dT}\left((T + T_1)\|A(T)\|_{H^s}^2\right) - \|A(T)\|_{H^s}^2 \right) \\
& \geq \frac{\mathscr{C} - 1}{2T_2}\frac{\|A(T)\|_{H^s}^2}{\nu^2\|A_0\|_{H^s}^2} \\
& \geq \frac{\mathscr{C} - 1}{8T_2}
\end{align*}
upon choosing $\epsilon_0 > 0$ sufficiently small.  We then have the following ODE for $\lambda(t)$:
\begin{equation}\label{ODEBoundOnLambda}
\left(1 - 2\frac{\langle \Lambda^{2s}_\epsilon z_{tt}, -\mathcal{N} \rangle}{\nu^2\omega^4 \|\Lambda^s A_0\|^2} g^\prime(\lambda(t)) \right)\lambda^\prime(t) \geq \frac{\mathscr{C} - 1}{8T_2}C_\omega \epsilon^2
\end{equation}
Now suppose that we consider a time $t_*$ at which $\lambda(t)$ is close enough to $1$ so that $$\epsilon^{q + 4} g^\prime(\lambda(t_*)) \geq \frac{8}{\mathscr{C}}$$  Then we can estimate the remaining unsimplified term in \eqref{ODEBoundOnLambda} as follows:
\begin{align*}
& \quad 2\frac{\langle \Lambda^{2s}_\epsilon z_{tt}, -\mathcal{N} \rangle}{\nu^2 \omega^4 \|\Lambda^s A_0\|^2} \\
& = 2\frac{\langle -\Lambda^{2s}_\epsilon z_{tt}, \omega^2\mathcal{B}_k\left(\epsilon^{q + 4}z + 2i\epsilon^5 (\epsilon^2 t + T_1) z|z|^{q - 1}\right) \rangle}{\nu^2 \omega^4 \|\Lambda^s A_0\|^2} \\
&  \geq 2\frac{\langle -\Lambda^{2s}_\epsilon \tilde{z}_{tt}, \omega^2\mathcal{B}_k\left(\epsilon^{q + 4}\tilde{z} + 2i\epsilon^5 (\epsilon^2 t + T_1) \tilde{z}|\tilde{z}|^{q - 1}\right) \rangle}{\nu^2 \omega^4 \|\Lambda^s A_0\|^2} - C\epsilon^{q + 7} \\
&  \geq 2\frac{\langle \omega^2\Lambda^{2s}_\epsilon \tilde{z}, \omega^2\left(\epsilon^{q + 4}\tilde{z} + 2i\epsilon^5 (\epsilon^2 t + T_1) \tilde{z}|\tilde{z}|^{q - 1}\right) \rangle}{\nu^2 \omega^4 \|\Lambda^s A_0\|^2} - C\epsilon^{q + 5} \\
& \geq 2\epsilon^{q + 4}\frac{\langle \Lambda^{2s} A, A + 2i(T + T_1)A|A|^{q - 1}\rangle}{\nu^2 \|\Lambda^s A_0\|^2} - C\epsilon^{q + 5} \\
& \geq \epsilon^{q + 4}\left(\frac{2\mathscr{C}\|A\|_{H^s}^2}{\nu^2\|A_0\|_{H^s}^2} - C\epsilon\right) \\
& \geq \epsilon^{q + 4} \left(2\mathscr{C}\lambda(t) - C\epsilon\right)
\end{align*}
In the above we used Proposition \ref{RemainderEstimates} in the second step, Lemma \ref{WavePacketCutoff} in the third step as well as Proposition \ref{LambdaPrimeMultiscale} in the third and fourth steps, and the growth condition of Assumption 1 in the fifth step.  The final step is just the definition of $\lambda(t)$ along with Propositions \ref{LambdaPrimeMultiscale} and \ref{RemainderEstimates}.

As a result, the coefficient of $\lambda^\prime(t_*)$ in \eqref{ODEBoundOnLambda} is negative.  Thus \textit{if} one assumes that $\lambda(t)$ is continuous in $t$ and $\lambda(t) < 1$, we conclude that that $\lambda^\prime(t)$ is negative when $\lambda(t)$ is close enough to 1 for small enough $\epsilon_0 > 0$, which yields an a priori bound of $\lambda(t)$ away from 1.

This is not literally correct since Proposition \ref{zLocalWellPosed} alone does not assert that $\lambda(t)$ is continuous, and since the governing equation \eqref{zLocalWellPosed} is only satisfied almost everywhere.  However, the estimate (iv) of Proposition \ref{zLocalWellPosed}(b) gives at least some control of the rate at which $\lambda(t)$ can change over the course of successive local constructions.  If $\lambda(t)$ becomes sufficiently close to $1$, then this weaker control permits us to conclude that $\lambda(t)$ is very close to 1 but still uniformly bounded away from 1 on an interval of time.  On such an interval, the same formal calculation above implies that $\lambda^\prime(t)$ remains almost everywhere bounded, which in turn implies that $\lambda(t)$ is absolutely continuous and is enough to make the naive ODE calculation above rigorous.  The precise argument is as follows:

\begin{proof}[Proof of Proposition \ref{zLongTimeWellPosed}.]
By Proposition \ref{zLocalWellPosed} (a), (b) there exists a solution to \eqref{zEquationOutline} for some short time $[0, \mathbf{t}_1]$.  The single use of Proposition \ref{zLocalWellPosed}(a) entails making only a single restriction on the smallness of $\epsilon_0$.  Then, using Proposition \ref{zLocalWellPosed}(b) along with Proposition \ref{RemainderEstimates}, we can iterate the construction on time intervals of the form $[0, \mathbf{t}_1], [\mathbf{t}_1, \mathbf{t}_2], \ldots, [\mathbf{t}_j, \mathbf{t}_{j + 1}], \ldots$.  On each interval the parameter $\mathscr{T}_1$ of Proposition \ref{zLocalWellPosed} is chosen so that $\mathscr{T}_1\epsilon^{-2} = \mathbf{t}_j$.   The length of the interval $[\mathbf{t}_j, \mathbf{t}_{j + 1}]$ depends on $s, p, q, \epsilon, \mathscr{T}, \|A_0\|_{H^s}, 1 - \lambda(t_j)$.  Notice also that for any $j$, if $\mathbf{t}_{j + 1} \leq \mathscr{T}\epsilon^{-2}$, then the hypotheses of Proposition \ref{RemainderEstimates} hold uniformly in $j$: therefore after each local construction, one can apply Proposition \ref{RemainderEstimates} on $[0, \mathbf{t}_{j + 1}]$ with the same choice of $\epsilon_0$ for each $j$ to conclude that the hypotheses (i)-(iii) on the initial data in Lemma \ref{NanoLocalWellposedness} continue to hold whenever $\mathbf{t}_{j + 1} \leq \mathscr{T}\epsilon^{-2}$.

Because of the dependence of the existence times on $\lambda(\mathbf{t}_j)$, it is at this point still possible that the lengths of the intervals may decrease in such a way that the total existence time converges.  We claim that this cannot happen; since the only dependence of the existence times that is not uniformly controlled in $j$ is through $\lambda(\mathbf{t}_j)$, it suffices to show that $1 - \lambda(\mathbf{t}_j)$ remains bounded below for all $j$

By Proposition \ref{zLocalWellPosed}(a) we know that $\lambda(0) \leq \frac12 + \frac{1}{2\nu}$, so that $\epsilon^{q + 4}g^\prime(\lambda(0)) < \frac{8}{\mathscr{C}}$.  We begin a contradiction argument.  Suppose that in the course of this construction there is a first interval $(\mathbf{t}_{j}, \mathbf{t}_{j + 1}]$ for which $\max_{t \in (\mathbf{t}_j, \mathbf{t}_{j + 1}]} \epsilon^{q + 4} g^\prime(\lambda(t^*)) > \frac{8}{\mathscr{C}}$.  (This notion is well-defined since $\lambda(t)$ is defined pointwise everywhere, c.f. Remark \ref{LambdaSomewhatControlled}).  On this interval, we have by Proposition \ref{zLocalWellPosed}(b) that for all $t \in (\mathbf{t}_j, \mathbf{t}_{j + 1}]$, 
\begin{equation}\label{DiscreteContinuityBound}
\frac{4}{\mathscr{C}} \leq \epsilon^{q + 4}g^\prime(\lambda(t)) \leq \frac{16}{\mathscr{C}}
\end{equation}  But then by following the calculations preceding \eqref{ODEBoundOnLambda} on this interval, we find that $\lambda^\prime(t)$ is defined almost everywhere on $[\mathbf{t}_j, \mathbf{t}_{j + 1}]$.  Specifically, if we choose $\epsilon_0 > 0$ sufficiently small, inserting \eqref{DiscreteContinuityBound} into \eqref{ODEBoundOnLambda} implies that
\begin{equation}
\lambda^\prime(t) \leq  -\frac{\mathscr{C} - 1}{504T_2}C_\omega \epsilon^2 \qquad \text{a.e. } t \in (\mathbf{t}_j, \mathbf{t}_{j + 1}],
\end{equation}
But then $\lambda(t)$ is absolutely continuous and strictly decreasing on $[\mathbf{t}_j, \mathbf{t}_{j + 1}]$, and thus $\lambda(t_j) > \lambda(t^*)$.  But then this contradicts the assumption that $(\mathbf{t}_j, \mathbf{t}_{j + 1}]$ is the first interval on which $\lambda(t)$ exceeds the value $8/\mathscr{C}$.  

This establishes the a priori bound $\lambda(t) \leq 1 - \left(\frac{\mathscr{C}}{16}\epsilon^{q + 4}\right)^2$.  Hence we can apply Proposition \ref{zLocalWellPosed}(b) a finite number of times depending on $s, p, q, \|A_0\|_{H^s}, \mathscr{T}, \epsilon$ to construct a solution defined on all of $[0, \mathscr{T}\epsilon^{-2}]$.
\end{proof}

\begin{remark}
We emphasize that although the number of constructions needed in the above proof depends on $\epsilon$, the number of smallness restrictions on $\epsilon_0 > 0$ is independent of the number of constructions.  Hence there is no danger that we are forced to choose a sequence of restrictions that forces $\epsilon_0$ to be zero.
\end{remark}

%==============================================================================
%==============================================================================
%==============================================================================
\section{Conclusion}
%==============================================================================
%==============================================================================
%==============================================================================

We have shown that the NLS equations \eqref{NormalizedNLSPrescale} are globally well-posed and, moreover, enjoy the property that all of their subcritical inhomogeneous Sobolev norms grow at a polynomial rate.  Since solutions to \eqref{NormalizedNLSPrescale} with large initial data avoid the characteristic finite time blow-up of the focusing type equation, we subsume both defocusing and hyperbolic NLS equations under the term \textit{non-focusing} NLS equations.  We conclude by further discussing the method used in this paper, comparing with known results, and considering possible generalizations.

The original goal of this paper was to prove global well-posedness for $(qHNLS)$.  However, this result improves the regularities at which large data solutions to $(qNLS)$ are globally well-posed; for all $q \in 2\mathbb{N} + 1$, we conclude global well-posedness for large data in any subcritical Sobolev spaces.  The best known results in this direction when the critical index $s_c$ does not correspond to a conserved quantity are a consequence of either the I-method or the Fourier truncation method of Bourgain\footnote{Also called the ``high-low'' method.}, both of which do not allow one to conclude global well-posedness in the entire subcritical range (c.f. \cite{BourgainHighLowSurvey} for a survey of the Fourier truncation method, \cite{CKSTT2003b} for an application of the I-method to $(qNLS)$, and \cite{TaoNonlinearDispersiveEquations} for further discussion).  

Even so, there is room to improve the results presented here.  Because of the strong dependence on subcritical persistence of regularity, we cannot conclude anything about the global well-posedness in critical Sobolev spaces.  Similarly, the bounds of Theorem \ref{BigSummaryTheorem} cannot be used directly to conclude whether the equations considered here scatter (even in light of Remark \ref{BringDownHugeConstant}).

It is natural to ask whether a fully nonlinear progenitor equation is needed to carry out the program of this paper, since the full nonlinearity creates significant technical complications.  Of course we will not attempt to conclusively demonstrate here that a progenitor with a full nonlinearity is required in order to carry out the strategy of this paper; a simpler or less cumbersome progenitor may very well exist, especially in light of the fact that the NLS equations under consideration here arise generically in the modulation regime of a large number of dispersive PDEs.  However, if one designs a progenitor with the same linearization and power nonlinearity as in \eqref{zEquationOutline}, some sort of penalization term as is defined here must be introduced, since nothing about such a progenitor would control $\Lambda^s_\epsilon z$ or its derivatives (in this context see Remark \ref{HamIsQuantUseless}).  If one introduces a penalization term depending only on $z$ and $z_t$, then there are two possibilities.  (1) The introduced term contributes terms on the order $O(\epsilon^3)$ or larger on the NLS time scales, which forces one to either change the NLS equation in question or discard the error estimates, both of which spoil the full justification.  (2) The penalization term is beyond the order in $\epsilon$ of the NLS equation: in this case the error estimates for $z$, $z_t$, $z_{tt}$ remain valid, but the penalization term does not affect the dynamics of the progenitor over the NLS time scales, producing essentially the same effect as having no penalization term whatsoever.  Introducing a full nonlinearity allows one to simultaneously keep $z$, $z_t$, $z_{tt}$ in the modulation regime as well as being able to control the evolution of $z_{tt}$ through an ODE (see Remark \ref{HighDerivsNotWavePackets}).

In a forthcoming work, we plan to generalize the method used in this paper to NLS equations of arbitrary non-focusing type\footnote{As in this paper, this is a shorthand term for NLS equation which are either elliptic with the defocusing choice of sign on the power nonlinearity, or else of arbitrary mixed signature with either choice of sign on the power nonlinearity.} in higher dimensions $\mathbb{R}^d$ for $d \geq 3$.  This setting is especially interesting because it would overcome a long-standing obstacle to showing global well-posedness for NLS equations for which the critical index $s_c > 1$.  Current methods for showing global well-posedness for such equations fail because all of the natural conserved quantities scale supercritically with respect to the natural scaling of the equation.  This makes them unsuitable for use in supplying useful global a priori bounds.  This is the same obstacle one faces in showing long-time existence for NLS equations of hyperbolic type, except that the Hamiltonian is useless for reasons of scaling instead of lack of coercivity.  A priori bounds analogous to those presented in this paper would provide a partial substitute for the missing Hamiltonian and show global well-posedness in subcritical Sobolev spaces.  This would be the first method that provides global well-posedness results for this class of equations.

\bibliography{Mybib}{}
\bibliographystyle{plain}

\end{document}